\documentclass[final,3p,times]{elsarticle}

\usepackage{amssymb}
\usepackage{lipsum}
\usepackage{amsfonts}
\usepackage{graphicx}
\usepackage{epstopdf}
\usepackage{algorithmic}
\usepackage{natbib}
\usepackage{graphicx}
\usepackage{float}
\usepackage{subfig}
\usepackage{amsopn}
\usepackage{amsmath}
\usepackage{hyperref}
\usepackage{cleveref}
\usepackage{color,xcolor}

\newcommand{\tpr}{\boldsymbol{\otimes}}   % tensor product
\newcommand{\dpr}{\boldsymbol{\cdot}}     % dot product
\newcommand{\vc}[1]{\boldsymbol{#1}}      % bold for vector values

\newcommand{\dudx}[2]{\frac{\partial{#1}}{\partial{#2}}}

\newtheorem{remark}{Remark}

\newtheorem{proposition}{Proposition}
\newtheorem{proof}{Proof}

\begin{document}

\begin{frontmatter}



\title{Mathematical modelling of transport phenomena in compressible multicomponent flows}


\author[label1]{Chao Zhang}
\author[label1,label5]{Lifeng Wang}
\author[label1,label5]{Wenhua Ye}
\author[label1]{Junfeng Wu}
\author[label1,label5]{Zhijun Shen}
\author[label3,label4]{Igor Menshov}

\address[label1]{Institute of Applied Physics and Computational Mathematics, Beijing, China}
%\address[label2]{Lomonosov Moscow State University}
\address[label3]{Keldysh Institute for Applied Mathematics RAS, Moscow, Russia}
\address[label4]{SRISA RAS, Moscow, Russia}
\address[label5]{Center for Applied Physics and Technology, HEDPS, Peking University, Beijing, China}

\begin{abstract}
The present article proposes a diffuse interface model for compressible multicomponent flows with transport phenomena of mass, momentum and energy (i.e., mass diffusion, viscous dissipation and heat conduction). The model is reduced from the seven-equation Baer-Nuziato type model with asymptotic analysis in the limit of instantaneous mechanical relaxations. The main difference between the present model and the Kapila's five-equation model consists in that different time scales for pressure and velocity relaxations are assumed, the former being much smaller than the latter. Thanks to this assumption, the velocity disequilibrium is retained to model the mass diffusion process. Aided by the diffusion laws, the final model still formally consists of five equations. The proposed model satisfy two desirable properties : (1) it respects the laws of thermodynamics, (2) it is free of the spurious oscillation problem in the vicinity of the diffused interface zone.  For solution of the model governing equations, we implement the fractional step method to split the model into five physical steps: the hydrodynamic step, the viscous step, the heat transfer step, the heat conduction step and the mass diffusion step. The split governing equations for the hydrodynamic step formally coincide with the Kapila's five equation model and are solved with the Godunov finite volume method. The mass diffusion, viscous dissipation and heat conduction processes contribute parabolic partial differential equations that are solved with the Chebyshev method of local iterations. Numerical results show that the proposed model maintains pressure, velocity and temperature equilibrium near the diffused interface. Convergence tests demonstrate that the numerical methods achieve second order in space and time. The proposed model and numerical methods are applied to simulate the laser-driven RM instability problem in inertial confinement fusion, good agreement with experimental results are observed.
\end{abstract}




\end{frontmatter}


\section{Introduction}
\label{sec:intro}

Compressible multicomponent flows are of significance to many applications, such as the inertial confinement fusion (ICF),  the explosion of core-collapse supernova, underwater explosion (UNDEX) and so forth. These physical processes include Rayleigh-Taylor (RT) and Richtmyer-Meshkov (RM) hydrodynamic instabilities that rapidly develop in the presence of small initial perturbations. Up to now, understanding the development of these nonlinear instabilities  still heavily relies on numerical simulations. The present research is motivated by the need to simulate the  laser-driven plasma instability developed at the interface between dissimilar materials within ICF capsules. In such applications, the transport phenomena (of mass, momentum and energy) accompanying the hydrodynamic process play significant role. First of all, the laser energy deposited in the plasma is transported through the heat conduction process. Moreover, at small spatial scales the effect of viscous dissipation and the mass diffusion begin to impact the instability growth \cite{robey2004effects}. Numerically, the transport process is vital for achieving a grid-converged DNS (Direct Numerical
Simulation) \cite{vold2021plasma}.   

To correctly simulate the dissipative processes in multicomponent flows, the governing model  should satisfy two important criteria, i.e.,  physically admissibility  and numerical consistency. The former stipulates that the model should respect the first and second law of thermodynamics. The latter dictates that the closure relations and numerics do not cause non-physical spurious oscillations near the material interface. Several works in the literature have attempted in this direction \cite{Thornber2018,Cook2009Enthalpy}. %However, there is no proof that the above two criteria are maintained. 
The present work is performed in the framework of the diffuse interface model (DIM). We aim to incorporate various dissipative transport phenomena (including mass diffusion, viscous dissipation and heat conduction) into this framework with the above two criteria being maintained. 

The fully conservative four-equation DIM (i.e., Euler equation supplemented with one conservation equation for the partial density) is notorious for triggering spurious oscillations in pressure and velocity (at the hydrodynamic stage) when solved with the Godunov finite volume method (FVM). Analysis works on this phenomenon have been performed both physically and numerically \cite{abgrall1996prevent,abgrall2001computations,saurel1999multiphase}.    In this model there exists only one temperature and one pressure, implicitly assuming that the components  are in thermal and mechanical equilibrium. {\color{black}This assumption is commonly used for combustion and boiling problems, however, it maybe too strong to be valid for interface problems.}  Pressure and temperature disequilibria are diminished by compaction (pressure relaxation) and heat transfer (temperature relaxation) between components, respectively. However, these two processes may take place at very different time scales. For example, for the deflagration-to-detonation transition (DDT) in granular materials, the characteristic time scales for pressure relaxation and temperature relaxation are 0.03$\mu s$ and 18000$\mu s$ after sufficient combustion \cite{kapila2001two}, respectively. It is evident that forcing thermodynamical equilibrium to be reached at the same time scale may lead to physical inconsistency. 

To understand the temperature closure relations within the computational cell, let us look at the  resolved interface problem in in-miscible multicomponent flows (\Cref{fig:Tjump}). According to the Rankine-Hugoniot relation, the jump conditions across the interface (the line $C$ in \Cref{fig:Tjump}) along its normal direction read:
\begin{subequations}
\begin{align}
    \llbracket \rho Y_1 (u_n-\mathcal{U})\rrbracket = 0, \\
\llbracket\rho Y_2 (u_n-\mathcal{U})\rrbracket=0,\\
\llbracket p + \rho \vc{u}(u_n-\mathcal{U}) \rrbracket =0, \label{eq:mom_jump} \\
\llbracket (u_n-\mathcal{U})\left(\rho e+\frac{\rho |\vc{u}-\mathcal{U}\vc{n}|^{2}}{2}  + p\right) + J_{qn} \rrbracket =0, \label{eq:en_jump}
\end{align}
\end{subequations}
where $\rho$, $Y_k$, $e$ are the mixture density, the mass fraction of component $k$, and the mixture internal energy. $u_n$, $\mathcal{U}$,  $J_{qn}$ are normal components to the interface $C$ of the particle velocity $\vc{u}$, the interface velocity, and the heat flux, respectively.  The operator $\llbracket \phi \rrbracket = \phi_2 - \phi_1$ with the subscripts $1/2$ denoting the variables on the left/right of the discontinuity, or the states of the 1/2-fluid.

\begin{figure}[htbp]
\centering
\subfloat[Equilibrium temperatures]{\label{Tjump:eq}\includegraphics[width=0.25\textwidth]{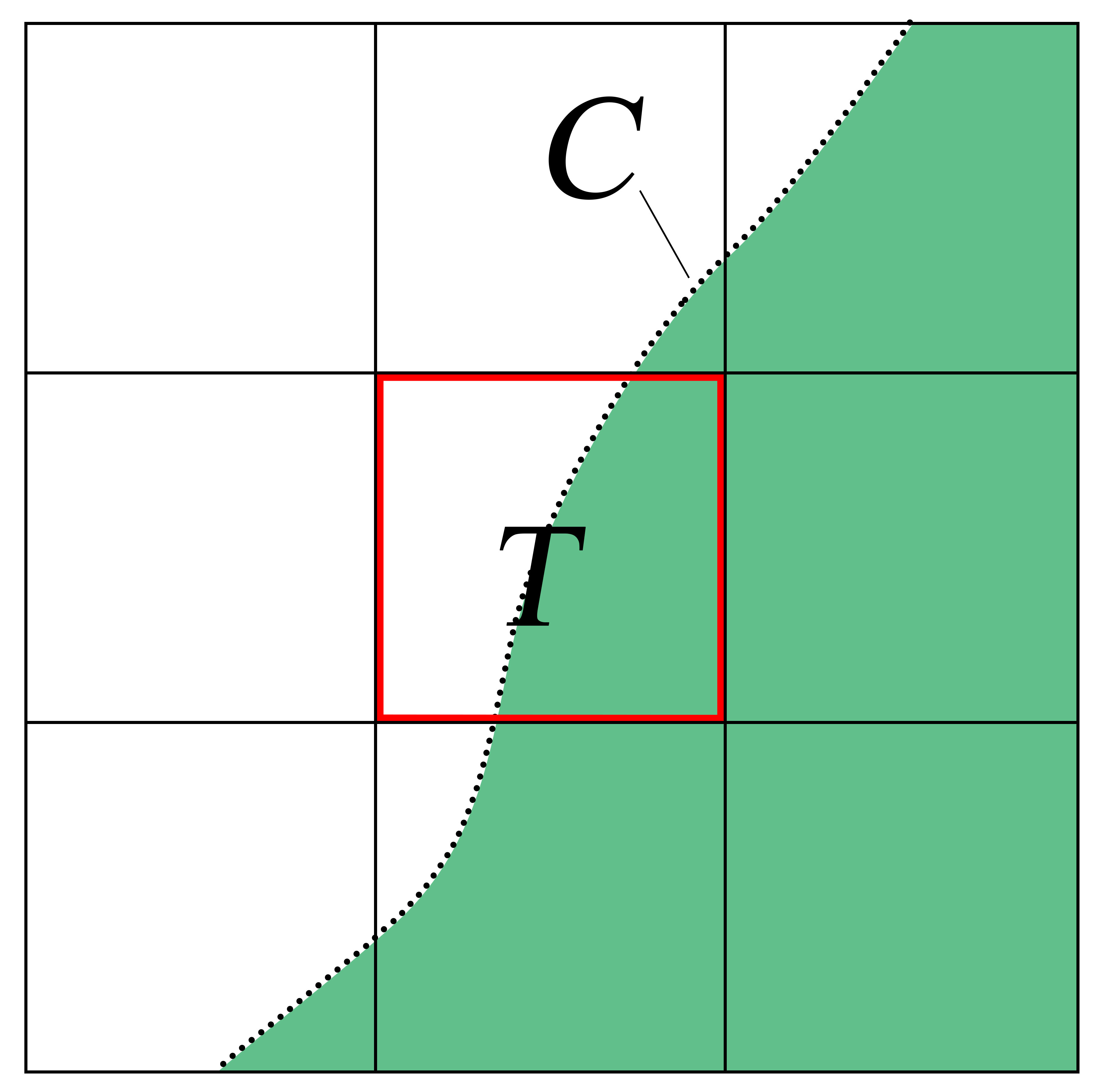}} \quad\quad
\subfloat[Disequilibrium temperatures]{\label{Tjump:ineq}\includegraphics[width=0.25\textwidth]{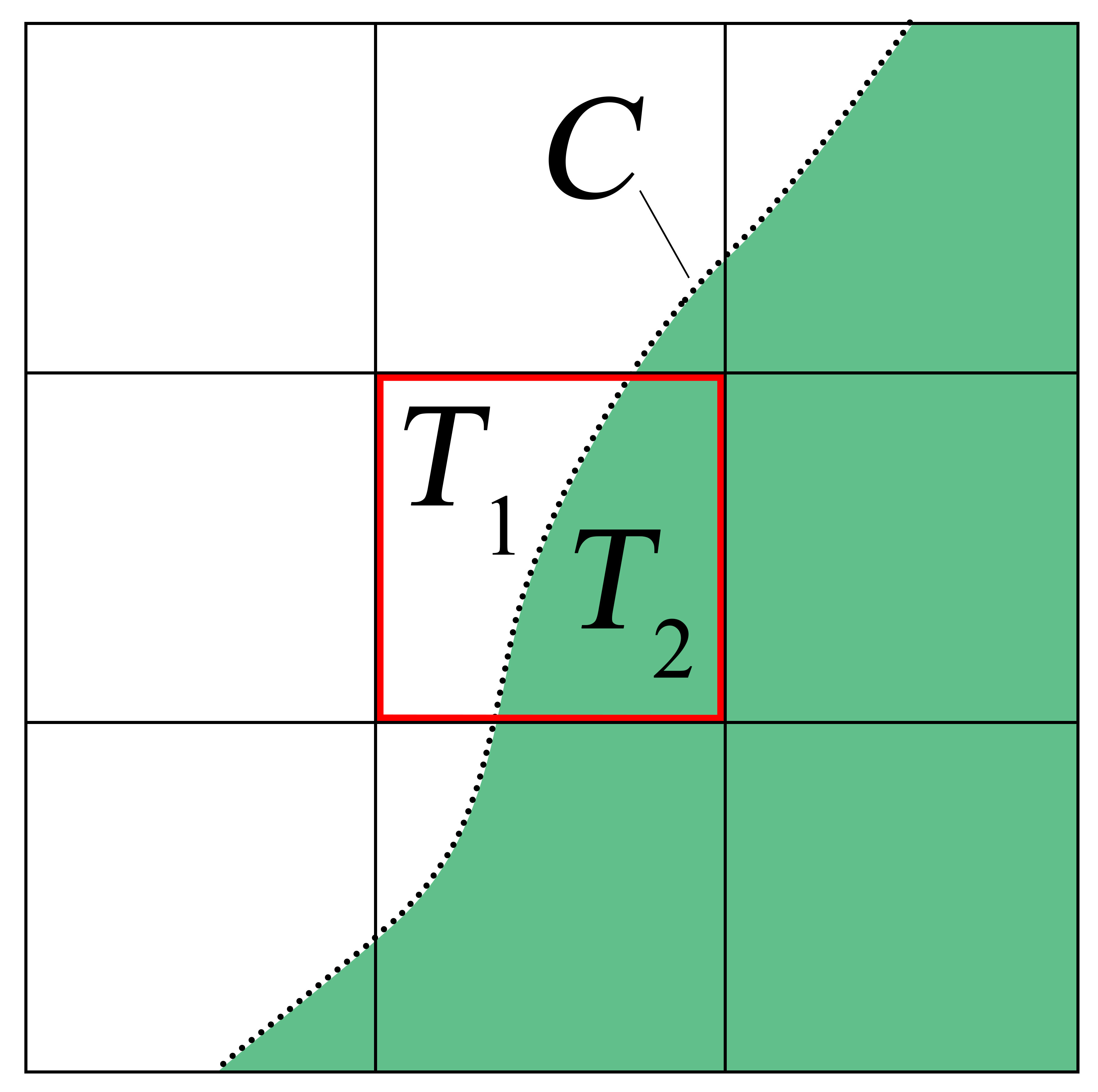}}\quad\quad
\subfloat[Miscible interface]{\label{fig:mass_diff_interface}\includegraphics[width=0.25\textwidth]{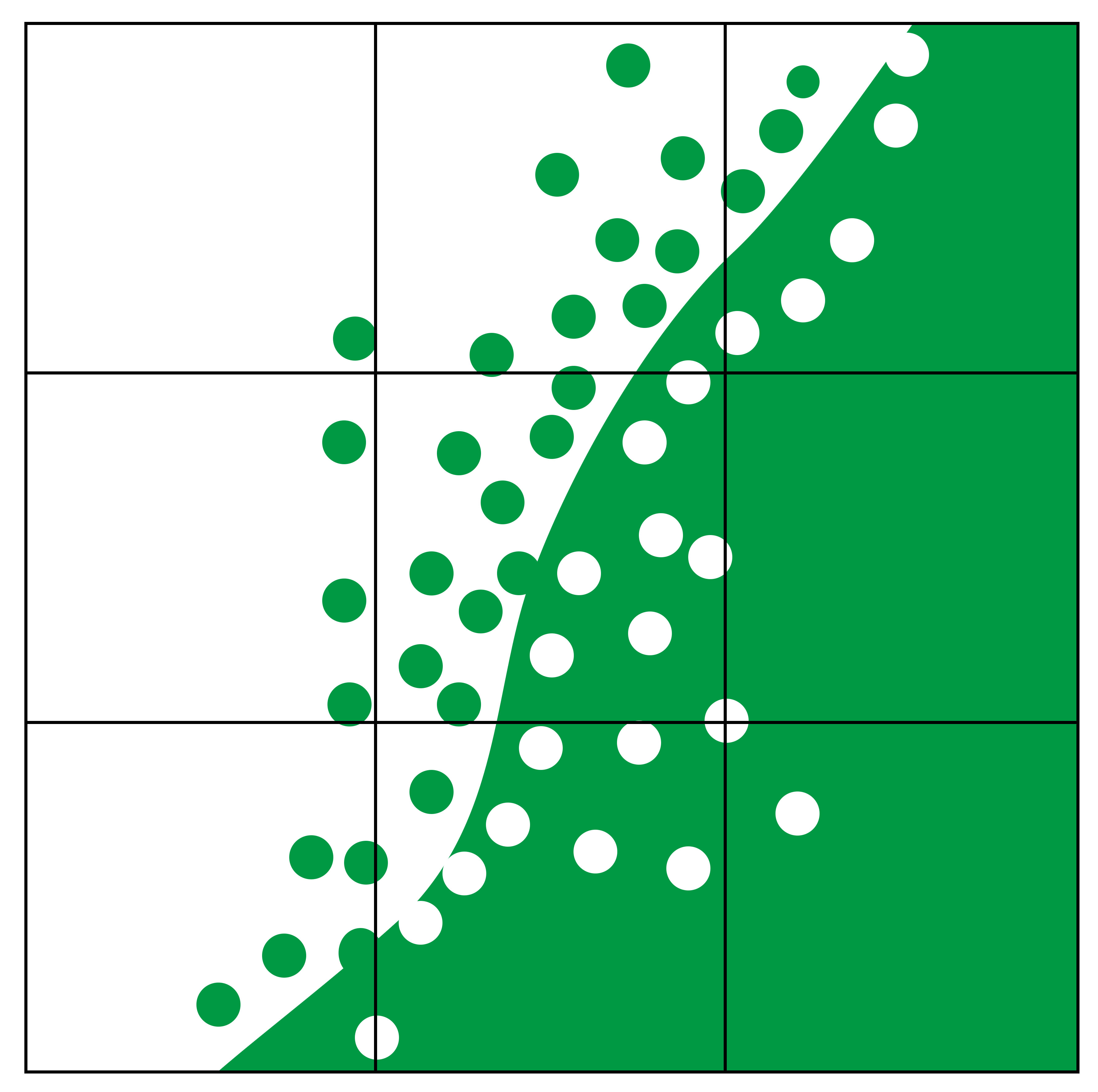}}
\caption{Interfaces on computational grid.}
\label{fig:Tjump} 
\end{figure}

The interface represents a contact discontinuity, across which velocity is continuous, i.e., $  u_{n1} = u_{n2} = \mathcal{U}$. From \cref{eq:mom_jump,eq:en_jump} follows $\llbracket p \rrbracket = 0$ and $\llbracket J_{qn} \rrbracket = 0$, respectively. With the Fourier's heat flux, the latter can be expressed as:
\begin{equation}\label{eq:heat_flux_jump}
  \left( \lambda_1 \nabla{T_1} \right) \dpr \vc{n} = \left( \lambda_2 \nabla{T_2} \right) \dpr \vc{n},
\end{equation}
where $\lambda_k$ is the heat conduction coefficient of component $k$. 

In the framework of the DIM, the material interface contained within a computational cell is not tracked and the properties of fluids are diffused within the interfacial zone. Thermal conductivity imposes temperature continuity across the interface. If the model is equipped with only one temperature, then the two materials inside a computational cell share the same temperature (i.e., $T_1 = T_2 = T$, \Cref{Tjump:eq}), which maybe inconsistent with \cref{eq:heat_flux_jump}. This inconsistency results in numerical errors of FVM or physical inconsistency in the vicinity of the interface. In fact, the equilibrium temperature assumption means that temperature relaxation rate is infinitely large so that phase temperature equilibrium is reached instantaneously. This assumption deprives the model of the ability to deal with physically finite relaxation rates.

To get rid of the above-mentioned temperature inconsistency, we turn to the temperature-disequilibrium models. In such models each cell state is characterized with temperatures for each component (\Cref{Tjump:ineq}), and thus the temperature equilibrium is  not enforced. 
One representative of such models is the Baer-Nuziato (BN) model \cite{BAER1986861} and its variant \cite{saurel1999multiphase} for compressible multiphase flows. Formally, the BN model include balance equations for the partial density, the phase momentum and the phase total energy and is argumented with an evolution equation for volume fraction. The latter is vital for maintaining the free-of-oscillation property at the interface. In this model, each component is described with a full set of parameters (density, velocity, pressure and temperature) and governed by the single-phase Navier-Stokes (NS) equations away from the interface. The interactions between components only happen in the neighbourhood of the interface. These interactions include the kinetic, mechanical and thermal relaxations that strive to erase the disequilibria in velocity, pressure and temperature. The temperature equilibrium (imposed in one-temperature DIMs) is reached only after the complete temperature relaxation. The characteristic relaxation rates depend on particular physical problems.  The model is unconditionally hyperbolic and respect the laws of thermodynamics. Moreover, in DIM the phase temperatures exist all through the computational domain since even the pure fluid is approximated as a fluid with negligible amount of other components. The disequilibrium temperature closure in each computational cell does not introduce inconsistency with the heat flux jump conditions.

Although the BN model is physically complete, it consists of complicated wave structure and stiff relaxation processes, making its numerical implementation quite cumbersome. For many application scenarios, the model can be simplified to a large extent. For example, for the DDT process, Kapila et al. have derived two reduced models in the limit of instantaneous mechanical relaxations \cite{kapila2001two}. The first model is obtained in the limit of instantaneous velocity relaxation and consists of six equations. Only one equilibrium velocity exists in this model. This formulation is used for solving Kapila's five-equation model in \cite{saurel2009simple} for improving robustness. The second is the well-known five-equation model and derived in the limit that both pressure and velocity relaxation rates approach infinity (\Cref{fig:model_schematic}).

\begin{figure}[htbp]
\centering
\includegraphics[width=0.8\textwidth]{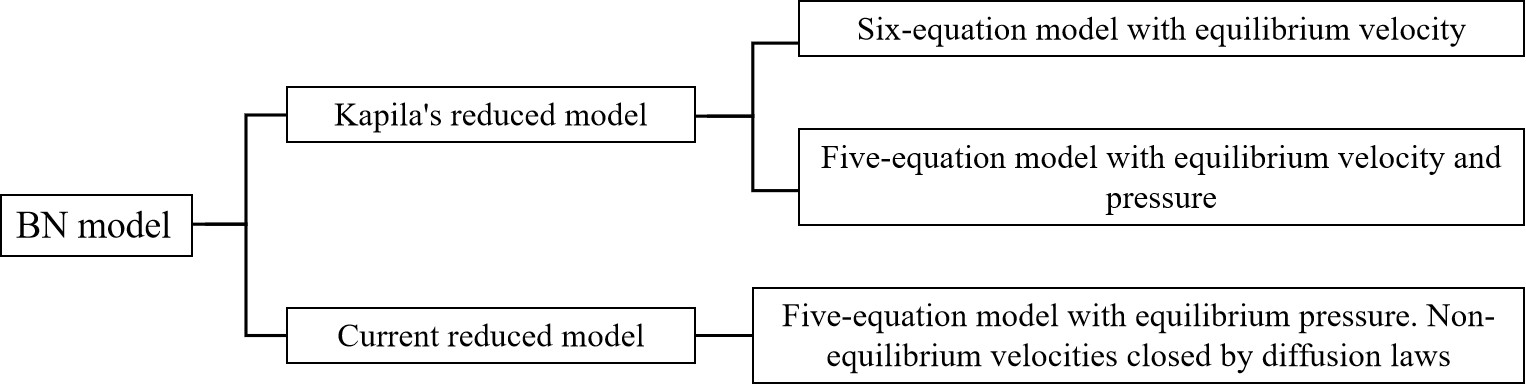}
\caption{Schematic of reduced models.}
\label{fig:model_schematic} 
\end{figure}

Most current DIM works in literature focus on the resolved interface problem, however,  the components maybe miscible and penetrate into each other (\Cref{fig:mass_diff_interface}). Under such circumstances, there is no longer a definite sharp discontinuous material interface, but a a physically diffused zone with finite thickness. This diffused zone develops as a result of mass diffusion and enthalpy diffusion. The significance of the latter to maintain the thermodynamical consistency has been demonstrated in \cite{Cook2009Enthalpy}. In fact, the enthalpy diffusion appear as a result of replacing the phase velocities with the mass-weighted one for any multi-velocity model, including the BN model.

From a microscopic point of view, mass diffusion is the result of molecular random motions, which drives the molecular distribution toward uniformity. In the macroscopic continuum mechanics where each spatial element contains enormous number of molecules, the mass diffusion flux is characterized by the velocity difference between different species. Thus, the one-velocity models of Kapila is incapable of modelling this process. When the mass diffusion effect is strong, the velocity relaxation time scale is comparable to that of the problem considered and the assumption of instantaneous velocity relaxation is inappropriate. Therefore, to model the mass diffusion, we have to retain the velocity disequilibrium and then  close it with diffusion laws (\Cref{fig:model_schematic}), which is one of the major contributions of the present paper.

The reduction procedures in \cite{murrone2005five,perigaud2005compressible} assume the same relaxation time scale for velocity and pressure. This approach inevitably leads to velocity-equilibrium models where no mass diffusion sustains.  Instead, weighing  simplicity and thermodynamic consistency, we assume  different time scales for the velocity and pressure. The pressure relaxation time scale is much smaller than that of the velocity. Such an assumption is in fact supported many physics, for example, see \cite{kapila2001two,bilicki1996evaluation,guillard2006numerical,zein2010}. With this assumption and the diffusion laws, the seven-equation BN model can be reduced to a five-equation one without losing the ability to model the mass diffusion process. 

In reducing the velocity-disequilibrium model, the second order term of the velocity relaxation time is abandoned. The finally obtained model formally contains one velocity. However, this velocity is the mass-weighted mixture velocity rather than the equilibrium velocity in Kapila's model.  The component velocities can be derived by invoking diffusion laws. The finally obtained model satisfy both of the two criteria (physical admissibility and numerical consistency) proposed above.

%The obtained model is consistent with the one temperature model of \cite{Cook2009Enthalpy}. Moreover, these diffusions are solved under the pressure and temperature equilibrium, thus, all the thermodynamical considerations also apply to our model.

We propose numerical methods for solving the model with the fractional step method. In numerical implementation, the model is split into five physical parts, i.e., the hydrodynamic part, the viscous part, the temperature relaxation part, the heat conduction part and the mass diffusion part. The first part formally coincides with the original Kapila's formulation, however, the velocity owns different physical meanings (as discussed in the last paragraph). Various Godunov FVM methods in literature (for example, see \cite{Coralic2014Finite,perigaud2005compressible,kreeft2010new,Zhangchao2020}) can be used to solve the governing equations for the hydrodynamic part. In the first two parts the component temperatures are in disequilibrium.  The temperature relaxation takes place at a much larger time scale and solved separately after these parts. The diffusion processes (including viscous dissipation, heat conduction and mass diffusion processes) are governed by parabolic PDEs, that are solved with the locally iterative method based on the Chebyshev parameters \cite{Zhukov2010,zhukov2018}.
The heat conduction and mass diffusion equations are solved maintaining the temperature equilibrium, i.e., assuming an instantaneous temperature relaxation.  Note that finite temperature relaxation can also be considered straightforwardly.

The rest of the present paper is organized as follows. In \Cref{sec:model} we derive the reduced model by performing the asymptotic analysis on the BN-type model in the limit of instantaneous mechanical relaxations. In \Cref{sec:numer_meth} we develop numerical methods for solving the proposed model. In \Cref{sec:numer_res} we present numerical  results for several multicomponent problems with diffusions and apply the model and numerical methods to the laser ablation  problem in the field of ICF.

\section{Model formulation}
\label{sec:model}
\subsection{The BN-type seven-equation model}
The starting point of the following model formulation is the complete  BN-type seven-equation model \cite{BAER1986861,saurel1999multiphase,petitpas2014,perigaud2005compressible}, which can be derived by using the averaging procedure of \cite{drew1983mathematical}. It reads:
\begin{subequations} \label{eq:bn}
\begin{align}
\label{eq:bn:mass}
\dudx{\alpha_k\rho_k}{t} + \nabla\dpr(\alpha_k\rho_k\vc{u}_k) = 0, \\
  \label{eq:bn:mom}
  \dudx{\alpha_k\rho_k\vc{u}_k}{t} + \nabla\dpr\left(\alpha_k\rho_k\vc{u}_k\tpr\vc{u}_k - \alpha_k \overline{\overline{T}}_k \right) =  -  \overline{\overline{T}}_I \dpr \nabla {\alpha_k}  + \mathcal{M}_k,\\
  \label{eq:bn:en}
  \dudx{\alpha_k \rho_k E_k}{t} + 
  \nabla\dpr\left(
\alpha_k \rho_k E_k \vc{u}_k   - \alpha_k \overline{\overline{T}}_k \dpr \vc{u}_k
  \right) 
  = 
  - \vc{u}_{I} \dpr \left( \overline{\overline{T}}_I \dpr \nabla {\alpha_k} \right) 
    + \vc{u}_I \mathcal{M}_k 
- p_I \mathcal{F}_k + \mathcal{Q}_k + q_k + \mathcal{I}_k,\\
   \label{eq:bn:vol}
  \dudx{\alpha_k }{t} + 
  \vc{u}_I \dpr \nabla\alpha_k = \mathcal{F}_k,
\end{align}
\end{subequations}
where the notations used are standard: $\alpha_k, \; \rho_k, \; \vc{u}_k, \; p_k, \; \overline{\overline{T}}_k, \; E_k$ are the volume fraction, phase density, velocity, pressure, stress tensor, and total energy of phase $k$. For the sake of clarity we constrict our discussions within the scope of two-phase flows, $k=1,2$. The phase density $\rho_k$ is defined as the mass per unit volume occupied by $k$-th phase. The mixture density $\rho$ is the sum of the partial densities $\alpha_k \rho_k$, i.e., $\rho = \sum {\alpha_k \rho_k}$. 
The last equation \cref{eq:bn:vol} is written for only one component thanks to the saturation constraint for volume fractions 
$\sum_{k=1}^2 \alpha_k  = 1$.
%the last equations are written only for only $N-1$ volume fractions, where $N$ is the number of components.  
The total energy is $E_k  = e_k + \mathcal{K}_k$ where $e_k$, and $\mathcal{K}_k = \frac{1}{2}\vc{u}_k \dpr \vc{u}_k$ are the  internal energy and kinetic energy, respectively. 
%$\mathcal{I}_k = \alpha_k I_k$ with $I_k$ being the external energy source released in the phase $k$.

The variables with the subscript ``I'' represent the variables at interfaces, for which there are several possible definitions \cite{saurel1999multiphase,perigaud2005compressible,saurel2018diffuse}. 
%Whatever the definitions we choose, $\lim_{\eta\to\infty} p_{I} = \lim_{\eta\to\infty}p_{k} = p$, $\lim_{\vartheta\to\infty}\vc{u}_{I} = \lim_{\vartheta\to\infty}\vc{u}_{k} = \vc{u}$, and $\lim_{\vartheta\to\infty} {\overline{\overline{\tau}}}_{I} = \lim_{\vartheta\to\infty}{\overline{\overline{\tau}}}_{k} = {\overline{\overline{\tau}}}$.
Here we choose the following
\begin{equation}
\vc{u}_I = \overline{\vc{u}} = {\sum y_k \vc{u}_k }, \quad p_I = \sum \alpha_k p_k, \quad \overline{\overline{T}}_{I} = - p_I \overline{\overline{I}} +  \overline{\overline{\tau}}_I.
\end{equation}
where $y_k$ denotes the mass fraction $y_k  = \alpha_k \rho_k / \rho$, and $\overline{\vc{u}}$ is the mass-fraction weighted mean velocity. The interfacial stress $\overline{\overline{\tau}}_I$ is defined in such way that the thermodynamical laws are respected. 

The inter-phase exchange terms include the velocity relaxation $\mathcal{M}_k$, the pressure relaxation $\mathcal{F}_k$, and the temperature relaxation $\mathcal{Q}_k$. They are as follows:
\begin{equation}\label{eq:relaxations}
\begin{split}
\mathcal{M}_k=\vartheta\left( \vc{u}_{k^*}-\vc{u}_k \right), \quad \mathcal{F}_k=\eta\left( {p}_{k} - {p}_{k^*} \right), \quad \mathcal{Q}_k=\varsigma\left( T_{k^*} - T_{k} \right).
\end{split}
\end{equation}
where $k^{*}$ denotes the conjugate component of the $k$-th component,  i.e., $k=1,\; k^{*} =2$ or $k=2,\; k^{*} =1$. The relaxation velocities are all positive $\vartheta > 0, \; \eta > 0, \; \varsigma > 0$.

The phase stress tensor, $\overline{\overline{T}}_k$, can be written as
\begin{equation}
 \overline{\overline{T}}_k  = - p_k \overline{\overline{I}} +  \overline{\overline{\tau}}_k.
\end{equation}

For the viscous part we use the  Newtonian approximation
\begin{equation}\label{eq:newton_vis}
\overline{\overline{\tau}}_k = 2\mu_k \overline{\overline{D}}_k + \left(\mu_{b,k} - \frac{2}{3}\mu_k \right) \nabla \dpr \vc{u}_k,
\end{equation}
where $\mu_k > 0$ is the coefficient of shear viscosity and  $\mu_{b,k} > 0$ is the coefficient of bulk viscosity. 

The deformation rate $\overline{\overline{D}}_k$ is
\[
\overline{\overline{D}}_k = \frac{1}{2} \left[ \nabla \vc{u}_k + \left(  \nabla \vc{u}_k \right)^{\text{T}} \right] = \overline{\overline{D}}_a + \overline{\overline{D}}_{wk},
\]
where the average part 
\[
\overline{\overline{D}}_a = \frac{1}{2} \left[ \nabla \overline{\vc{u}} + \left(  \nabla \overline{\vc{u}} \right)^{\text{T}} \right].
\]
the diffusion part
\[
\overline{\overline{D}}_{wk} = \frac{1}{2} \left[ \nabla \vc{w}_k + \left(  \nabla \vc{w}_k \right)^{\text{T}} \right],
\]
where the diffusion velocity $\vc{w}_k$ is defined as
\begin{equation}
\vc{w}_k = \vc{u}_k - \overline{\vc{u}} .
\end{equation}

With the definition of $\overline{\overline{D}}_{a}$ and $\overline{\overline{D}}_{wk}$, we can further split $\overline{\overline{\tau}}_k$ into the average and diffusion parts:
\begin{subequations}\label{eq:stress_disp}
\begin{align}
\overline{\overline{\tau}}_{ak} = 2\mu_k \overline{\overline{D}}_a + \left(\mu_{b,k} - \frac{2}{3}\mu_k \right) \nabla \dpr \overline{\vc{u}},\\
\overline{\overline{\tau}}_{wk} = 2\mu_k \overline{\overline{D}}_{wk} + \left(\mu_{b,k} - \frac{2}{3}\mu_k \right) \nabla \dpr {\vc{w}}_k.
\end{align}
\end{subequations}

The heat conduction term is
\begin{equation}\label{eq:q}
q_k = - \nabla \dpr \vc{J}_{qk},
\end{equation}
where the Fourier heat flux is
\begin{equation}\label{eq:fourier_flux}
\vc{J}_{qk} = - \alpha_k \lambda_k \nabla T_k
\end{equation}

By performing the averaging procedure of Drew et al.\cite{drew1983mathematical}, one can derive the external energy source
\begin{equation}\label{eq:I}
\mathcal{I}_k = \alpha_k I_k,
\end{equation}
where $I_k$ denotes the the intensity of the external heat source released in the $k$-th phase, $I_k (\vc{x},t) \geq 0$.

Without the diffusion and relaxation processes, the seven-equation model is unconditionally hyperbolic with the following set of wave speeds $u_k \pm a_k, u_k, u_I$, where $a_k$ is the sound speed 
\begin{equation}
a_k ^2 = \left( \dudx{p_k}{\rho_k} \right)_{s_k} = \frac{\frac{p_k}{\rho_k^2} - \left( \dudx{e_k}{\rho_k} \right)_{p_k} }{\left( \dudx{e_k}{p_k} \right)_{\rho_k}} > 0.
\end{equation}

\begin{remark}
For the sake of objectivity, the constitutive relation for the interfacial stress $\overline{\overline{\tau}}_I$ may depend on the following list of frame-invariant variables
\[ \alpha, \;\; \text{D}_{k} \alpha / \text{D} t, \;\; \nabla \alpha, \;\; \mathbf{u}_{1}-\mathbf{u}_{2}, \;\; \text{D}_{2} \mathbf{u}_{1} / \text{D} t-\text{D}_{1} \mathbf{u}_{2} / \text{D} t, \;\; \overline{\overline{D}}_{1}, \; \overline{\overline{D}}_{2}, \;\; \nabla\left(\mathbf{u}_{1}-\mathbf{u}_{2}\right), \]
where $\text{D}_k \cdot / \text{D} t$ is the material derivative defined in \cref{eq:mat_der}. 
We postulate that $\overline{\overline{\tau}}_I$ takes the following form
\begin{equation}\label{eq:tauI}
\overline{\overline{\tau}}_I = \mathcal{B} \left( \vc{u}_{k}-\vc{u}_{k*} \right) \nabla \alpha_k,
\end{equation}
where $\mathcal{B}>0$ is a function of the above objective variables. The term $\nabla \alpha_k$ acts as a ``Delta function-like'' vector that picks out the diffused interface zone. We will show that this definition of $\overline{\overline{\tau}}_I$ is consistent with the second law of thermodynamics under the temperature equilibrium.

In fact, our reduced model to be derived below only includes the mixture momentum equation, where $\overline{\overline{\tau}}_I$ is cancelled out. In general, $\overline{\overline{\tau}}_I$ has an impact on the variation of the volume fraction in the mass diffusion process. %In the present work we consider the mass diffusion under the temperature equilibrium, with which the volume fraction is determined without $\overline{\overline{\tau}}_I$ or $\mathcal{M}_k$.
\end{remark}

\begin{remark}  
In \cref{eq:bn} we neglect the ``viscous pressure'' terms due to the pulsation damping of bubbles \cite{saurel2003multiphase,perigaud2005compressible}. These terms do not impact our model reduction in the limit of the instantaneous pressure relaxation. 
\end{remark}

\subsubsection{Equations for the primitive variables}
In this section, we derive some equations for some primitive variables, which are to be used for further analysis. We introduce the material derivative related to the phase velocity $\vc{u}_k$ and the interfacial velocity $\vc{u}_I$,
\begin{equation}\label{eq:mat_der}
 \frac{\text{D}_g \Phi}{\text{D} \Phi} = \dudx{\Phi}{t} + \vc{u}_g \cdot \nabla{\Phi}, \;\; g = k, I.
\end{equation} 

We also present some thermodynamical relations to be used below:
\begin{equation}\label{eq:gibbs}
T_k \frac{\text{D}_k s_k}{\text{D} t} = \frac{\text{D}_k e_k}{\text{D} t} - \frac{p_k}{\rho_k^2} \frac{\text{D}_k \rho_k}{\text{D} t}, \;\;\frac{\text{D}_k e_k}{\text{D} t} = \chi_k \frac{\text{D}_k \rho_k}{\text{D} t} + \xi_k \frac{\text{D}_k p_k}{\text{D} t}, \;\;
\frac{\text{D}_k p_k}{\text{D} t} = a_k^2 \frac{\text{D}_k \rho_k}{\text{D} t} + \omega_k \frac{\text{D}_k s_k}{\text{D} t},
\end{equation}
where the first expression is the Gibbs relation, $s_k$ is the phase entropy, \[\chi_k = \dudx{e_k}{\rho_k}\Big|_{p_k}, \; \xi_k = \dudx{e_k}{p_k}\Big|_{\rho_k}, \;\omega_k = \dudx{p_k}{s_k}\Big|_{\rho_k},\]
simple manipulations of \cref{eq:gibbs} lead to  $\chi_k = p_k/\rho_k^2 - \xi_k a_k^2$.

 The Mie-Gr{\"u}neisen coefficient $\Gamma_k$ is defined as
 \begin{equation}\label{eq:Gam}
     \Gamma_k = \frac{1}{\rho_k} \dudx{p_k}{e_k}\Big|_{\rho_k} = \frac{1}{\rho_k \xi_k}.
 \end{equation}
 
 With the aid of \cref{eq:Gam}, we reformulate the second relation in \cref{eq:gibbs} as follows
 \begin{equation}\label{eq:dedrhodp}
     \Gamma_k \frac{\text{D}_k e_k}{\text{D} t} = \left( \overline{\gamma}_k - \Gamma_k \right) p_k \frac{\text{D}_k v_k}{\text{D} t} + v_k \frac{\text{D}_k p_k}{\text{D} t},
 \end{equation}
 where the specific volume $v_k = 1/\rho_k$, the adiabatic exponent $\overline{\gamma}_k = A_k / p_k$ and $A_k = \rho_k a_k^2$.

By performing a procedure similar to that in \cite{murrone2005five,zhangPHD2019}, we can deduce the equations with respect to the primitive variables as follows:
%Following a procedure similar to \cite{murrone2005five,zhangPHD2019}, one can obtain the following equations with respect to the primitive variables (including the phase entropy $s_k$, the phase velocity $\vc{u}_k$, the phase pressure $p_k$ and the volume fraction $\alpha_k$):
\begin{subequations} \label{eq:bn_prim}
\begin{align}
\alpha_{k} \rho_{k} T_{k} \frac{\mathrm{D}_{k} s_{k}}{\mathrm{D} t} = \left(\boldsymbol{u}_{I}-\boldsymbol{u}_{k}\right) \cdot \mathcal{M}_{k} + \left(p_{k}-p_{I}\right) {\mathcal{F}}_{k}+\left(p_{I}-p_{k}\right)\left(\boldsymbol{u}_{I}-\boldsymbol{u}_{k}\right) \cdot \nabla \alpha_{k} \nonumber\\  + \left( \vc{u}_k -\vc{u}_I \right) \cdot \left(  \overline{\overline{\tau}}_I \cdot \nabla \alpha_k \right) + \mathcal{G}_k \label{eq:sk}\\
\alpha_{k} \rho_{k} \frac{\mathrm{D}_{k} \boldsymbol{u}_{k}}{\mathrm{D} t} = \nabla\cdot \left( \alpha_k \overline{\overline{T}}_k \right) - \overline{\overline{T}}_I \cdot \nabla \alpha_{k} + \mathcal{M}_{k} \label{eq:uk}\\
\frac{\mathrm{D}_{k} p_{k}}{\mathrm{D} t}= -\frac{\rho_{k} a_{Ik}^{2}}{\alpha_{k}} \mathcal{F}_{k} + \frac{\vc{u}_k - \vc{u}_I}{\alpha_k \rho_k \xi_k} \left[ \left( \overline{\overline{\tau}}_I - \xi_k \rho_k^2 a_{Ik}^2 \overline{\overline{I}} \right) \cdot \nabla\alpha_k - \mathcal{M}_k \right]  + \frac{\mathcal{G}_k }{ \alpha_k \rho_k \xi_k} - \rho_k a_k^2 \nabla\cdot \vc{u}_k \label{eq:pk}\\ 
\frac{\mathrm{D}_{I} \alpha_{k}}{\mathrm{D} t} = \mathcal{F}_{k} \label{eq:alpk}
\end{align}
\end{subequations}
where 
\begin{equation}\label{eq:Gk}
    \frac{\rho_{k} a_{Ik}^{2}}{\alpha_{k}} = \frac{\rho_{k} a_{k}^{2}}{\alpha_{k}} + \frac{p_{I}-p_{k}}{\alpha_{k} \rho_{k} \xi_{k}}, \;\; \mathcal{G}_k = \alpha_k \overline{\overline{\tau}}_k : \overline{\overline{D}}_k + \mathcal{Q}_k + q_k + \mathcal{I}_k.
\end{equation}

One can check that omitting the dissipative terms, the above equations (\cref{eq:sk,eq:uk,eq:pk,eq:alpk}) can be reduced to those in \cite{murrone2005five}.

%\begin{remark}
%With the definition for $\overline{\overline{\tau}}_I$ in \cref{eq:tauI}, the term including $\overline{\overline{\tau}}_I$ on the right hand side of \cref{eq:sk} makes a non-negative contribution.
%\end{remark}

\subsubsection{Equations for the mixture}\label{eq:mixeqns}
In what follows, we derive some average balance equations by replacing the phase velocities with the mass fraction weighted one.
 
\paragraph{The mixture continuity equation} 

\Cref{eq:bn:mass} can be rewritten as:
\begin{equation}\label{eq:mass_diff}
\dudx{\alpha_k\rho_k}{t} + \nabla\dpr(\alpha_k\rho_k \overline{\vc{u}}) = - \nabla\dpr \vc{J}_k,
\end{equation}
where the diffusion flux $\vc{J}_k$ is defined as
\begin{align}\label{eq:Jk}
\vc{J}_k = \alpha_k \rho_k \vc{w}_k.
\end{align}

Note that
\begin{equation}\label{eq:sumJ}
\sum \vc{J}_k = \sum\alpha_k \rho_k \vc{w}_k = 0.    
\end{equation}

In literature there exist some simplified closure relations for the diffusion velocity $\vc{w}_k$ of binary mixtures  \cite{williams1985}, for example:
\begin{enumerate}
    \item[(1)] the Fick's law
\begin{equation}\label{eq:fick_law}
\vc{w}_k = - D \frac{\nabla y_k}{y_k},
\end{equation}
    \item[(2)] the Stefan-Maxwell equation 
\begin{equation}\label{eq:stefan}
\vc{\nabla} X_{k}=\sum_{j=1}^{2} \frac{X_{k} X_{j}}{D}\left(\vc{w}_{j}-\vc{w}_{k}\right).
\end{equation}
\end{enumerate}

Here,  $X_{k}$ is the mole fraction of the component $k$. One can solve  diffusion velocities in \cref{eq:stefan} by using \cref{eq:sumJ}.

The parameter $D$ is the binary diffusion coefficient, for ideal gases, 
\begin{equation}
    D = D_{ij} = \frac{3 \overline{W} k^{0} T}{16 \rho \mu_{ij} \Omega_{ij}},
\end{equation}
where $\overline{W}$ is the average molecular weight of the mixture, $k^{0}$ is the Boltzmann's constant, $\mu_{ij}$ is the reduced mass $\mu_{ij} = M_i M_j / \left( M_i + M_j  \right)$, $M_i$ and $M_j$ are the masses of the colliding molecules, $\Omega_{ij} = \sigma_{ij} v_{ij} /4$ is the collision integral. $v_{ij} = \sqrt{8k^{0}T/\pi \mu_{ij}}$ is the Maxwellian velocity distribution. From this equation follows $D \sim T^{3/2}/p$.

%The value of $D$ typically ranges from $10^{-2}$cm$^2$/s to $10$cm$^2$/s at atmospheric conditions. 

Summing \cref{eq:mass_diff} leads to the equation for the mixture density
\begin{equation}
\dudx{\rho}{t} + \nabla\dpr(\rho \overline{\vc{u}}) = 0.
\end{equation}

\paragraph{The mixture momentum equation} 

Summing \cref{eq:bn:mom}, one can obtain the equation for the mixture momentum
\begin{equation}\label{eq:mix_mom}
\dudx{\rho \overline{\vc{u}}}{t} + \nabla\dpr\left(\rho \overline{\vc{u}} \tpr \overline{\vc{u}} - \overline{\overline{T}} \right) =  - \nabla\dpr\left(\vc{J}_k \tpr {\vc{w}_k} \right),
\end{equation}
where $\overline{\overline{T}} = \sum \alpha_k \overline{\overline{T}}_k = - \overline{P} \; \overline{\overline{I}} + \sum \alpha_k {\overline{\overline{\tau}}}_k, \;\; \overline{P} = \sum \alpha_k p_k $. 

We further separate the stress tensor into average and velocity-disequilibrium parts in the following way
\begin{subequations}\label{eq:newton_vis_disp}
\begin{align}
\sum \alpha_k {\overline{\overline{\tau}}}_k =  {\overline{\overline{\tau}}} + \sum \alpha_k {\overline{\overline{\tau}}}_{wk},
\label{eq:newton_vis_av}\\
{\overline{\overline{\tau}}} = \sum \alpha_k \overline{\overline{\tau}}_{ak} = 2\mu \overline{\overline{D}}_a + \left(\mu_{b} - \frac{2}{3}\mu \right) \nabla \dpr \overline{\vc{u}}, \\ \mu = \sum \alpha_k \mu_k, \;\; \mu_b = \sum \alpha_k \mu_{b,k}.
\end{align}
\end{subequations}

Thus, \Cref{eq:mix_mom} can be recast as follows:
\begin{equation}\label{eq:mix_mom1}
\dudx{\rho \overline{\vc{u}}}{t} + \nabla\dpr\left(\rho \overline{\vc{u}} \tpr \overline{\vc{u}} + \overline{P}\; \overline{\overline{I}}  - \overline{\overline{\tau}} \right) = \sum  \left[ \nabla\dpr \left( \alpha_k \overline{\overline{\tau}}_{wk} \right) - \nabla\dpr\left(\vc{J}_k \tpr {\vc{w}_k} \right)  \right],
\end{equation}

\paragraph{The mixture energy equation} 

And the summation of \cref{eq:bn:en} leads to 
\begin{equation}\label{eq:mix_en}
  \dudx{ \rho E}{t} + 
  \nabla\dpr\left(
\rho E \overline{\vc{u}}   - \overline{\overline{T}} \dpr \overline{\vc{u}}
  \right) 
  = - \nabla \dpr \sum  E_k \vc{J}_k +  \nabla \dpr \sum \alpha_k \overline{\overline{T}}_k \dpr \vc{w}_k + \sum q_k + \sum {\mathcal{I}}_k,
\end{equation}
where the mixture total energy is 
\begin{equation}\label{eq:rhoE}
      \rho E = \sum \alpha_k \rho_k E_k 
      =\rho e + \rho \frac{{|\overline{\vc{u}}}|^2}{2}   + \sum \alpha_k \rho_k \frac{|\vc{w}_k |^2}{2}, \quad \rho e = \sum \alpha_k \rho_k  e_k.
\end{equation}

Note that
\begin{equation}
\sum  E_k \vc{J}_k = \sum \left( e_k + \frac{ \left( \overline{\vc{u}} + \vc{w}_k \right) \dpr  \left( \overline{\vc{u}} + \vc{w}_k \right) }{2} \right)\vc{J}_k  = \sum e_k \vc{J}_k  + \sum \vc{J}^{ww}_{k} + \sum \vc{J}^{uw}_{k},
\end{equation}
\[\vc{J}^{ww}_{k} = \frac{1}{2} |\vc{w}_k|^2 \vc{J}_k, \;\; \vc{J}^{uw}_{k} =  \left( \overline{\vc{u}} \dpr \vc{w}_k \right) \vc{J}_k,\]
and 
\begin{align}
\sum \alpha_k \overline{\overline{T}}_k \dpr \vc{w}_k = \sum \alpha_k \overline{\overline{\tau}}_k \dpr \vc{w}_k - \sum \frac{p_k}{\rho_k} \vc{J}_k, \label{eq:Jk_vis}
\end{align}
thus, with the aid of \cref{eq:newton_vis_disp}, \cref{eq:mix_en} can be recast as
\begin{equation}\label{eq:mix_en1}
  \dudx{ \rho E}{t} + 
  \nabla\dpr\left(
\rho E \overline{\vc{u}} + \overline{P} \; \overline{\vc{u}}  - \overline{\overline{\tau}} \dpr \overline{\vc{u}}
  \right) 
  =  - \nabla \dpr \sum \left( \vc{J}^{h}_{k} + \vc{J}^{vis}_k +  \vc{J}^{ww}_{k} + \vc{J}^{uw}_{k}\right) + \sum q_k + \sum {\mathcal{I}}_k,
\end{equation}
where  $h_k = e_k + p_k/\rho_k$ is the phase enthalpy,
\begin{equation}\label{eq:Jhk}
\vc{J}^{h}_{k} = h_k \vc{J}_k
\end{equation}

The term $\vc{J}^{h}_{k}$ is hereby termed as the enthalpy diffusion flux. 

\begin{equation}
\vc{J}^{vis}_k = - \alpha_k \overline{\overline{\tau}}_{wk} \dpr \overline{\vc{u}} - \alpha_k \overline{\overline{\tau}}_{ak} \dpr \vc{w}_k - \alpha_k \overline{\overline{\tau}}_{wk} \dpr \vc{w}_k,
\end{equation}
is the viscous diffusion flux. Note that the first term comes from the term $-\overline{\overline{T}}\dpr\overline{\vc{u}}$ on the left of \cref{eq:mix_en}, and the last two terms come from the decomposition of the viscous part of \cref{eq:Jk_vis}.

% Without the viscous terms, \cref{eq:mix_en1} is reduced to the result given in \cite{Ramshaw2002}. 

In summary, we have obtained the mixture balance equations for mass (\cref{eq:mass_diff}), momentum (\cref{eq:mix_mom1}), and energy(\cref{eq:mix_en1}). The left side of these equations takes the same form as the single phase NS equation.

\paragraph{The mixture entropy equation}

We define the mixture entropy by assuming the additivity of phase entropies,
\begin{equation}
\rho s = \sum \alpha_k \rho_k s_k,
\end{equation}
and the mixture material derivative along the streamline of phases
\begin{equation}\label{eq:mixture_mat_der}
\rho \frac{\mathrm{D}_{m} \Phi}{\mathrm{D}_{m} t} = \sum \left[ \dudx{\alpha_k \rho_k \Phi_k}{t} + \nabla \dpr \left(\alpha_k \rho_k \vc{u}_k \Phi_k \right)  \right] = \sum \alpha_k \rho_k \frac{\mathrm{D}_{k} \Phi_k}{\mathrm{D}_{k} t}.
\end{equation}

With \cref{eq:mixture_mat_der}, the mixture material derivative for the mixture entropy is
\begin{equation}\label{eq:mixture_mat_der1}
\rho \frac{\mathrm{D}_{m} s}{\mathrm{D}_{m} t} = \sum \alpha_k \rho_k \frac{\mathrm{D}_{k} s_k}{\mathrm{D}_{k} t} = \rho \frac{\mathrm{D}^{ex}_{m} s}{\mathrm{D}_{m} t} + \rho \frac{\mathrm{D}^{in}_{m} s}{\mathrm{D}_{m} t},
\end{equation}
where $\rho {\mathrm{D}^{ex}_{m} s}/{\mathrm{D}_{m} t}$ and $\rho {\mathrm{D}^{in}_{m} s}/{\mathrm{D}_{m} t}$ represent the external entropy flux and the internal entropy production, respectively. The entropy flux is
\begin{equation}
\rho \frac{\mathrm{D}^{ex}_{m} s}{\mathrm{D}_{m} t} = \sum \nabla\dpr \frac{\vc{q}_k}{T_k}.
\end{equation}

For an irreversible process, the entropy production should be non-negative.

\subsection{Reduction of the of the seven-equation model}
%Although physically complete, due to the complex wave structure and stiff relaxations of the BN model, its numerical implementation is rather cumbersome. According to the evaluations of different mechanical/thermal relaxation time scales in \cite{kapila2001two}, this model can be simplified significantly for certain application scenarios. Thus, a variety of reduced models consisting of three to six equations are proposed in literature \cite{kapila2001two,Lemartelot2014,murrone2005five,lund2012hierarchy}. 
%
%One of the best-known reduced model is the Kapila's five-equation model \cite{kapila2001two} that is developed for modelling the deflagration-to-detonation transition in granular materials. This model is developed in the absence of viscosity, heat conduction and external source terms. It is a limit model of the BN model in the case of instantaneous mechanical (pressure and velocity) relaxations. This model is widely used for simulating compressible flows with resolved fluid interfaces thanks to its simplicity and ability to ensure the interface jump conditions (i.e., the continuity in pressure and normal velocity). 

In the present work, we do not attempt to solve the complete seven-equation model, which involves complicated wave structure and stiff relaxations. Instead, we derive a simplified version of the seven-equation model in a manner similar to that of the derivation of the one-velocity one-pressure Kapila's model \cite{kapila2001two}. However, Kapila's five-equation model assumes instantaneous velocity equilibrium and pressure equilibrium. The former assumption strips this model of the capability to model the mass diffusion that is characterized by the velocity difference. 
To restore this ability, the velocity non-equilibrium is retained in the model presented below.

In our approach, the velocity non-equilibrium is reserved by assuming different time scales of the velocity relaxation and the pressure relaxation. We assume that the velocity relaxation time $\varepsilon_{u} = \varepsilon$ and the pressure relaxation time scale $\varepsilon_{p} \leq \mathcal{O}(\varepsilon^2)$.  The corresponding relaxation rates are $\vartheta = \frac{1}{\varepsilon}$ and $\eta \geq \mathcal{O} (\frac{1}{\varepsilon ^ 2})$, respectively.  
This assumption is adopted  on the basis of the following arguments:
\begin{enumerate}
    \item[(1)] We are interested in problems in presence of strong shocks such as detonation and ICF, where the surface tension is negligible.
    \item[(2)] As estimated in \cite{kapila2001two} for the DDT problem, the time scales for the velocity relaxation, the pressure relaxation, and the temperature relaxation are 0.1$\mu$s, 0.03$\mu$s, and 18 ms, respectively. In this problem, the time scale for the pressure relaxation  is approximately one order smaller than that of the velocity relaxation. Moreover, large amount of physical evaluations show that in common cases $ 0 \sim \varepsilon_p < \varepsilon_u < \varepsilon_T < \varepsilon_g$ \cite{guillard2006numerical,bilicki1996evaluation,petitpas2009modelling,zein2010}, where $\varepsilon_T$ and $\varepsilon_g$ are the relaxation time for temperature and chemical potential, respectively.
\end{enumerate}

The reduction of the seven-equation model is to derive a limit model when $\varepsilon \to 0$ and $\vartheta \to \infty$, $\eta \to \infty$. The difference of our approach from Kapila's consists in  the following two aspects:
\begin{enumerate}
    \item[(1)] Different time scales are assumed for the pressure relaxation and the velocity relaxation, while the same time scale is used in deriving the Kapila's model \cite{murrone2005five,kapila2001two}.
    \item[(2)] The reduced model is an approximation of the seven-equation model after reserving terms to the order $\mathcal{O}(\varepsilon)$ and abandoning smaller terms, while  Kapila's model keeps terms of order $\mathcal{O}(1)$.
\end{enumerate}

Such assumptions and manipulations allow the velocity disequilibrium and thus can be used to model the mass diffusion.

\subsubsection{The pressure relaxation}
In this section we drive the phase pressures into equilibrium with the condition $\varepsilon_p  \to 0$.

The primitive form of the BN model (\cref{eq:bn_prim}) can be recast in the following vector form:
\begin{equation}
\frac{\partial \boldsymbol{U}}{\partial t}+\sum_{d=1}^{3} \vc{A}_{d}(\vc{U}) \frac{\partial \vc{U}}{\partial x_{d}}=\frac{1}{\varepsilon _p} \vc{H}(\vc{U})+\vc{R}(\vc{U}),
\end{equation}
where
$\vc{U}=\left[\begin{array}{lllllllllllllll}
s_{1} & s_{2} & \vc{u}_{1} & \vc{u}_{2} & p_{1} & p_{2} & \alpha_{1}
\end{array}\right]$, $\vc{H}(\vc{U})$ is a vector containing the pressure relaxation terms, $\vc{R}(\vc{U})$ is the right hand side terms containing the velocity relaxation and diffusion terms, $d$ is the dimension index.

Without loss of generality, we consider the following one-dimensional split form 
\begin{equation}\label{eq:1D_vecForm}
\frac{\partial \vc{U}}{\partial t}=\vc{L}(\vc{U})+\frac{1}{\varepsilon_p} \vc{H}(\vc{U}), \quad \vc{L}(\vc{U})=-\vc{A}(\vc{U}) \frac{\partial \vc{U}}{\partial x}+\vc{R}(\vc{U})
\end{equation}

Let us assume the following asymptotic expansion of the solution $\vc{U}$ in the vicinity of the equilibrium one $\vc{U}^{(0)}$:
\begin{equation}\label{eq:asympP}
\vc{U}=\vc{U}^{(0)}+ \varepsilon_p \vc{U}^{(1)} + \mathcal{O}\left(\varepsilon_p^{2}\right).
\end{equation}
where $ \varepsilon_p \vc{U}^{(1)}$ represents a fluctuation of order $\varepsilon_p$ in the neighbourhood of $\vc{U}^{(0)}$.

The functions $\vc{L}(\vc{U})$ and $\vc{H}(\vc{U})$ are regular enough to allow  Taylor series expansion, with the aid of which \cref{eq:1D_vecForm} becomes
\begin{equation}\label{eq:asymP_expansion}
\left[\frac{\partial \vc{U}^{(0)}}{\partial t}-\vc{L}\left(\vc{U}^{(0)}\right)-\frac{\partial \vc{H}(\vc{U})}{\partial \vc{U}}\left(\vc{U}^{(0)}\right) \cdot \vc{U}^{(1)}\right] -\frac{1}{\varepsilon_p} \vc{H}\left(\vc{U}^{(0)}\right)+\mathcal{O}(\varepsilon_p)=0
\end{equation}

In the order of $\mathcal{O}(1/\varepsilon_p)$, we have 
\begin{equation}\label{eq:orderep}
\vc{H}\left(\vc{U}^{(0)}\right)=0,
\end{equation}
which gives 
\begin{equation}\label{eq:epP}
p_1^{(0)} = p_2^{(0)}.
\end{equation}

Neglecting terms of order $\mathcal{O}(\varepsilon_p)$ and smaller ones, we have 
\begin{equation}\label{eq:order1}
\frac{\partial \vc{U}^{(0)}}{\partial t}-\vc{L}\left(\vc{U}^{(0)}\right)-\frac{\partial \vc{H}(\vc{U})}{\partial \vc{U}}\left(\vc{U}^{(0)}\right) \cdot \vc{U}^{(1)}=0.
\end{equation}

Combination of \cref{eq:orderep} and \cref{eq:order1} leads to 
\begin{subequations} \label{eq:bn_prim_1}
\begin{align}
\alpha_{k} \rho_{k} T_{k} \frac{\mathrm{D}_{k} s_{k}}{\mathrm{D} t} = \left(\vc{u}_{I}-\vc{u}_{k}\right) \cdot \mathcal{M}_{k} +  \left( \vc{u}_k -\vc{u}_I \right) \cdot \left(  \overline{\overline{\tau}}_I \cdot \nabla \alpha_k \right) + \mathcal{G}_k \label{eq:sk_1}\\
\alpha_{k} \rho_{k} \frac{\mathrm{D}_{k} \vc{u}_{k}}{\mathrm{D} t} = \nabla\cdot \left( \alpha_k \overline{\overline{T}}_k \right) - \overline{\overline{T}}_I \cdot \nabla \alpha_{k} + \mathcal{M}_{k} \label{eq:uk_1}\\
\frac{\mathrm{D}_{k} p_{k}}{\mathrm{D} t}= -\frac{\rho_{k} a_{Ik}^{2}}{\alpha_{k}} \mathcal{F}_{k}^{(1)} + \mathcal{W}_k + \frac{\mathcal{G}_k }{ \alpha_k \rho_k \xi_k} - \rho_k a_k^2 \nabla\cdot \vc{u}_k \label{eq:pk_1}\\ 
\frac{\mathrm{D}_{I} \alpha_{k}}{\mathrm{D} t} = \mathcal{F}_{k}^{(1)} \label{eq:alpk_1}
\end{align}
\end{subequations}
where 
\begin{equation}\label{eq:Wk}
  \mathcal{W}_k =  \frac{\vc{u}_k - \vc{u}_I}{\alpha_k \rho_k \xi_k} \left[ \left( \overline{\overline{\tau}}_I - \xi_k \rho_k^2 a_{Ik}^2 \overline{\overline{I}} \right) \cdot \nabla\alpha_k - \mathcal{M}_k \right].
\end{equation}

For simplicity the superscript ``(0)'' over the variables in \cref{eq:bn_prim_1,eq:Wk} is omitted. The terms including $\vc{U}^{(1)}$ are lumped into $\mathcal{F}_{k}^{(1)}$.

Summing \cref{eq:pk_1} over $k$, one can obtain 
\begin{equation}
\frac{\text{D}_{m} p}{\text{D}_{m} t} + A \nabla \dpr \vc{u} =  A \sum_{k} \frac{\Gamma_k \mathcal{G}_k}{A_k} + A \sum_{k} \frac{\alpha_k \mathcal{W}_k - \alpha_k \vc{w}_k \dpr \nabla p}{A_k} - A \sum \alpha_k \nabla\dpr \vc{w}_k,
\end{equation}
where $1/A = \sum \alpha_k /A_k$.

With \cref{eq:pk_1,eq:epP} and $\mathcal{F}_{1}^{(1)} + \mathcal{F}_{2}^{(1)} = 0$, one can solve
\begin{subequations}
\begin{align}\label{eq:Fk0}
    \mathcal{F}_{1}^{(1)} &= \mathcal{F}_{1,Kap}^{(1)} + \mathcal{F}_{1,Diff}^{(1)},  \\
    \mathcal{F}_{1,Kap}^{(1)} &= \alpha_1 \alpha_2 \frac{ (A_2 - A_1) \nabla \dpr \vc{u} - (\mathcal{G}_{a2} \Gamma_2 / \alpha_2 - \mathcal{G}_{a1} \Gamma_1 / \alpha_1 ) }{A_1 \alpha_2 + A_2 \alpha_1},  \\
    \mathcal{F}_{1,Diff}^{(1)} &= \alpha_1 \alpha_2 \frac{ (A_2 \nabla \dpr \vc{w}_2 - A_1 \nabla \dpr \vc{w}_1)  + (\vc{w}_2 - \vc{w}_1)\dpr \nabla p + ( \mathcal{W}_2 -  \mathcal{W}_1) }{A_1 \alpha_2 + A_2 \alpha_1}\\
   & + \alpha_1 \alpha_2 \frac{ (\alpha_1 \overline{\overline{\tau}}_{1} : \overline{\overline{D}}_1) \Gamma_1 / \alpha_1 - (\alpha_2 \overline{\overline{\tau}}_{2} : \overline{\overline{D}}_1) \Gamma_2 / \alpha_2   }{A_1 \alpha_2 + A_2 \alpha_1} - \mathcal{F}_{1,vis}^{(1)}.
\end{align}
\end{subequations}

Note that the first term $\mathcal{F}_{1,Kap}^{(1)}$ coincides with the corresponding result of Kapila's model (in the absence of heat conduction and viscosity). This term represents the volume fraction variation due to the compaction effect. The term $\mathcal{G}_{ak}$ is defined by replacing the component viscous dissipation in \cref{eq:Gk} with the average one:  
\[\mathcal{G}_{ak} = \alpha_k \overline{\overline{\tau}}_{ak} : \overline{\overline{D}}_a + \mathcal{Q}_k + q_k + \mathcal{I}_k.\]

The second term $\mathcal{F}_{1,Diff}^{(1)}$ is new and due to the velocity non-equilibrium effect (or the mass diffusion process). All velocity-disequilibrium terms are included in $\mathcal{F}_{1,Diff}^{(1)}$. For the definition of $\mathcal{F}_{1,vis}^{(1)}$, see \cref{eq:F1vis}.

The first term can be further split into five parts according to the corresponding contribution of each physical process
\begin{equation}
\mathcal{F}_{1,Kap}^{(1)} = \mathcal{F}_{1,hd}^{(1)} + \mathcal{F}_{1,vis}^{(1)} + \mathcal{F}_{1,ht}^{(1)} + \mathcal{F}_{1,hc}^{(1)} + \mathcal{F}_{1,ex}^{(1)},
\end{equation}
where the terms due to the hydrodynamic process,  the viscous dissipation, the inter-phase heat transfer, the heat conduction, and the external heat source are as follows
\begin{subequations}
\begin{align}
\mathcal{F}_{1,hd}^{(1)} = \alpha_1 \alpha_2 \frac{ (A_2 - A_1) \nabla \dpr \vc{u} }{A_1 \alpha_2 + A_2 \alpha_1},\\\label{eq:F1vis}
\mathcal{F}_{1,vis}^{(1)} = \alpha_1 \alpha_2 \frac{ (\alpha_1 \overline{\overline{\tau}}_{a1} : \overline{\overline{D}}_a) \Gamma_1 / \alpha_1 - (\alpha_2 \overline{\overline{\tau}}_{a2} : \overline{\overline{D}}_a) \Gamma_2 / \alpha_2   }{A_1 \alpha_2 + A_2 \alpha_1}, \\
\mathcal{F}_{1,ht}^{(1)} = \alpha_1 \alpha_2 \frac{ \mathcal{Q}_1 \Gamma_1 / \alpha_1 - \mathcal{Q}_2 \Gamma_2 / \alpha_2   }{A_1 \alpha_2 + A_2 \alpha_1}, \\
\mathcal{F}_{1,hc}^{(1)} = \alpha_1 \alpha_2 \frac{ q_1 \Gamma_1 / \alpha_1 - q_2 \Gamma_2 / \alpha_2   }{A_1 \alpha_2 + A_2 \alpha_1},\\
\mathcal{F}_{1,ex}^{(1)} = \alpha_1 \alpha_2 \frac{ \mathcal{I}_1 \Gamma_1 / \alpha_1 - \mathcal{I}_2 \Gamma_2 / \alpha_2   }{A_1 \alpha_2 + A_2 \alpha_1}.
\end{align}
\end{subequations}

\subsubsection{The velocity relaxation}
We continue to perform asymptotic analysis of \cref{eq:bn_prim_1} with respect to the velocity relaxation time $\varepsilon_u = \varepsilon \to 0$. In a similar way we can express the velocity in the following asymptotic expansion:
\begin{subequations}
\begin{align}\label{eq:asympU}
{\vc{U}}^{\prime} &= {\vc{U}}^{\prime(0)} + \varepsilon_u {\vc{U}}^{\prime(1)} + \mathcal{O}\left(\varepsilon_u^{2}\right),
\end{align}
\end{subequations}
where ${\vc{U}}^{\prime}$ is the reduced state variable with equilibrium pressure of \cref{eq:bn_prim_1}, $\vc{U}^{\prime} = \left[
s_{1}^{\prime} \;\; s_{2}^{\prime} \;\; \vc{u}_{1}^{\prime} \;\; \vc{u}_{2}^{\prime} \;\; p^{\prime} \;\; \alpha_{1}^{\prime}\right]$.

Similar to the analysis in the above section, one can deduce
\begin{equation}
\vc{u}_{1}^{\prime(0)} = \vc{u}_{2}^{\prime(0)} = \vc{u}^{\prime(0)}.
\end{equation}

Then we have 
\begin{equation}
\overline{\vc{u}}^{\prime} = \vc{u}^{\prime(0)} +  \varepsilon_{u} \left( y_1 \vc{u}_{1}^{\prime(1)}  + y_2 \vc{u}_{2}^{\prime(1)}\right) + \mathcal{O}(\varepsilon_u^2),
\end{equation}
\begin{equation}\label{eq:wk_ep}
\vc{w}_k^{\prime} = \vc{u}_{k}^{\prime} - \overline{\vc{u}}^{\prime} = \varepsilon_u y_{k*} (\vc{u}_{k}^{\prime(1)} - \vc{u}_{k*}^{\prime(1)}) + \mathcal{O} (\varepsilon_u^2).
\end{equation}

From \cref{eq:wk_ep}, we deduce
\begin{equation}
\left| \vc{w}_k^{\prime} \right|^2 = \mathcal{O}(\varepsilon_u^2) = \mathcal{O}(\varepsilon_p).
\end{equation}

At this stage the mixture equations derived in \Cref{eq:mixeqns} still hold. In the reduced model we only retain terms to the order $\mathcal{O}(\varepsilon_u)$.
To be consistent with \cref{eq:asymP_expansion}, the term $\left| \vc{w}_k \right|^2$ should be abandoned in the reduction.

Thus, the R.H.S. of \cref{eq:mix_mom,eq:mix_en1} can be simplified as follows
\begin{equation}
\nabla\dpr\left(\vc{J}_k \tpr {\vc{w}_k} \right)  \approx \vc{0}.
\end{equation}
\begin{equation}
\vc{J}^{ww}_{k} \approx \vc{0}, \;\; \vc{J}^{uw}_{k} \approx \vc{0}.
\end{equation}

The definition of the mixture total energy \cref{eq:rhoE} is reduced to
\begin{equation}
    \rho E \approx \sum \alpha_k \rho_k E_k 
      =\rho e + \frac{1}{2} {\overline{\vc{u}}} \dpr \overline{\vc{u}}.
\end{equation}

Moreover, $\vc{J}^{vis}_k$ becomes 
\begin{equation}\label{eq:Jk_vis1}
\vc{J}^{vis}_k \approx - \alpha_k \overline{\overline{\tau}}_{wk} \dpr \overline{\vc{u}} - \alpha_k \overline{\overline{\tau}}_{ak} \dpr \vc{w}_k,
\end{equation}

\subsubsection{The complete model}
Combining \cref{eq:mass_diff,eq:mix_mom,eq:mix_en,eq:alpk_1}, we summarize the final model in the limit $\varepsilon \to 0$ as follows:
\begin{subequations}\label{eq:final_model}
\begin{align}
\dudx{\alpha_k\rho_k}{t} + \nabla\dpr(\alpha_k\rho_k \overline{\vc{u}}) = - \nabla\dpr \vc{J}_k, \\
\dudx{\rho \overline{\vc{u}}}{t} + \nabla\dpr\left(\rho \overline{\vc{u}} \tpr \overline{\vc{u}} + p \overline{\overline{I}} - \overline{\overline{\tau}} \right) =  \nabla \dpr \sum \alpha_k \overline{\overline{\tau}}_{wk},\\
  \dudx{ \rho E}{t} + 
  \nabla\dpr\left(
\rho E \overline{\vc{u}} + p \overline{\vc{u}}  -  \overline{\overline{\tau}} \dpr \overline{\vc{u}}
  \right) 
  = - \nabla \dpr \sum \left( \vc{J}^{h}_{k} + \vc{J}^{vis}_k \right) + \sum q_k + \sum {\mathcal{I}}_k,\\
  \dudx{\alpha_k}{t} + \overline{\vc{u}} \dpr \nabla\alpha_k = \mathcal{F}_{k}^{(1)},
\end{align}
\end{subequations}
where the terms $\vc{J}_k$, $\vc{J}^{h}_{k}$, $\vc{J}^{vis}_k$, $\mathcal{F}_{k}^{(1)}$ are defined in \cref{eq:Jk}, \cref{eq:Jhk}, \cref{eq:Jk_vis1}, \cref{eq:Fk0}, respectively.

\begin{remark}
It appears that the RHS (right hand side) term of the volume fraction equation $\mathcal{F}_{k}^{(1)}$ is very complicated  in comparision with the Kapila's one-velocity  model.  However, in the case of the concerned scenario where mass diffusion goes under the temperature equilibrium, it can be significantly simplified as we demonstrate below.
\end{remark}

\begin{remark}
In the above model, the mixture stress tensor is a volume fraction weighted average of component stress tensors. It depends on the diffusion velocity $\vc{w}_k$, which is different from the formulation of \cite{Cook2009Enthalpy}. Our formulation is consistent with the analysis of \cite{GEURST1986455,gouin2008dissipative}.
\end{remark}

\subsubsection{Thermodynamical consistency}
\begin{proposition}
The reduced model satisfies the entropy condition
\begin{equation}
\rho {\mathrm{D}^{in}_{m} s}/{\mathrm{D}_{m} t} \geq 0.
\end{equation}
\end{proposition}
\begin{proof}
Since no terms of order $\mathcal{O}\left( {\varepsilon_u}^2 \right)$ participate in \cref{eq:sk_1}, it still holds after velocity relaxation. We write the equation for the entropy production as follows:
\begin{equation}
\rho {\mathrm{D}^{in}_{m} s}/{\mathrm{D}_{m} t} = \sum \frac{1}{T_k} \left( \mathcal{G}_k - \nabla\dpr \frac{\vc{q}_k}{T_k} \right) + \sum \frac{1}{T_k} \left( \overline{\vc{u}}-\vc{u}_{k}\right) \cdot \mathcal{M}_{k}  +  \sum \frac{1}{T_k} \left( \vc{u}_{k} - \overline{\vc{u}} \right) \cdot \overline{\overline{\tau}}_I \cdot \nabla \alpha_k.
\end{equation}
By using \cref{eq:relaxations,eq:tauI,eq:newton_vis,eq:fourier_flux}, one can prove that the three terms on the right hand side are all non-negative after some algebraic manipulations, which is omitted here. 
\end{proof}

\section{Numerical method}
\label{sec:numer_meth}
The model (\ref{eq:final_model}) can be split into five distinct physical processes including the inviscid hydrodynamic process, the viscous process, the heat transfer process, the heat conduction process, the mass diffusion process. The splitting and solution procedures are performed on the basis of physical concerns and assumptions. First, the pressure relaxation takes place much faster than the thermal process. Second, heat conduction and mass diffusion proceeds under temperature and pressure equilibrium. In the first four steps, only the mass fraction averaged velocity is involved. Velocity disequilibrium that leads to the mass diffusion only appears in the mass diffusion process. The last two stages are accompanied by inter-phase heat transfer to maintain the temperature equilibrium.

We write the split processes as follows:
\begin{enumerate}
\item[(a)] The inviscid hydrodynamic process
\begin{subequations}\label{eq:HD}
\begin{align}
\dudx{\alpha_k\rho_k}{t} + \nabla\dpr(\alpha_k\rho_k \overline{\vc{u}}) = 0, \\
\dudx{\rho \overline{\vc{u}}}{t} + \nabla\dpr\left(\rho \overline{\vc{u}} \tpr \overline{\vc{u}} + p\overline{\overline{I}} \right) =  0,\\
  \dudx{ \rho E}{t} + 
  \nabla\dpr\left(
\rho E \overline{\vc{u}} +  p \overline{\vc{u}}
  \right) 
  = 0,\\
  \dudx{\alpha_k}{t} + \overline{\vc{u}} \dpr \nabla\alpha_k = \mathcal{F}_{k,hd}^{(1)},
\end{align}
\end{subequations}

\item[(b)] The viscous process
\begin{subequations}\label{eq:VIS}
\begin{align}
\dudx{\alpha_k\rho_k}{t} = 0, \label{eq:VIS_mk}\\
\dudx{\rho \overline{\vc{u}}}{t} =  \nabla\dpr\left( \overline{\overline{\tau}} \right), \label{eq:VIS_mom}\\
  \dudx{ \rho E}{t}   
  = \nabla\dpr\left(  \overline{\overline{\tau}} \dpr \overline{\vc{u}}
  \right), \label{eq:VIS_en}\\
  \dudx{\alpha_k}{t} = \mathcal{F}_{k,vis}^{(1)},
\end{align}
\end{subequations}

\item[(c)] The heat transfer process
\begin{subequations}\label{eq:HT}
\begin{align}
\dudx{\alpha_k\rho_k}{t}  = 0, \\
\dudx{\rho \overline{\vc{u}}}{t}  =  0,\\
  \dudx{ \rho E}{t} 
  =  0,\\
  \dudx{\alpha_k}{t} = \mathcal{F}_{k,ht}^{(1)},
\end{align}
\end{subequations}

\item[(d)] The heat conduction process
\begin{subequations}\label{eq:HC}
\begin{align}
\dudx{\alpha_k\rho_k}{t}  = 0, \\
\dudx{\rho \overline{\vc{u}}}{t}  =  0,\\
  \dudx{ \rho E}{t} 
  =  \sum q_k + \sum \mathcal{Q}_k^{HC},\\
  \dudx{\alpha_k}{t} =  \mathcal{F}_{k,hc}^{(1)} + \mathcal{F}_{k,hcht}^{(1)},\label{eq:HC_alp}
\end{align}
\end{subequations}
where the term $\mathcal{Q}_k^{HC}$ represents the heat transfer between two components in the course of heat conduction, that drives the phase temperatures towards equilibrium. Although the heat transfer does not impact the mixture energy equation ($\sum \mathcal{Q}_k^{HC} = 0$), it leads to the variation of volume fraction through the term $\mathcal{F}_{k,hcht}^{(1)}$.

\item[(e)] The mass diffusion process
\begin{subequations}\label{eq:MD}
\begin{align}
\dudx{\alpha_k\rho_k}{t} = - \nabla\dpr \vc{J}_k, \label{eq:MD_par_mass}  \\
\dudx{\rho \overline{\vc{u}}}{t}  =   \nabla \dpr \sum \alpha_k \overline{\overline{\tau}}_{wk}, \label{eq:MD_mom}\\
  \dudx{ \rho E}{t} 
  = - \nabla \dpr \sum  \vc{J}^{h}_{k}  +  \nabla \dpr \sum \alpha_k \overline{\overline{\tau}}_{wk} \dpr \overline{\vc{u}}+  \nabla \dpr \sum \alpha_k \overline{\overline{\tau}}_{ak} \dpr {\vc{w}}_k + \sum \mathcal{Q}_{k}^{MD}, \label{eq:MD_en}\\
  \dudx{\alpha_k}{t} = \mathcal{F}_{k,Diff}^{(1)} + \mathcal{F}_{k,mdht}^{(1)}, \label{eq:MD_vol}
\end{align}
\end{subequations}
\end{enumerate}
where the term $\mathcal{Q}_{k}^{MD}$ is the heat transfer between components in the process of mass diffusion and $\sum \mathcal{Q}_k^{MD} = 0$. $\mathcal{F}_{k,mdht}^{(1)}$ is the volume fraction variation caused by the heat transfer $\mathcal{Q}_{k}^{MD}$.

For the solution of this model (\ref{eq:final_model}), we implement the fractional step method, i.e., each set of split governing equations for the physical processes are solved one by one in order. The solution obtained at each step serves as the initial condition for the next step.  

%The viscous, the heat conduction and mass diffusion steps contribute (non-linear) parabolic equations with respect to the velocity, the temperature and the mass fraction, respectively. Due to the commonality in the algorithm for solving these parabolic PDEs, we summarize the corresponding numerical methods in \Cref{subsec:numer_met_para}. 

In numerical implementation, the solution of non-linear parabolic PDEs with respect to the velocity, the temperature and the mass fraction are involved at the heat conduction step, viscous step and mass diffusion step, respectively. For their solution, we implement an efficient explicit local iteration method that is to be described in \Cref{subsec:numer_met_para}.

\subsection{Hydrodynamic part}
%In literature there exists an abundance of numerical methods for solving the hydrodynamic part  (i.e. \cref{eq:HD}). Moreover, the hyperbolicity, Riemann invariants and jump conditions of this sub-system have been investigated in \cite{murrone2005five,kapila2001two}.

The hydrodynamic part  (i.e. \cref{eq:HD}) in fact coincides with the original Kapila's model whose jump conditions, Riemann invariants and numerical solutions have been sufficiently studied in literature \cite{murrone2005five,kapila2001two}.

%To preserve pressure-velocity equilibrium property in the framework of the Godunov FVM, it is suggested to solve the advection equation for the volume fraction instead of seeking a conservative reformulation by invoking the mass conservation equation \cite{abgrall1996prevent}. 

It is established that one should solve the non-conservative advection equation for the volume fraction in DIM for preserving the pressure-velocity equilibrium. Most trials to use the conservative reformulation with the aid of mass conservation fail, as summarized in \cite{Abgrall2001,abgrall1996prevent}.

%In order to implement the Godunov method, we recast the volume fraction equation as follows: 

To implement the Godunov method, we reformulate the volume fraction equation as follows:
\begin{equation}\label{eq:reduced_five_hyper:vol1}
\frac{\partial \alpha_{1}}{\partial t}+   \nabla \cdot \left( \alpha_{1} \vc{u} \right) =  \alpha_1 \nabla\cdot \vc{u} +  \mathcal{F}_{k,hd}^{(1)} = \frac{A}{A_1}\alpha_1 \nabla\cdot \vc{u}.
\end{equation}

The hydrodynamic subsystem can be written in the vector form as follows:
\begin{equation}\label{eq:reduced_five_hyper_conv}
\dudx{\vc{U}}{t} + \nabla\dpr \vc{F}\left(  \vc{U}\right) =  \vc{S}\left(  \vc{U} \right) \nabla\dpr \vc{u},
\end{equation}
where 
\[
\vc{U} = \left[ \alpha_1 \rho_1 \;\; \alpha_2 \rho_2 \;\;  \rho u \;\; \rho v \;\; \rho E \;\; \alpha_1 \right]^{\text{T}}, \quad
\vc{F}\left(\vc{U}\right) = u \vc{U} + p \vc{D},
\]
\[
\vc{D}\left(\vc{U}\right) = \left[ 0 \;\; 0 \;\; 1 \;\; 0 \;\; u \;\; 0 \right]^{\text{T}},
\quad
\vc{S}\left(\vc{U}\right) = \left[ 0 \;\; 0 \;\; 0 \;\;  0  \;\; 0 \;\; \frac{A}{A_1}\alpha_1 \right]^{\text{T}}.
\]

%The equation is discretized on a Cartesian grid. 
%The conservative part of \cref{eq:reduced_five_hyper_conv} (without the right hand side term) is solved with the Godunov method \cite{Godunov1959}. The numerical flux is determined by using the Riemann solution that is approximated with the HLLC scheme \cite{Toro2009Riemann}. For the high order extension, we adopt the fifth order WENO scheme or the MUSCL scheme \cite{Coralic2014Finite,Johnsen2006Implementation,JIANG1996202} for spatial reconstruction of the local characteristic variables  on cell faces. The two-stage Heun method (i.e., the modified Euler method) is used for time integration. 

We use the Godunov method with the HLLC approximate solver \cite{Toro2009Riemann} to evaluate the numerical flux of the conservative part of \cref{eq:reduced_five_hyper_conv} (i.e., temporarily omit the right hand side). High orders are achieved by using the fifth order WENO scheme  \cite{Coralic2014Finite,Johnsen2006Implementation,JIANG1996202} for the spatial reconstruction of the local characteristic variables  on cell faces or the MUSCL scheme \cite{Toro2009Riemann,leveque2002finite} for the spatial reconstruction of the primitive physical variables. The two-stage Heun method (i.e., the modified Euler method) is used for the time integration.

The non-conservative term $\vc{S}\left(  \vc{U} \right) \nabla\dpr \vc{u}$ is calculated as follows
\begin{equation}
\frac{1}{V_{i j k}} \int_{V_{i j k}} \frac{A}{A_1}\alpha_1 \nabla \cdot \vc{u} \mathrm{d} V \approx \frac{1}{V_{i j k}} \left( \frac{A}{A_1}\alpha_1 \right)_{i j k} \int_{{\sigma}_{i j k}} \vc{u} \cdot \vc{n} \mathrm{d} {\sigma},
\end{equation}
where the subscript $_{ijk}$ denotes the index of the considered cell. The denotations ${V_{i j k}}$, ${{\sigma}_{i j k}}$ and $\vc{n}$ are the cell volume, the surface, and the surface normal, respectively. The variables $A, \; A_1, \; \alpha_1$ are taken to be the cell-averaged values as in \cite{Tiwari2013A,Johnsen2006Implementation,Coralic2014Finite}.

%The solution of the hyperbolic part is denoted as $\vc{U}_{pri}^{(1)} = [\rho_1^{(1)} \;\; \rho_2^{(1)} \;\; \vc{u}^{(1)} \;\; p^{(1)} \;\; \alpha_1^{((1))}], $ and serves as the initial data for the following computation.

\subsection{Viscous part}
Observing \cref{eq:VIS_mk}, it can be seen that $m_k = \alpha_k \rho_k$ does not vary at this stage, nor does the mixture density $\rho = \sum m_k$ or the mass fraction $y_k$. 

The momentum equation and energy equation (\ref{eq:VIS_en}) can be rewritten in the following form
\begin{subequations}
\begin{align}
    \rho \dudx{ \overline{\vc{u}}}{t} =  \nabla\dpr  \overline{\overline{\tau}}, \label{eq:vis_mom}\\
\rho \dudx{  E}{t} =  \nabla \dpr \left( \overline{\overline{\tau}} \dpr \overline{\vc{u}} \right). \label{eq:vis_en}
\end{align}
\end{subequations}

% \Cref{eq:vis_mom} forms a parabolic PDE set, which is reduced to the following form in 1D:
% \begin{equation}\label{eq:viscous_para_pde}
% \rho \frac{\partial \overline{u}}{\partial t}=\frac{\partial}{\partial x}\left(\frac{4}{3} \mu \frac{\partial \overline{u}}{\partial x}\right) + \sum \frac{\partial}{\partial x} \left(  \frac{4}{3} \alpha_k \mu_k \dudx{{w}_k}{x} \right),
% \end{equation}
% where the mixture dynamic viscosity $\mu=\sum \alpha_k \mu_k$.

 \Cref{eq:vis_mom} forms a parabolic PDE set, which is reduced to the following form in 1D:
 \begin{equation}\label{eq:viscous_para_pde}
 \rho \frac{\partial \overline{u}}{\partial t}=\frac{\partial}{\partial x}\left(\frac{4}{3} \mu \frac{\partial \overline{u}}{\partial x}\right),
 \end{equation}
 where the mixture dynamic viscosity $\mu=\sum \alpha_k \mu_k$.

In general, the parabolic PDE set is non-linear due to the dependence of the coefficients on the unknowns.
For some application scenarios, the phase viscosity depends on the phase density $\rho_k$, the pressure $p$ and the temperature $T_k$, i.e., $\mu_k = \mu_k (\rho_k, p_k, T_k)$ and   the mixture viscosity $\mu = \mu (\alpha_k,  p_k, \rho_k, T_k)$. The temperature is subject to the impact of the viscosity terms. The latter varies with the temperature and the pressure. Such non-linearity issues are considered by using the method of iterations, where the coefficients are frozen in each iteration. By doing so, \cref{eq:viscous_para_pde} represents a linear PDE in each iteration. The linearized parabolic PDE is solved with the LIM algorithm \cite{Zhukov2010}. The numerical methods for solving such parabolic equations are  summarized in \Cref{subsec:numer_met_para}.

\subsection{Temperature relaxation part}\label{subsec:TR}
Simple algebraic manipulations of  \cref{eq:HT} give
\begin{equation}\label{eq:mkdek}
m_k = m_k^{(0)} = const, \;\; {\overline{\vc{u}}} ={\overline{\vc{u}}}^{(0)} =  const, \;\; \rho E = (\rho E)^{(0)} = const,
\end{equation}
where the superscript ``${(0)}$'' represent the variables at the beginning of the current stage.

Combination of the first three equations in \cref{eq:HT} leads to
\begin{equation}\label{eq:dedt_HT}
\dudx{e}{t} = 0.
\end{equation}

The reduced model is in pressure equilibrium, which means:
 \begin{equation}\label{eq:pres_eq}
p_1\left( T_1, \rho_1 \right) = p_2\left( T_2, \rho_2 \right) = p.
 \end{equation}

The saturation condition for volume fractions leads to
\begin{equation}\label{eq:vol_saturation}
\frac{m_1}{\rho_1} + \frac{m_2}{\rho_2} = 1.
\end{equation}

With \cref{eq:pres_eq,eq:vol_saturation}, the phase density can be expressed as 
\begin{equation}\label{eq:rhok_fun}
\rho_k = \rho_k \left(m_1, m_2, T_1, T_2 \right).
\end{equation}

Further, we obtain
\begin{equation}\label{eq:ek_mk_Tk}
e_k = e_k \left( \rho_k, T_k \right) = e_k \left(m_1, m_2, T_1, T_2 \right),
\end{equation}
and
\begin{equation}\label{eq:e_fun}
e = \sum y_k e_k = \sum \frac{m_k}{m_1 + m_2} e_k \left(m_1, m_2, T_1, T_2 \right) =  e \left(m_1, m_2, T_1, T_2 \right).
\end{equation}

Combination of \cref{eq:mkdek,eq:e_fun,eq:dedt_HT} gives
\begin{align}\label{eq:tr_eq1}
\mathcal{A}_{1}\dudx{T_1}{t} + \mathcal{A}_{2} \dudx{T_2}{t} = 0,
\end{align}
where \[\mathcal{A}_{1} = \dudx{e}{T_1}, \;\; \mathcal{A}_{2} = \dudx{e}{T_2}.\]

The time derivative is approximated as
\begin{equation}\label{eq:dTk_dis}
\dudx{T_k}{t} = \frac{T_k^{\prime} - T_k^{(0)}}{\Delta t},
\end{equation}
here and below the superscripts ``$(0)$'' and ``$\prime$'' represent the variables at the beginning and the end of the current stage, respectively.  

We assume the heat transfer is large enough to reach a temperature equilibrium at the end of the current time step, thus, we have 
\begin{equation}\label{eq:Teq}
T_1^{\prime} = T_2^{\prime} = T^{\prime}.
\end{equation}

By using \cref{eq:tr_eq1,eq:dTk_dis,eq:Teq} one can obtain:
 \begin{equation}\label{eq:tr_temp_av}
 T^{\prime} = \frac{\mathcal{A}_1 T_1^{(0)} + \mathcal{A}_2 T_2^{(0)}}{\mathcal{A}_1 + \mathcal{A}_2}
 \end{equation}
 
Having $T^{\prime}$, we can solve for $\rho_k^{\prime}$ with \cref{eq:rhok_fun}, and then for $p^{\prime}$ with \cref{eq:pres_eq}. Since the partial densityss does not vary, i.e. $m_k^{\prime} = m_k^{(0)}$, the volume fractions can be evaluated with $\alpha_k^{\prime} = m_k^{\prime}/\rho_k^{\prime}$. In this way, we can determine the temperature-relaxed state in each cell.

\begin{remark}
The above manipulations for the temperature  relaxation is based on the infinite relaxation rate assumption. In such case the temperature relaxation term $Q_k$ does not appear explicitly. To deal with finite temperature relaxation, the governing equation for the internal energy of each phase is useful. We can obtain these equations by following a similar procedure in the derivation of \cref{eq:dekdt_hc} in the next subsection:
\begin{subequations}
\begin{align}
\alpha_1 \rho_1 \dudx{e_1}{t} = {\mathcal{Q}}_1 - p\dudx{\alpha_1}{t},\\
\alpha_2 \rho_2 \dudx{e_2}{t} =  {\mathcal{Q}}_2 - p\dudx{\alpha_2}{t}.
\end{align}
\end{subequations}
The temperature relaxation term $\mathcal{Q}_k$ is prescribed according to specific physical laws.
\end{remark}

\subsection{Heat conduction part}\label{subsec:HC}
The procedure for the heat conduction is totally analogous to that of the heat transfer. From \cref{eq:HC}, one can deduce
\begin{equation}\label{eq:dmkdt0}
\dudx{m_k}{t} = 0, \;\; \dudx{y_k}{t} = 0, \;\; \dudx{\overline{\vc{u}}}{t} = 0.
\end{equation}

Invoking equation (\ref{eq:e_fun}), one can  write the equation for the mixture internal energy as
\begin{equation}\label{eq:e_hc}
\rho \dudx{e}{t} = \rho \left( \mathcal{A}_{1}\dudx{T_1}{t} + \mathcal{A}_{2} \dudx{T_2}{t} \right) = \sum_k q_k + \sum_k {\mathcal{Q}}_k^{HC}.
\end{equation}

By using the definition of the mixture internal energy $e = \sum_k y_k e_k$, from \cref{eq:e_hc} one can deduce
\begin{equation}\label{eq:diff_e}
\alpha_1 \rho_1 \dudx{e_1}{t} + \alpha_2 \rho_2 \dudx{e_2}{t} =  q_1 + q_2 + {\mathcal{Q}}_1^{HC} + {\mathcal{Q}}_2^{HC}.
\end{equation}

Aided by \cref{eq:dedrhodp}, differentiation of \cref{eq:pres_eq} yields
\begin{equation}\label{eq:diff_pres}
\frac{p\Gamma_1 - A_1}{\alpha_1} \dudx{\alpha_1}{t} + \rho_1 \Gamma_1 \dudx{e_1}{t} = \frac{p\Gamma_2 - A_2}{\alpha_2} \dudx{\alpha_2}{t} + \rho_2 \Gamma_2 \dudx{e_2}{t}.
\end{equation}

Solution of \cref{eq:diff_e,eq:diff_pres} with respect to $\alpha_k \rho_k \dudx{e_k}{t}$ gives
\begin{subequations}
\begin{align}\label{eq:dekdt_hc}
\alpha_1 \rho_1 \dudx{e_1}{t} = q_1 + {\mathcal{Q}}_1^{HC} - p\dudx{\alpha_1}{t},\\
\alpha_2 \rho_2 \dudx{e_2}{t} = q_2 + {\mathcal{Q}}_2^{HC} - p\dudx{\alpha_2}{t},
\end{align}
\end{subequations}
where the last term $p\dudx{\alpha_k}{t}$ represents the thermodynamical work due to the motion of the interface. The term $\dudx{\alpha_k}{t}$ is defined in \cref{eq:HC_alp} and depends on $q_k$, $\mathcal{Q}_k^{HC}$.

The temperature relaxation term $\mathcal{Q}_k^{HC}$ represents the heat exchange between phases, which drives phase temperatures towards equilibrium. 
%It may take different forms for different physical processes.  
Here, we assume such a model for $\mathcal{Q}_k^{HC}$ that the phase temperature equilibrium is maintained in the course of the multicomponent heat conduction, i.e.,
\begin{equation}\label{eq:temp_equi}
 \dudx{T_1}{t} = \dudx{T_2}{t} = \dudx{T}{t}.
 \end{equation} 

Note that the condition \cref{eq:temp_equi} is in fact an implicit condition assumed in one-temperature models, for example, the conservative model used in \cite{Lemartelot2014}.

By using \cref{eq:dmkdt0,eq:temp_equi,eq:dekdt_hc,eq:ek_mk_Tk}, one can deduce
\begin{equation}\label{eq:deT}
\sum m_k  \dudx{e_k}{T_k} \dudx{T}{t} = \sum q_k,
\end{equation}
where $\dudx{e_k}{T_k} = C_{v,k}$ for the ideal gas EOS that is considered in the current work. Solving \cref{eq:deT}, one an obtain the temperature $T^{\prime}$ at the end of this stage. Having $T^{\prime}$, the other variables are computed in the same way as in the last (heat transfer) stage.

\subsection{Mass diffusion part}
Different from previous stages where the partial densities $m_k$ remain constant, at this stage the mass diffusion leads to the variation of partial densities with time, as can be seen from \cref{eq:MD_par_mass}.  

Since $\sum\vc{J}_k = \vc{0}$, summing \cref{eq:MD_par_mass} over $k$ leads to
\begin{equation}\label{eq:sumrhok}
 \dudx{\rho}{t} = 0.
 \end{equation} 

From \cref{eq:sumrhok,eq:MD_mom} follows that
\begin{equation}\label{eq:MD_vel}
\rho \dudx{\overline{\vc{u}}}{t} = \sum \nabla \dpr \overline{\overline{\tau}}_{wk}.
\end{equation}

To compute the RHS of \cref{eq:MD_vel}, we need $\vc{w}_k$ that is evaluated explicitly according to \cref{eq:fick_law}.

Combination of \cref{eq:sumrhok,eq:MD_par_mass,eq:fick_law} leads to
\begin{equation}\label{eq:md_para_pde}
\dudx{m_k}{t} = \rho \dudx{y_k}{t} =  \nabla \dpr \left( \rho D \nabla y_k  \right).
\end{equation}

Solving the parabolic PDE (\ref{eq:md_para_pde}) yields the parameters at the end of the mass diffusion stage: $y_k^{\prime}$, $m_k^{\prime}$. Explicit solution of \cref{eq:MD_en} yields  $({\rho E})^{\prime}$. 

In defining $p^{\prime}$ and $\alpha_k^{\prime}$, we use the temperature equilibrium condition which is assumed in the Fick's law and the Stefan-Maxwell law. 
%In these simplified diffusion laws, only an equilibrium temperature is involved. In the present work we also follow this assumption. In our model the mass diffusion is accompanied with the heat transfer that maintains the phase temperature equilibrium \cref{eq:temp_equi}.

With \cref{eq:sumrhok,eq:MD_vel,eq:MD_en} one can deduce 
\begin{equation}\label{eq:MD_diff_e}
\rho \dudx{e}{t} = \alpha_1 \rho_1 \dudx{e_1}{t} + \alpha_2 \rho_2 \dudx{e_2}{t} + \mathcal{C}_1 e_1 + \mathcal{C}_2 e_2 =  \mathcal{E}^{c}_1 + \mathcal{E}^{c}_2,
\end{equation}
here for simplicity we introduce 
\begin{subequations}
\begin{align*}
\mathcal{C}_k = - \nabla \dpr \vc{J}_k,\\
\mathcal{E}^{c}_k = -\nabla \cdot  \boldsymbol{J}_{k}^{h} +  \alpha_k \overline{\overline{\tau}}_{wk}:\overline{\overline{D}}_a +\nabla \cdot  \alpha_{k} \overline{\overline{\tau}}_{a k} \cdot \boldsymbol{w}_{k}.
\end{align*}
\end{subequations}

Further, from \cref{eq:MD_diff_e,eq:diff_pres} follows
\begin{subequations}\label{eq:MD_dekdt}
\begin{align}
\alpha_1 \rho_1 \dudx{e_1}{t} = - \widetilde{p} \dudx{\alpha_1}{t}  + \widetilde{\mathcal{E}}_1,\\
\alpha_2 \rho_2 \dudx{e_2}{t} = - \widetilde{p} \dudx{\alpha_2}{t}  + \widetilde{\mathcal{E}}_2,
\end{align}
\end{subequations}
where 
\begin{subequations}
\begin{align}
\widetilde{p} &= p - \frac{A_1/\alpha_1 + A_2/\alpha_2}{\Gamma_1/\alpha_1 + \Gamma_2/\alpha_2}, \\
\widetilde{\mathcal{E}}_1 &=  \frac{\mathcal{E}_s \Gamma_2 /\alpha_2 + \mathcal{H}_2 - \mathcal{H}_1}{\Gamma_1/\alpha_1 + \Gamma_2/\alpha_2}, \\
\widetilde{\mathcal{E}}_2 &= \frac{\mathcal{E}_s \Gamma_1 /\alpha_1 + \mathcal{H}_1 - \mathcal{H}_2}{\Gamma_1/\alpha_1 + \Gamma_2/\alpha_2} , \\
\mathcal{E}_s &= - \mathcal{C}_1 e_1  - \mathcal{C}_2 e_2  + \mathcal{E}^{c}_1 + \mathcal{E}^{c}_2,\\
\mathcal{H}_k &= \frac{\mathcal{C}_k \left( A_k - \Gamma_k p \right)}{\rho_k}.
\end{align}
\end{subequations}

Heat transfer leads to the variation of the volume fraction. This variation can be determined from \cref{eq:MD_dekdt,eq:ek_mk_Tk} as follows:

\begin{equation}\label{eq:MD_dalpdt}
\dudx{\alpha_1}{t} = 
\frac{m_1 m_2 \left(  e_{1,T} e_{2,m} - e_{2,T} e_{1,m} \right) +  \left( m_2 e_{2,T} \widetilde{\mathcal{E}}_1  - m_1  e_{1,T} \widetilde{\mathcal{E}}_2    \right)}{ \left( m_1 e_{1,T}  + m_2 e_{2,T}\right)  \widetilde{p}}.
\end{equation}
where 
\begin{subequations}
\begin{align}
    e_{k,T} = \dudx{e_k}{T_1} + \dudx{e_k}{T_2},\\
    e_{k,m} = \dudx{e_k}{m_1} \mathcal{C}_1 + \dudx{e_k}{m_2} \mathcal{C}_2.
\end{align}
\end{subequations}

For the ideal gas, we have
\begin{equation}\label{eq:dekdTj}
    \dudx{e_k}{T_j} = \delta_{kj} C_{v,k}, \;\; \dudx{e_k}{m_j} = 0,
\end{equation}
where $\delta_{kj}$ is the Kronecker function.

In this case \cref{eq:MD_dalpdt} can be simplified to a large extent as follows
\begin{equation}\label{eq:MD_dalpdt1}
\dudx{\alpha_1}{t} = 
\frac{  m_2 C_{v,2} \widetilde{\mathcal{E}}_1  - m_1  C_{v,1} \widetilde{\mathcal{E}}_2   }{ \left( m_1 C_{v,1}  + m_2 C_{v,2}\right)  \widetilde{p}}.
\end{equation}

Under the temperature equilibrium assumption we solve \cref{eq:MD_dalpdt1} instead of \cref{eq:MD_vol}. Solution of \cref{eq:MD_par_mass,eq:MD_mom,eq:MD_en,eq:MD_dalpdt} provide a full set of conservative variables at the end of this stage $[m_1^{\prime}, \;\;m_2^{\prime}, \;\;(\rho \vc{\overline{u})}^{\prime}, \;\;(\rho E)^{\prime}, \;\; (\alpha_1)^{\prime} ]$.

%%With \cref{eq:reduced_five_massDiff,eq:PT_condition,eq:rhok_fun}, one can obtain the solution for $\mathcal{Q}_k^c$.
%
%
%Algebraic manipulations of \cref{eq:reduced_five_massDiff} leads to
%\begin{align}\label{eq:md_eq1}
%\mathcal{A}_{k1}\dudx{T_1}{t} + \mathcal{A}_{k2} \dudx{T_2}{t} + \mathcal{C}_k E_k = \mathcal{E}_{k}^c + \frac{p}{\rho_I} \mathcal{C}_k + \mathcal{Q}_k^c.
%\end{align}
%
%\Cref{eq:md_eq1} forms a system of two equations (corresponding to $k=1$ and $k=2$) with two  unknowns: the equilibrium temperature derivative  $\dudxh{T}{t}$, and the temperature relaxation term $\mathcal{Q}_1^c = - \mathcal{Q}_2^c = \mathcal{Q}^c$. The solution of this equation set is given as:
%\begin{align}\label{eq:md_eq2}
%\dudx{T}{t} = \frac{ \sum \left( \mathcal{E}_k^{c} - {\mathcal{C}}_{k} E_{k} \right)}{ \mathcal{A} }.
%\end{align}
%
%The solution of \cref{eq:md_eq2} gives the equilibrium temperature after the current mass diffusion stage $T^{(5)}$.
%
%Having determined $m_k^{(5)}$ and $T_1^{(5)} = T_2^{(5)} =T^{(5)}$, one can further calculate $\rho_k^{(5)}$ according to \cref{eq:rhok_fun} and the pressure according to the EOS. Finally, we obtain the state variable at the end of this stage
%\begin{equation}\label{eq:Upri3}
%\vc{U}_{pri}^{(5)} = [\rho_1^{(5)} \;\; \rho_2^{(5)} \;\; \vc{u}^{(5)} \;\; p^{(5)} \;\; \alpha_2^{((5))}] .
%\end{equation}

We also describe another approach to determine $\dudx{\alpha_k}{t}$ is as follows:

 (1) By using \cref{eq:md_para_pde,eq:MD_diff_e,eq:e_fun,eq:temp_equi}, one can obtain $\dudx{T}{t}$,
 
 (2) Having $\dudx{T}{t}$, by using \cref{eq:rhok_fun,eq:md_para_pde},  we obtain $\dudx{\rho_k}{t}$,
 
 (3) Having $\dudx{\rho_k}{t}$, by using \cref{eq:md_para_pde}, we obtain $\dudx{\alpha_k}{t}$.

We find that these two approaches lead to numerical results with negligible differences.

%\begin{remark}
%The parameters $\vartheta > 0$ (the velocity relaxation rate) and $\mathcal{B} > 0$ (coefficient in the constitution law of $\tau_I$) appear as combination $\vartheta - \mathcal{B} |\nabla \alpha_k|^2$ on RHS of \cref{eq:MD_vol}.
%One can always find two parameters $\vartheta > 0$  and $\mathcal{B} > 0$  that render \cref{eq:MD_dalpdt1} and \cref{eq:MD_vol} consistent.
%\end{remark}

\subsection{Numerical methods for the parabolic diffusion PDEs}\label{subsec:numer_met_para}

The dissipation equations (\cref{eq:viscous_para_pde,eq:deT,eq:md_para_pde}) can be written in the following quasi-linear parabolic PDEs (in 1D) as follows:
\begin{equation}\label{eq:nonlin_para_pde}
\frac{\partial v}{\partial t}=L [v]+f(x, t), \quad x \in \Lambda \subset \mathbb{R}, 
\end{equation}
%where $L[\cdot]$ is a linear elliptic self-adjoint positive-definite operator.
where the operator $L[\cdot]$ represents a quasi-linear elliptic  operator that is positive definite and takes the following form
\begin{equation}
L[v] = \dudx{}{x}\left( k(v) \dudx{v}{x} \right).
\end{equation}

For solution of such non-linear parabolic equations, we use an iterative method as follows
\begin{align}
&v^{(0)} = v^n, \\
&\frac{\partial v^{(s+1)}}{\partial t} = \dudx{}{x}\left( k(v^{(s)}) \dudx{v^{(s+1)}}{x} \right) +f(x, t),\label{eq:lin_para_pde}
\end{align}
where the non-linear coefficient $k(v^{(s)})$ is linearised by assuming dependence on the solutions of last iteration. One can also use more advanced method such as the Newton-Raphson method to speed up the convergence. 

The sequences (\ref{eq:lin_para_pde}) is iterated until convergence that is defined as $||v^{(s+1)} - v^{(s)}|| < \chi$ ($\chi$ is a small positive number). 
To solve the linearised parabolic PDE (\ref{eq:lin_para_pde}) for the unknown $v^{(s+1)}$, one can use various implicit or explicit methods. Here, we use a monotonicity-preserving explicit local iteration method (LIM) \cite{Zhukov2010}. This scheme has a stable time step of order $\mathcal{O}(P^2)$ ($P$ represents the stencil size), thus alleviates the stiffness of the explicit implementation. It has advantages in computation efficiency and parallel scalability than the implicit schemes under not too big parabolic Courant number (less than $10^4$), see  \cite{zhukov2018,zhukov2018development}.

\subsection{Preservation of the pressure-velocity-temperature equilibrium}
 
To preserve the  pressure-velocity-temperature equilibrium for in the pure translation of an isolated interface, two different mixture EOSs are used in \cite{Alahyari2015,Johnsen2012}. However, this may results in  incompatibility with the second law of thermodynamics, since one can not define a mixture entropy due to the ambiguity in the EOS definition. In their approach only one temperature is involved, meaning that the temperature equilibrium is always reached.  The equilibrium temperature can be regarded as a specific average of the phase temperatures.

For thermodynamical considerations, we have explicitly introduced the temperature relaxation mechanism to reach the temperature equilibrium. Our approach ensures the entropy production with one uniquely defined mixture EOS. Moreover, it maintains the pressure-velocity-temperature equilibrium in the pure translation of an isolated material interface.

Let us consider the following Riemann problem:
\begin{align}\label{eq:initial1}
u^{L} =u^{R}=u > 0, \;\;
\rho_{k}^{L} =\rho_{k}^{R}=\rho_{k},  \;\;
e_{k}^{L} = e_{k}^{R} = e_{k}, \;\;
{\alpha_2}^{L} \neq {\alpha_2}^{R}, \;\;  p_{L} = p_{R} = p, \;\;  T_{L} = T_{R} = T, \;\; k=1,2.
\end{align}

The phase temperatures on both sides of the interface are in equilibrium at the initial moment, i.e.,
\begin{equation}\label{eq:initial2}
T_{1,L} = T_{2,L} = T_{L}, \; T_{1,R} = T_{2,R} = T_{R}.
\end{equation}

\begin{proposition}
The solution to the proposed model equations (in the absence of mass diffusion) maintains the pressure-velocity-temperature equilibrium with the initial discontinuity  (\ref{eq:initial1}) and (\ref{eq:initial2}).
\end{proposition}

\begin{proof}\label{proof:1}
The solutions are obtained in the framework of the Godunov FVM. We use a Riemann solver that restores the isolated contact discontinuity. We shall check the variables in the cell downstream the given discontinuity after a time step which is denoted as $\vc{U}^{*}$.

We have $\mathcal{F}_{k,hd}^{(1)} = 0$ since the velocity is uniform across the computational domain. In the absence of diffusions or external energy source, the model (\ref{eq:final_model}) is equivalent to that of \cite{allaire2002five}. Therefore, we directly use the  proved theorem in \cite{allaire2002five}: 
\begin{equation}
u^{*} = u, \quad
\rho_{k}^{*} =\rho_{k}, \quad
e_{k}^{*} =e_{k}, \quad
p^{*} =p.
\end{equation}
According to the EOS, we have $T_k^{*} = T_k (\rho_k^{*}, p^{*})$, which leads to 
\begin{equation}
T_k^{*} = T.
\end{equation}
Thus, the pressure-velocity-temperature equilibrium is maintained at the hydrodynamic stage.

%The temperature relaxation procedure \cref{eq:tr_temp_av} can be viewed as a specific averaging of the phase temperatures and thus does not alter the equilibrium temperature. Further, according to \cref{eq:rhok_fun}, the phase densities $\rho_k$ and volume fractions $\alpha_k$ remain unchanged during the temperature relaxation since the temperature and partial densities do not vary. As a function of temperature and phase density, the pressure also does not vary at this stage.

Formally, the temperature relaxation process described in \cref{eq:tr_temp_av} can be treated as an averaging procedure of phase temperatures, thus, the equilibrium temperature remains at this stage. Moreover, the partial densities $m_k$ are also constants in the course of temperature relaxation. By using the relations \cref{eq:rhok_fun}, one can conclude that phase densities $\rho_k$ and volume fractions $\alpha_k$ also remain constant. According to the EOS, the pressure does not change, either.

Due to the uniformity of velocity and temperature in space, the viscous dissipation and heat conduction do not alter the above solution. 

To summarize, the pressure-velocity-temperature equilibrium is maintained.
\end{proof}

%For example, to ensure $\rho_{k}^{*}=\rho_{k}$, the reconstruction schemes for $\alpha_k \rho_k$ and $\alpha_k$ should be linearly dependent.

%\begin{remark}
%To maintain the temperature equilibrium of the numerical solutions, the numerical schemes applied should also ensure all the conditions used in \Cref{proof:1}, especially in high-order extensions. 
%\end{remark}

\section{Numerical results}
\label{sec:numer_res}

In this section, we perform several numerical tests to verify the proposed model and numerical methods. Without mentioning, we use the following default setting: (a) CFL = 0.2, (b) fifth-order WENO scheme to reconstruct the characteristic variables that are linearized on the cell interface \cite{Coralic2014Finite}, (c) the two-stage Heun method for time integration. The considered tests demonstrate the effect of temperature relaxation, mass diffusion, viscosity and heat conduction. Some of the numerical results with heat conduction of the proposed model is compared to those of the conservative four-equation model \cite{Cook2009Enthalpy,larrouturou1991preserve}, demonstrating the ability of the present model to model miscible interface problems without spurious oscillations. We also apply the model for simulating the laser ablative RM instability problem in the ICF field. The results obtained with both models demonstrate noticeable difference in flow structures.
%In the following numerical results, we compare the results of three models : (1) the fully conservative four-equation model, (2) the reduced model, (3) the Kapila's model (i.e. the reduced model without temperature relaxation). 

\subsection{The pure transport problem}\label{subsec:pureTRS}
We first consider the pure translation of a smeared mass fraction profile, which is a mimic of an miscible interface. The pressure, temperature and velocity are uniformly $1\times 10^4$Pa, $5$K and 100m/s across the computational domain, respectively.   The two gases are characterized by the IG EOS with $\rho_1 = 20\text{kg}/\text{m}^3, \gamma_1 = 2.0$  and $\rho_2 = 1\text{kg}/\text{m}^3, \gamma_2 = 1.4$, respectively. The parameter $C_{v,k}$ should be chosen such that the initial phase temperatures are equilibrium, i.e. $(\gamma_1 - 1) \rho_1 C_{v,1} = (\gamma_2 - 1) \rho_2 C_{v,2}$. The initial mass fraction is smeared as \cite{Thornber2018,Kokkinakis2015}:
\begin{subequations}\label{eq:thornber_analy}
\begin{align}
    {\rho}=\frac{1}{2}\left({\rho}_{1}+{\rho}_{2}\right) - \frac{1}{2}\left({\rho}_{1}-{\rho}_{2}\right) \operatorname{erf}(z), \\
   {\rho Y_1}=\frac{1}{2}{\rho}_{1} - \frac{1}{2}{\rho}_{1} \operatorname{erf}(z), \\
    z=\frac{x  - x_0}{\sqrt{4 D t+h_{0}^{2}}},
\end{align}
\end{subequations}
where $h_0 = 0.02$m is the initial interface diffuseness, $x_0$ is the center of the computational domain $x\in \left[ 0\text{m}, 1\text{m} \right]$, $D = 0.01\text{m}^2$/s is the diffusivity.

Periodical boundary conditions are imposed on both sides of the computational domain. After $\Delta t = 5 / u$ (five periods) the solutions should return to the initial state.  The numerical results obtained with  the reduced five equation model are demonstrated in \Cref{fig:pureTRS}. It can be seen that the pressure and temperature equilibrium is well preserved.

%Note that the numerical results of the four-equation model is obtained by reconstructing the variables $ [ \rho \; u \; p \; y ]$ for preserving positivity of $y$ \cite{larrouturou1991preserve}. It can be seen that the four-equation model triggers noticeable non-physical  oscillations in pressure, velocity and temperature, while the reduced model is free of this problem.
%

\begin{figure}[ht]
\centering
\subfloat[Velocity]{\label{fig:pureTRS:vel}\includegraphics[width=0.5\textwidth]{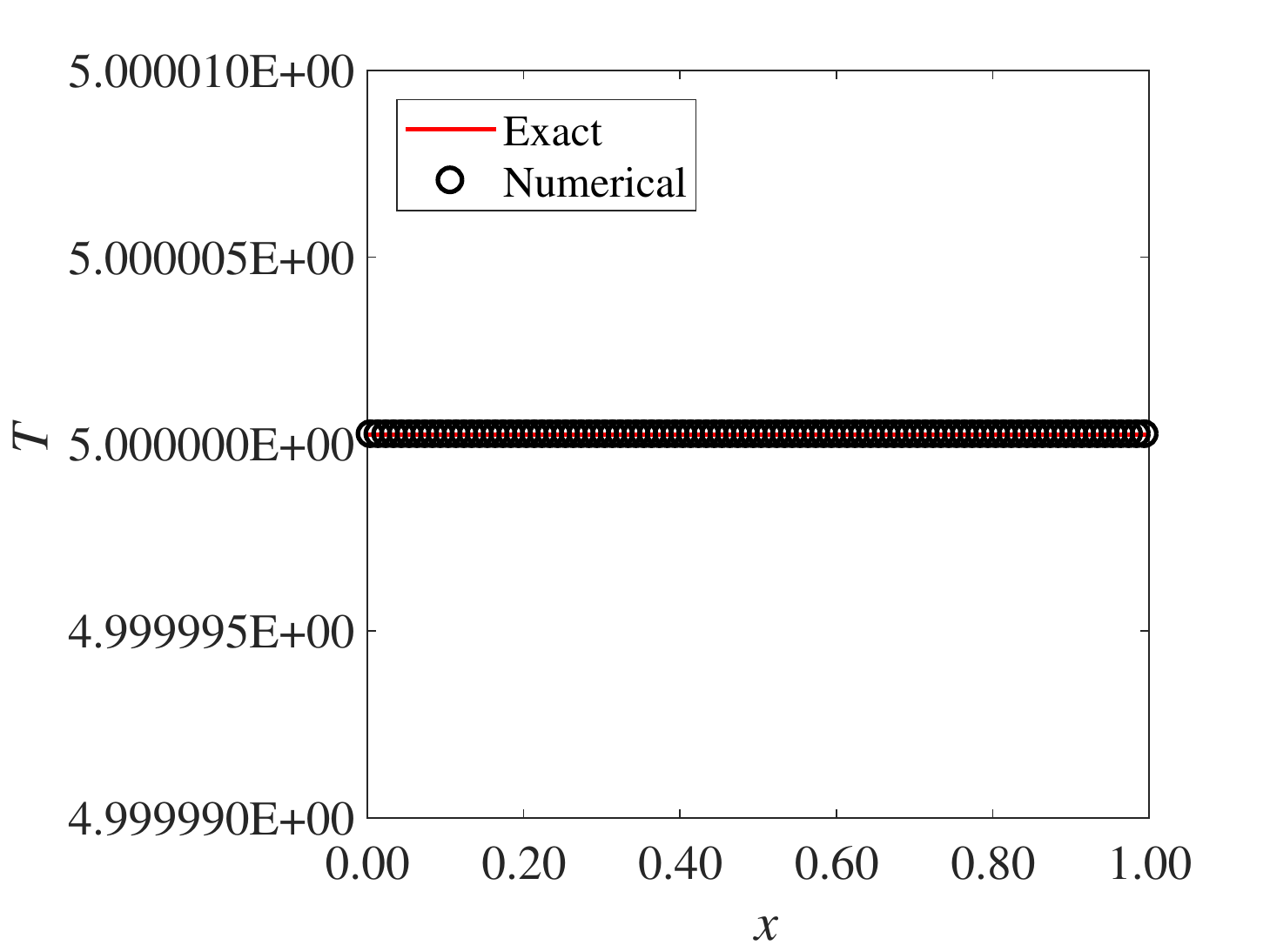}}
\subfloat[Velocity, locally enlarged]{\label{fig:pureTRS:vel1}\includegraphics[width=0.5\textwidth]{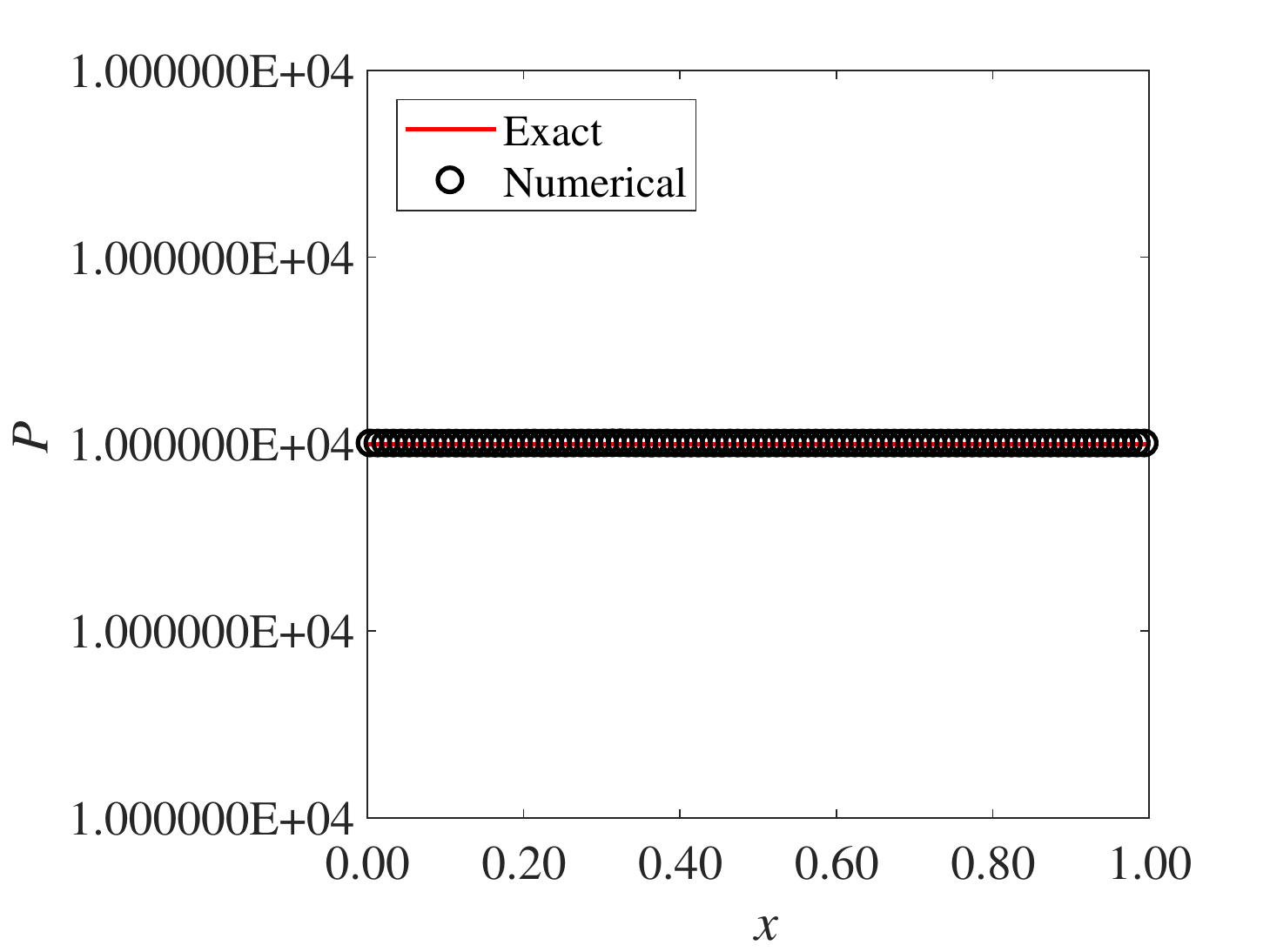}}
\caption{Numerical results for the pure translation of a diffused interface.}
\label{fig:pureTRS} 
\end{figure}

\subsection{The pure mass diffusion problem}\label{subsec:pureMD}
To verify the mass diffusion part of the model, we consider a problem with small diffusion velocity in comparison with the sound speed. In this case, the mass diffusion problem can be approximately treated as an incompressible one. 
The densities, adiabatic coefficients and specific heat capacities are the same as those in the last test.

%The two gases are characterized by the IG EOS with $\rho_1 = 20\text{kg}/\text{m}^3, \gamma_1 = 2.0$  and $\rho_2 = 1\text{kg}/\text{m}^3, \gamma_2 = 1.4$, respectively. The parameter $C_{v,k}$ should be chosen such that the initial phase temperatures are equilibrium, i.e. $(\gamma_1 - 1) \rho_1 C_{v,1} = (\gamma_2 - 1) \rho_2 C_{v,2}$. The initial mass fraction is smeared as \cite{Thornber2018,Kokkinakis2015}:
%\begin{subequations}\label{eq:thornber_analy}
%\begin{align}
%    {\rho}=\frac{1}{2}\left({\rho}_{1}+{\rho}_{2}\right) - \frac{1}{2}\left({\rho}_{1}-{\rho}_{2}\right) \operatorname{erf}(z), \\
%    {\rho Y_1}=\frac{1}{2}{\rho}_{1} - \frac{1}{2}{\rho}_{1} \operatorname{erf}(z), \\
%    z=\frac{x  - x_0}{\sqrt{4 D t+h_{0}^{2}}},
%\end{align}
%\end{subequations}
%where $h_0 = 0.02$m is the initial interface diffuseness, $x_0$ is the center of the computational domain $x\in \left[ 0\text{m}, 1\text{m} \right]$, $D = 0.01\text{m}^2$/s is the diffusivity.

It is assumed that mass diffusion goes at very small scale that the pressure and temperature are nearly uniform. In the incompressible limit of such a problem the governing equations are reduced to  \cite{Livescu2013,Kokkinakis2015,Thornber2018}:
\begin{equation}\label{eq:analytic_diff}
\frac{\partial {\rho}}{\partial t}=\frac{\partial}{\partial x}\left(D \frac{\partial {\rho}}{\partial x}\right), \quad \frac{\partial}{\partial x}\left({u}+\frac{D}{{\rho}} \frac{\partial {\rho}}{\partial x}\right)=0.
\end{equation}

%In the computational domain $•[0\text{m},1\text{m}]$, two different gases are initially separated by a diffused interface with a width $h_0 = 0.02\text{m}$ and center at $x_0 = 0.5\text{m}$. 
  With the zero-gradient boundary condition for the density, the solution to  \cref{eq:analytic_diff} is given by \cref{eq:thornber_analy}. To investigate the convergence performance of the models, the initial conditions are given with the integration of \cref{eq:thornber_analy}  and averaging within each cell \cite{Thornber2018}.  The initial pressure is uniformly set to be $1\times10^5\text{Pa}$, which results in small enough Mach number so that the compressibility effect can be neglected. Computations are performed on a series of refining grid. The corresponding numerical results are demonstrated in \Cref{fig:figMD}. It can be seen that the numerical results tend to converge to the analytical solution (in the incompressible limit) with the grid refinement with the second-order MUSCL for spatial reconstruction in the hydrodynamic stage. The convergence order of the numerical algorithm is approximately 2, as demonstrated in \Cref{figMD1:conv_rate}. We also compare the temperature error with the results in \cite{Thornber2018} (\Cref{figMD1:T_err}).

\begin{figure}[htbp]
\centering
\subfloat[Velocity]{\label{figMD:vel}\includegraphics[width=0.5\textwidth]{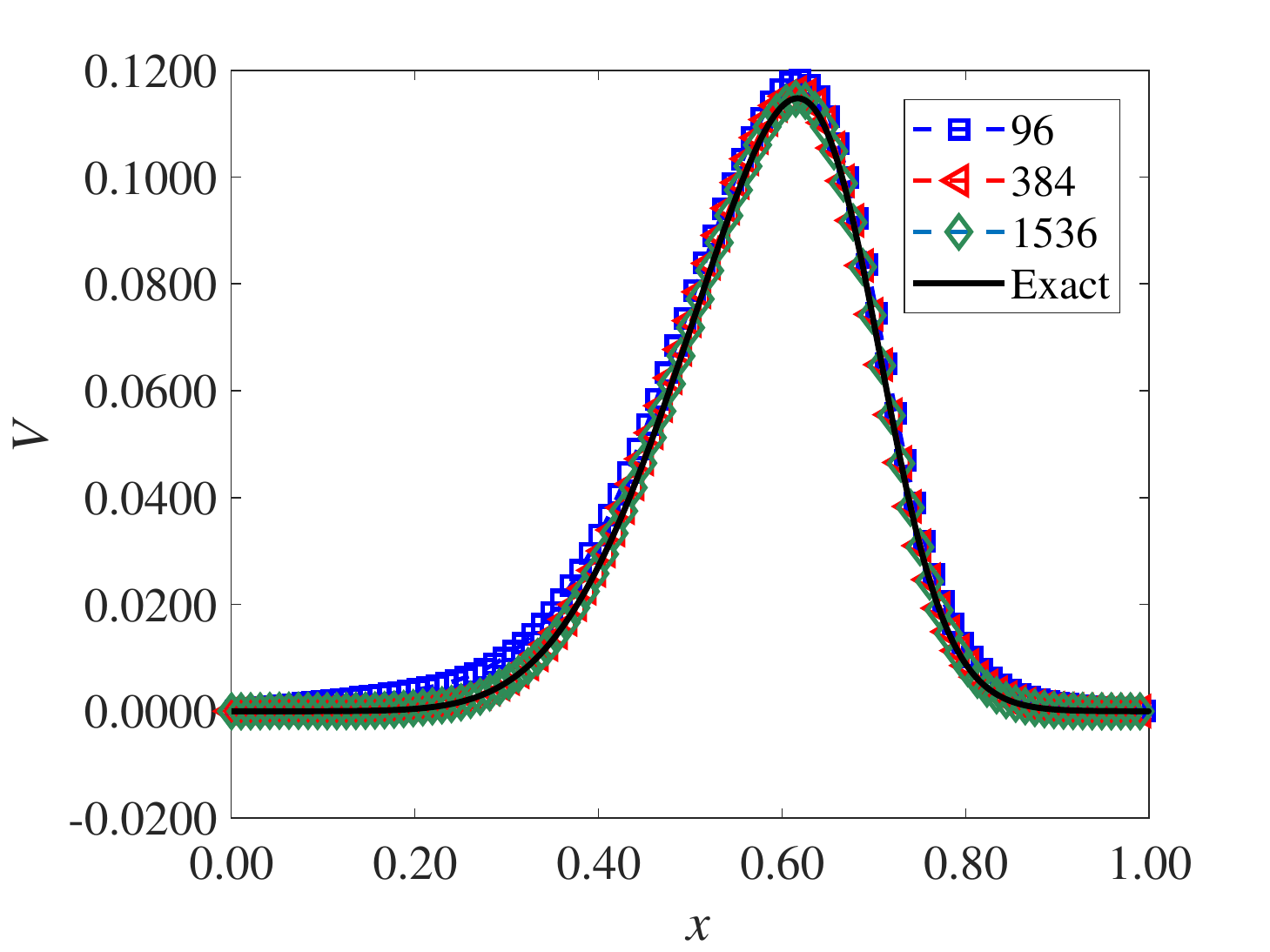}}
\subfloat[Velocity, locally enlarged]{\label{figMD:vel1}\includegraphics[width=0.5\textwidth]{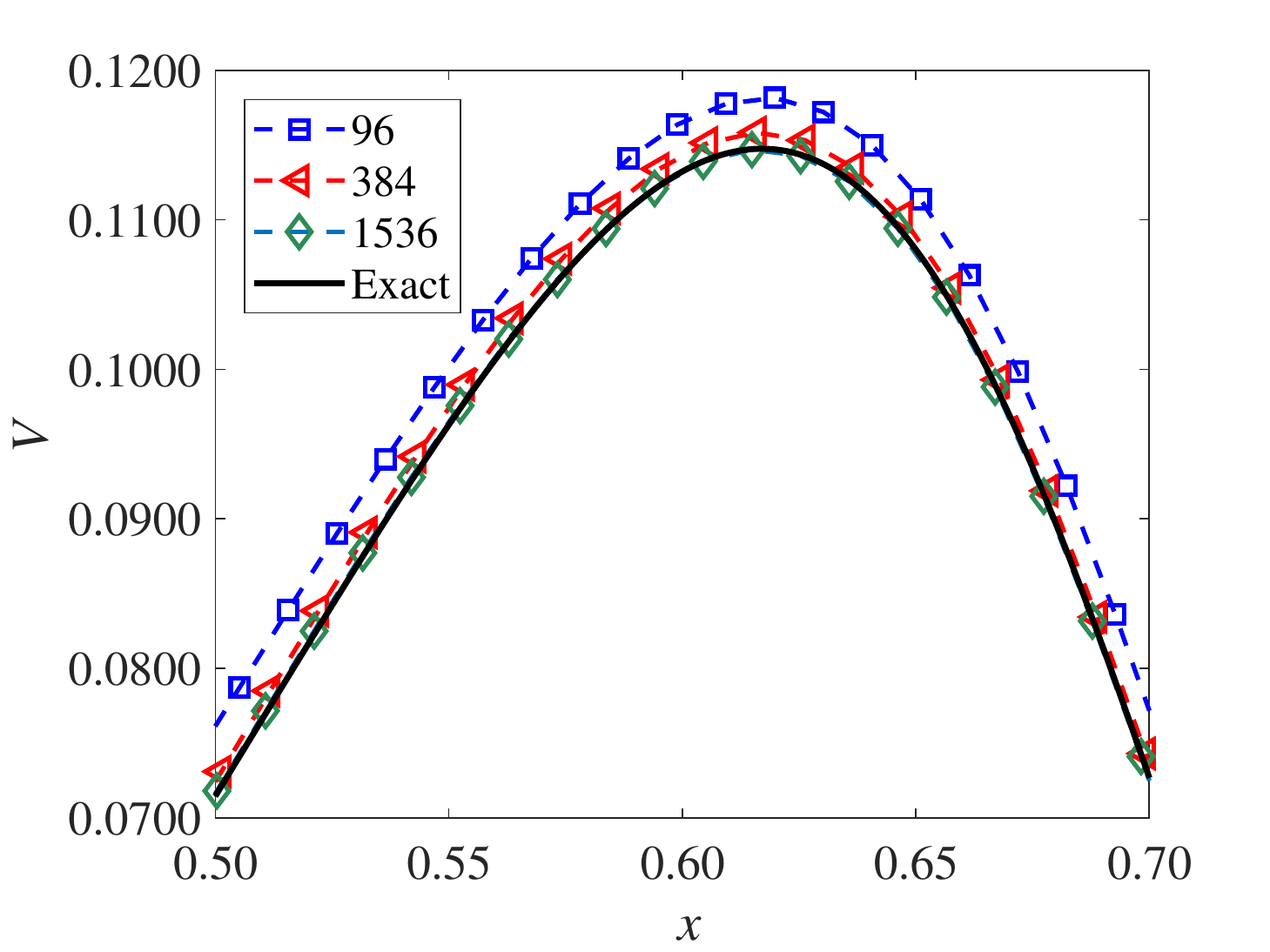}}\\
\subfloat[Density]{\label{figMD:pres}\includegraphics[width=0.5\textwidth]{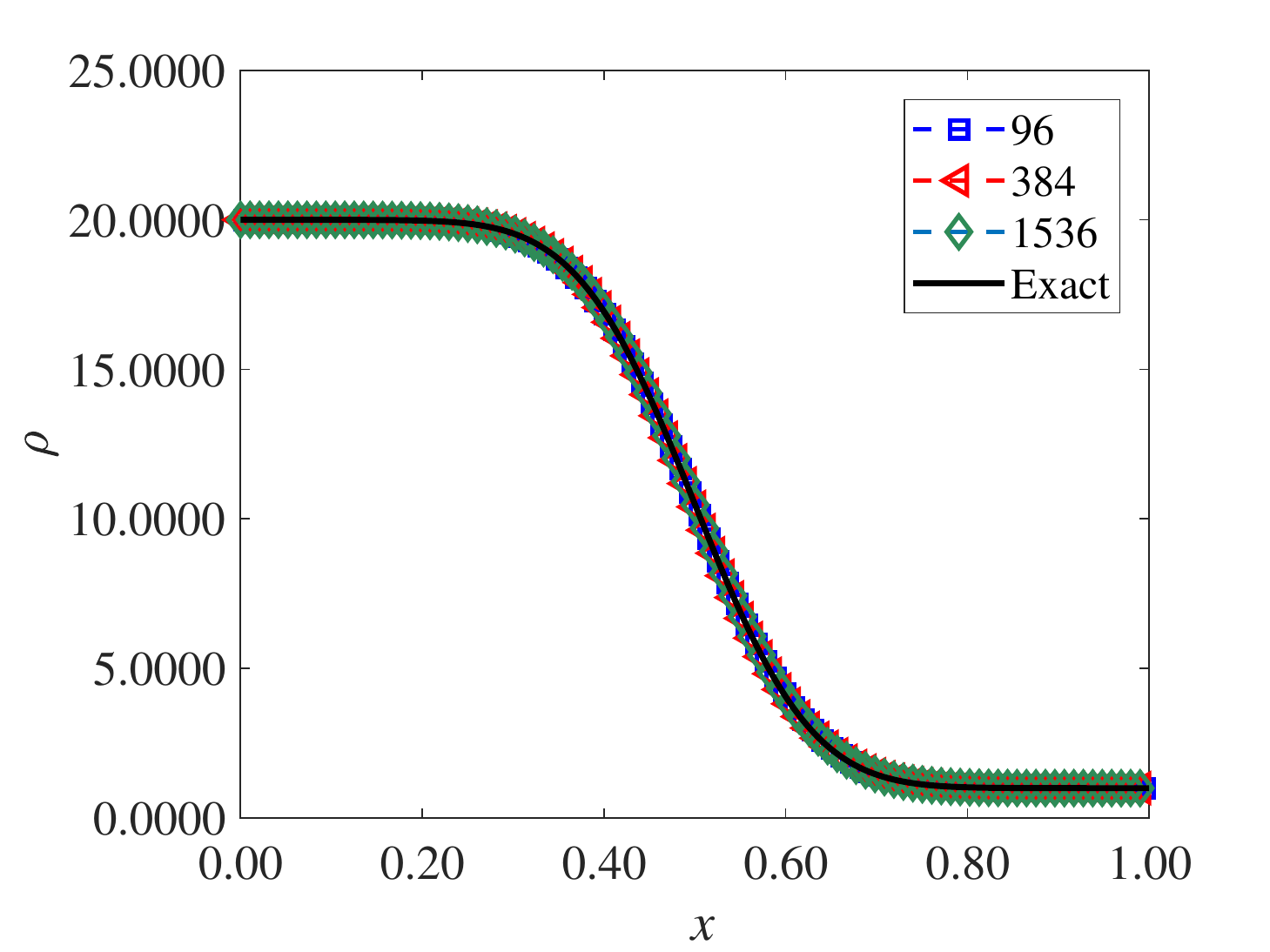}}
\subfloat[Density, locally enlarged]{\label{figMD:temp}\includegraphics[width=0.5\textwidth]{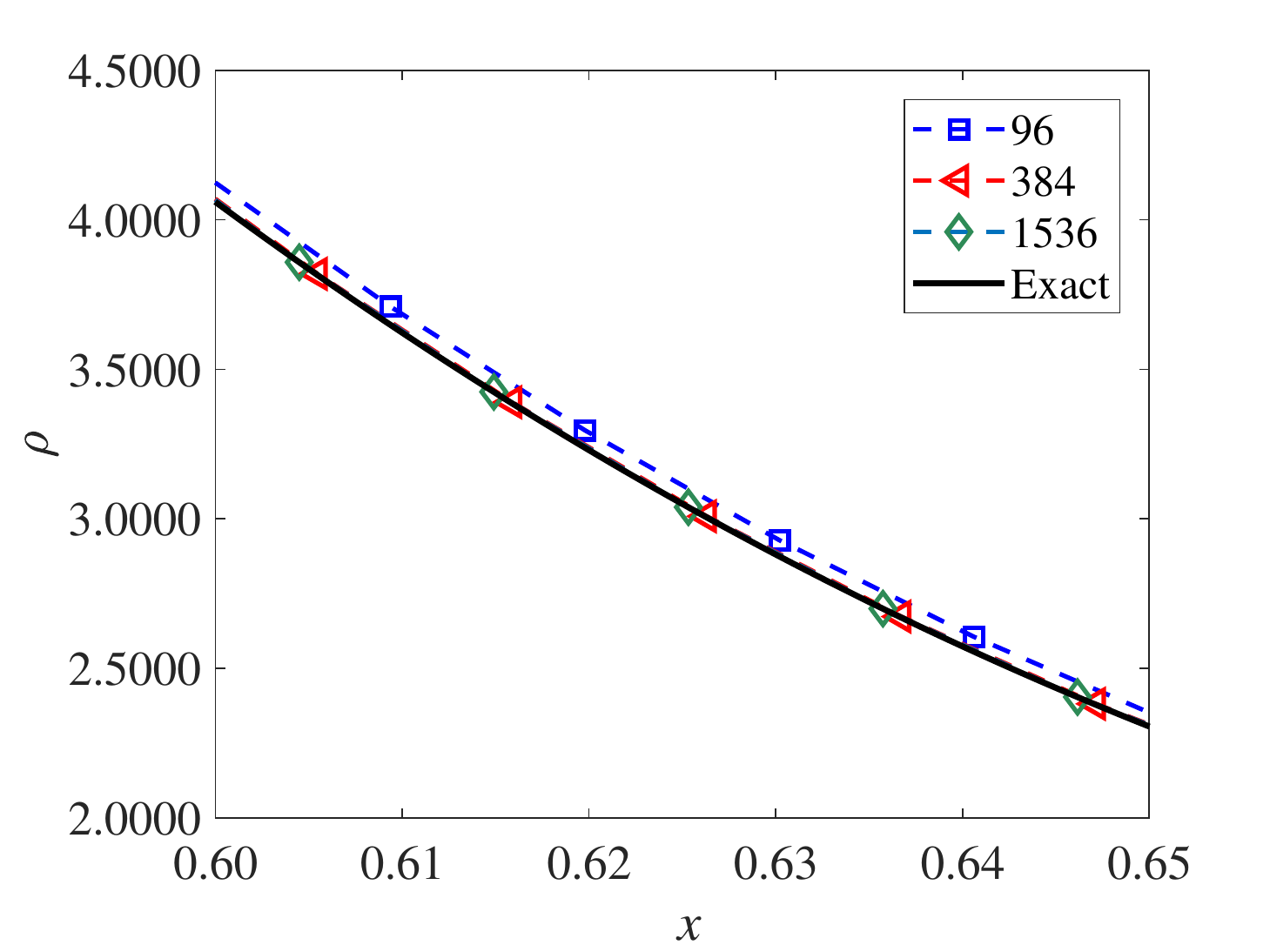}}
\caption{Numerical results for the pure mass diffusion problem.}
\label{fig:figMD} 
\end{figure}

\begin{figure}[htbp]
\centering
\subfloat[Accuracy order]{\label{figMD1:conv_rate}\includegraphics[width=0.5\textwidth]{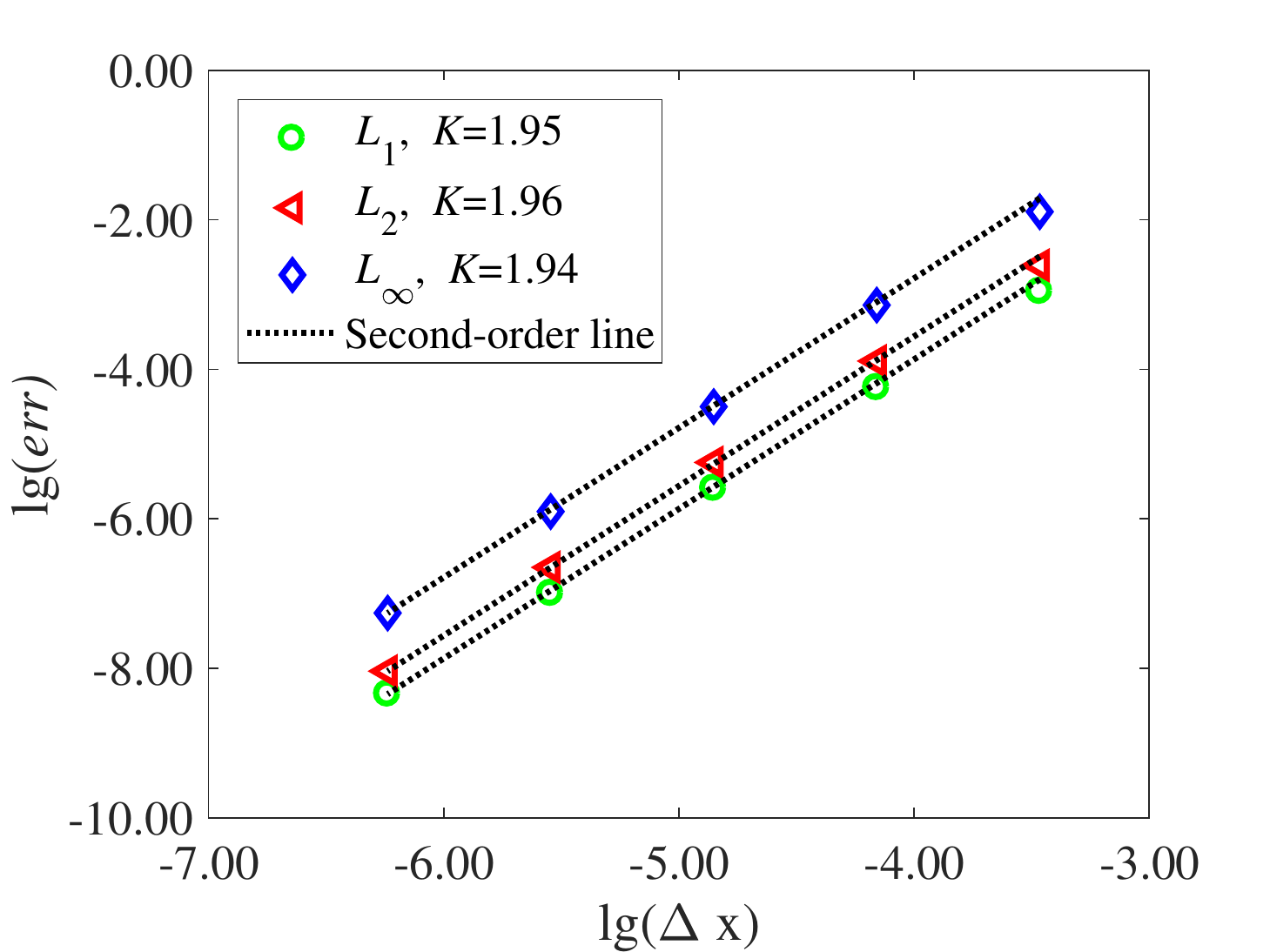}}
\subfloat[Temperature error]{\label{figMD1:T_err}\includegraphics[width=0.5\textwidth]{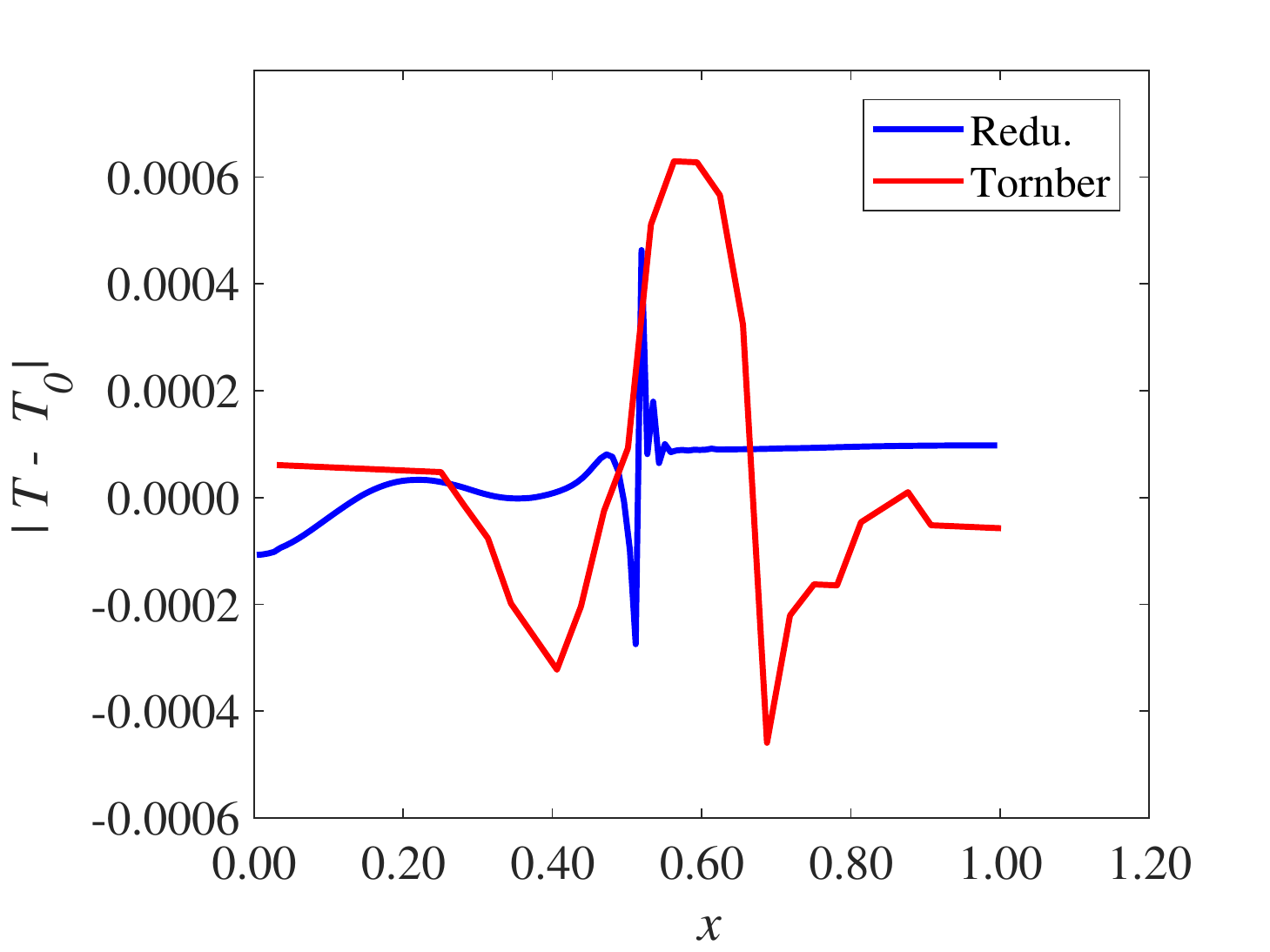}}
\caption{Numerical results of the reduced model for the pure mass diffusion problem. Left: convergence rate. Right: temperature error on 128-cell grid when $P = 1\times 10^4$.}
\label{fig:figMD1} 
\end{figure}

\subsection{The convergence of the viscous part}
Having validated the convergence of the mass diffusion, we then check the convergence performance of the viscous part. For such purpose, we manufacture an exact solution as follows
\begin{equation}
\left\{\begin{array}{l}
(\alpha_k \rho_k)(x,t) = 0.5 \text{kg}/\text{m}^3, \\
\alpha_1(x,t) = 0.7, \\
p(x,t) = 100 \text{Pa}, \\
u(x,t) = \frac{1+(2 a-1) \exp ((1-a) \xi / v)}{1+\exp ((1-a) \xi / v)}, \quad \xi=x-a t-x_{0}.
\end{array}\right.
\end{equation}

Note the the manufactured solution for $u(x,t)$ is the analytical solution to the viscous Burgers equation \[\dudx{u}{t} + \frac{\partial}{\partial x} \frac{u^2}{2} = \nu \frac{\partial^2 u}{\partial x^2}.\]

The properties for the materials are given as  $\gamma_1 = 5.0$, $C_{v,1} = 40.0 \text{J}/(\text{kg} \cdot \text{K})$,  $\gamma_1 = 1.4$, $C_{v,1} = 400.0 \text{J}/(\text{kg} \cdot \text{K})$ so that the initial temperature equilibrium is satisfied.
We use the constants $a = 0.4$, $x_0 = 0.1\text{m}$ and $\nu = 0.01\text{}$, the exact solution at $t = 0.05\text{s}$ is taken as the initial value. We perform computation to $t = 0.15\text{s}$. The numerical results tend to converge to the exact solution with grid refinement (\Cref{figvis:2}). The dependence of the error on the spatial resolution is demonstrated in \Cref{figvis:1}, demonstrating the accuracy order is approximately second order.

\begin{figure}[htbp]
\centering
\subfloat[Accuracy order]{\label{figvis:1}\includegraphics[width=0.5\textwidth]{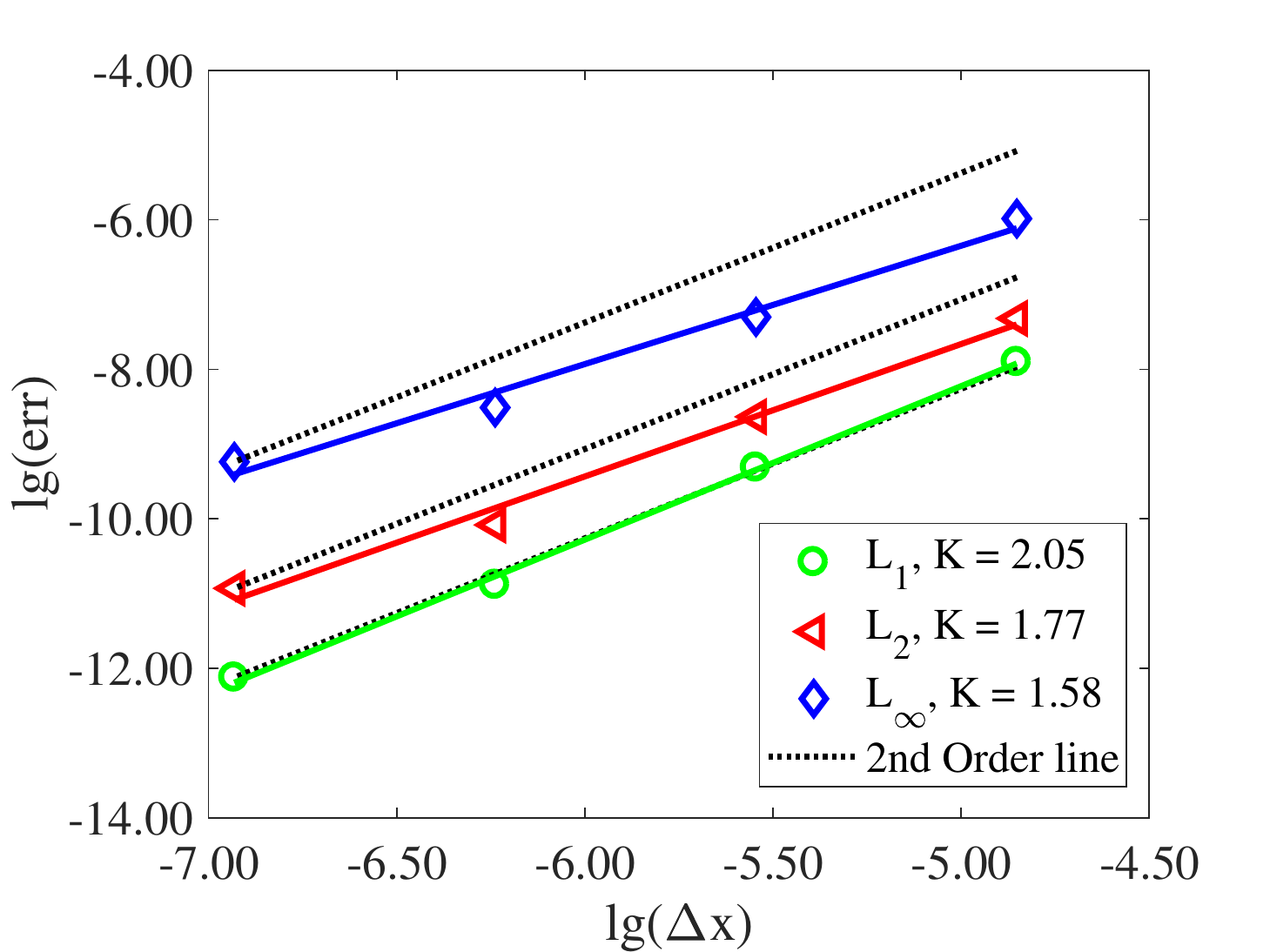}}
\subfloat[Velocity on different grids]{\label{figvis:2}\includegraphics[width=0.5\textwidth]{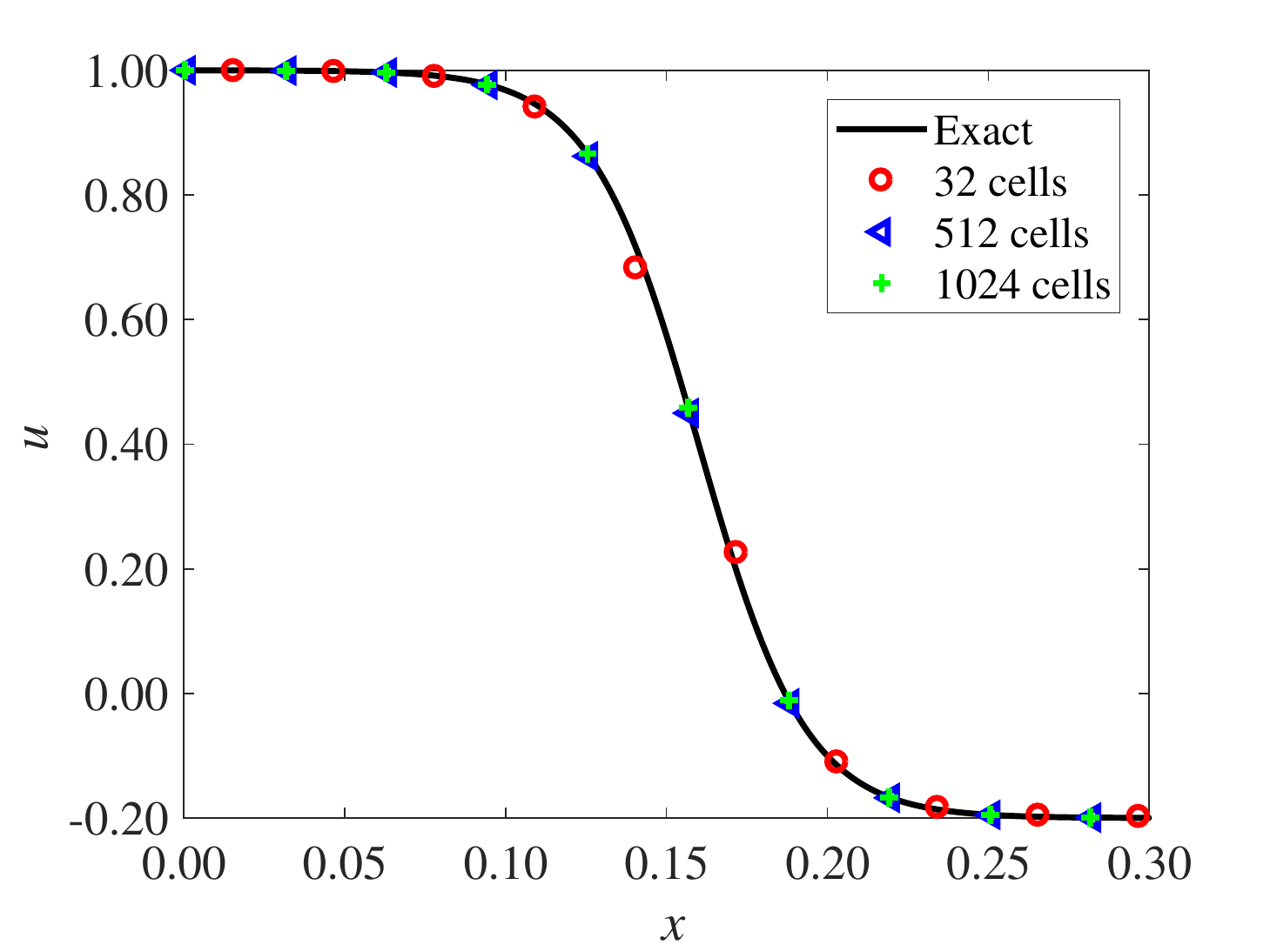}}
\caption{Convergence performance of the viscous part.}
\label{fig:visconvergence} 
\end{figure}

\subsection{The multi-component shock tube problem}\label{subsec:shock_interf_interaction}
To verify and compare different models, we consider a multicomponent shock tube problem with a resolved interface. The initial condition is given as follows
\begin{equation}
    \left(\rho, u, p, \gamma, C_v\right)= \begin{cases}\left(1000 \mathrm{~kg} \cdot \mathrm{m}^{-3}, 0 \mathrm{~m} \cdot \mathrm{s}^{-1}, 10^9 \mathrm{~Pa}, 4.4, 1606 \text{J}/(\text{kg}\cdot\text{K}) \right) & \text { for } x<0.7, \\ \left(50 \mathrm{~kg} \cdot \mathrm{m}^{-3}, 0, 10^5 \mathrm{~Pa}, 1.4, 714 \text{J}/(\text{kg}\cdot\text{K}) \right) & \text { for } 0.7 < x < 1. \end{cases}
\end{equation}

We have used two sets of grid, i.e., a coarse grid of 1200 uniform cells and a fine one of 12000. 
The numerical results obtained with different models are demonstrated in \Cref{fig:figSHOCK}. The exact solutions in this figure are those of the multi-component Euler equation without any relaxations and (numerical or physical) diffusion. Therefore, the solution of the hydrodynamic part (\cref{eq:HD}) without any relaxation agrees better with the exact solutions (\Cref{figSHOCK:dens1}). 

The numerical results of the one-temperature four-equation model and the temperature-disequilibrium model with the complete temperature relaxation deviate from the exact solution in the vicinity of the interface. Moreover, the one-temperature model introduce overshoot in temperature (\Cref{figSHOCK:temp1}).

\begin{figure}
\centering
\subfloat[density]{\label{figSHOCK:dens}\includegraphics[width=0.5\textwidth]{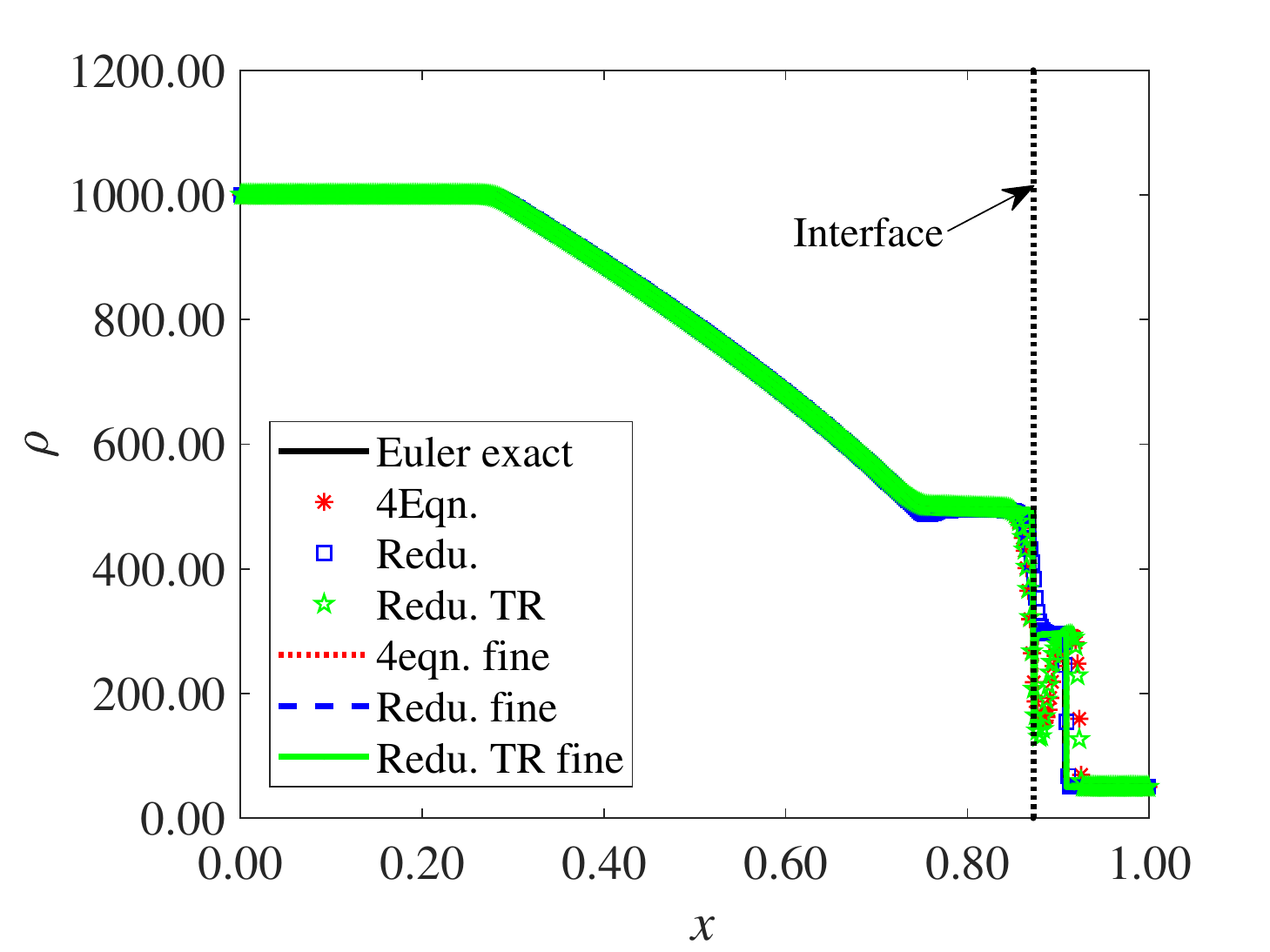}}
\subfloat[density, locally enlarged]{\label{figSHOCK:dens1}\includegraphics[width=0.5\textwidth]{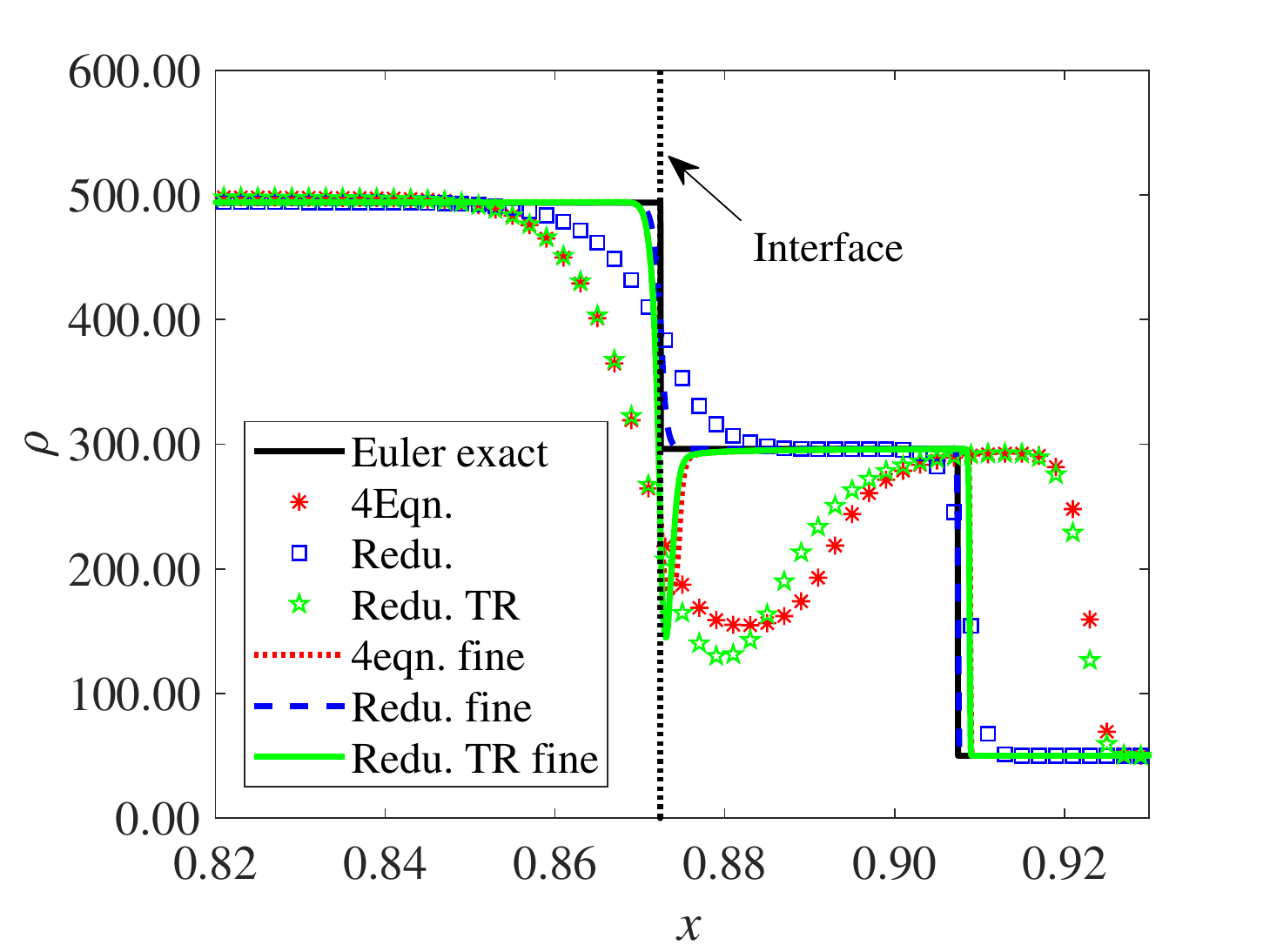}}\\
\subfloat[Temperature]{\label{figSHOCK:temp}\includegraphics[width=0.5\textwidth]{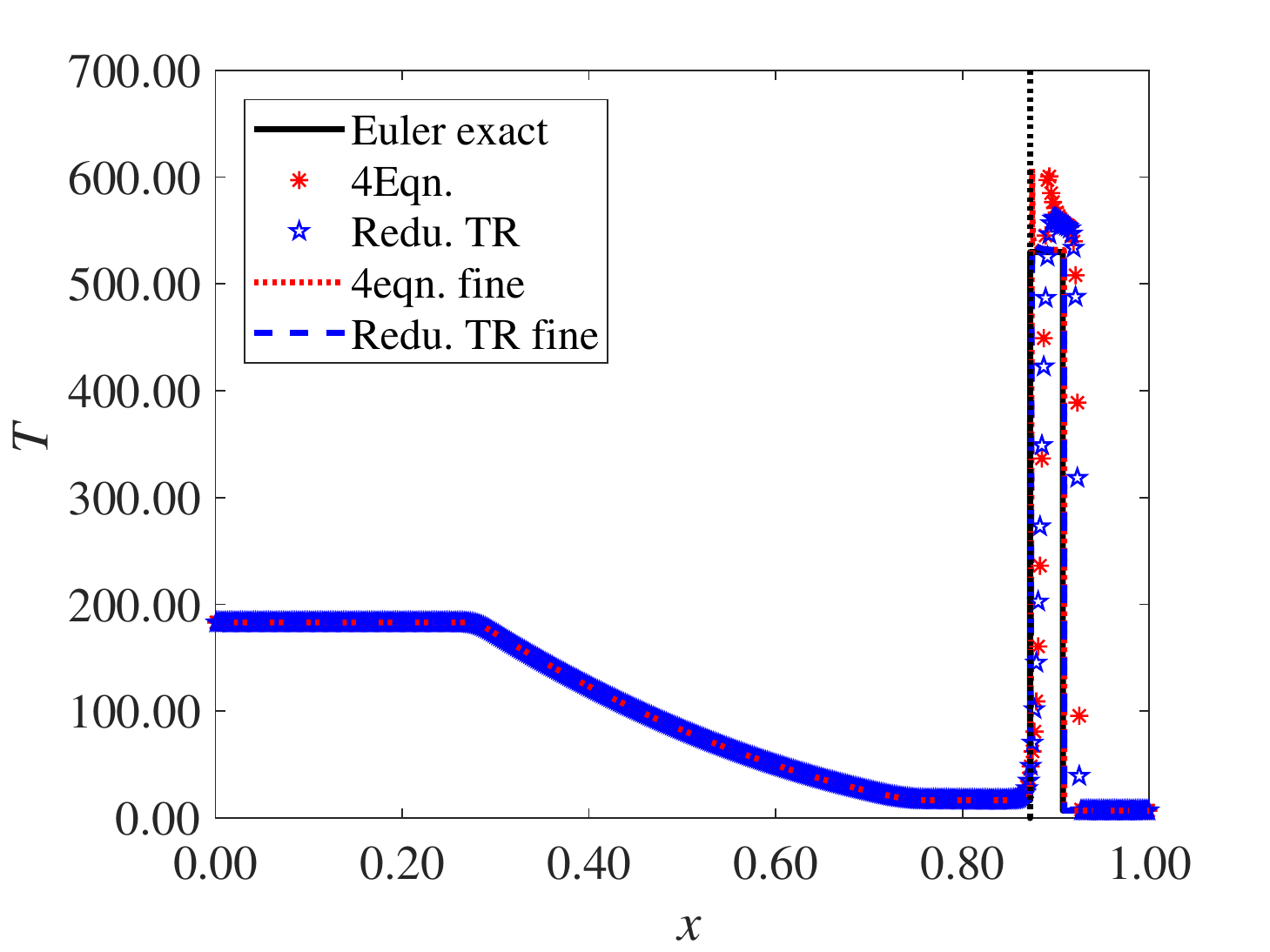}}
\subfloat[Temperature, locally enlarged]{\label{figSHOCK:temp1}\includegraphics[width=0.5\textwidth]{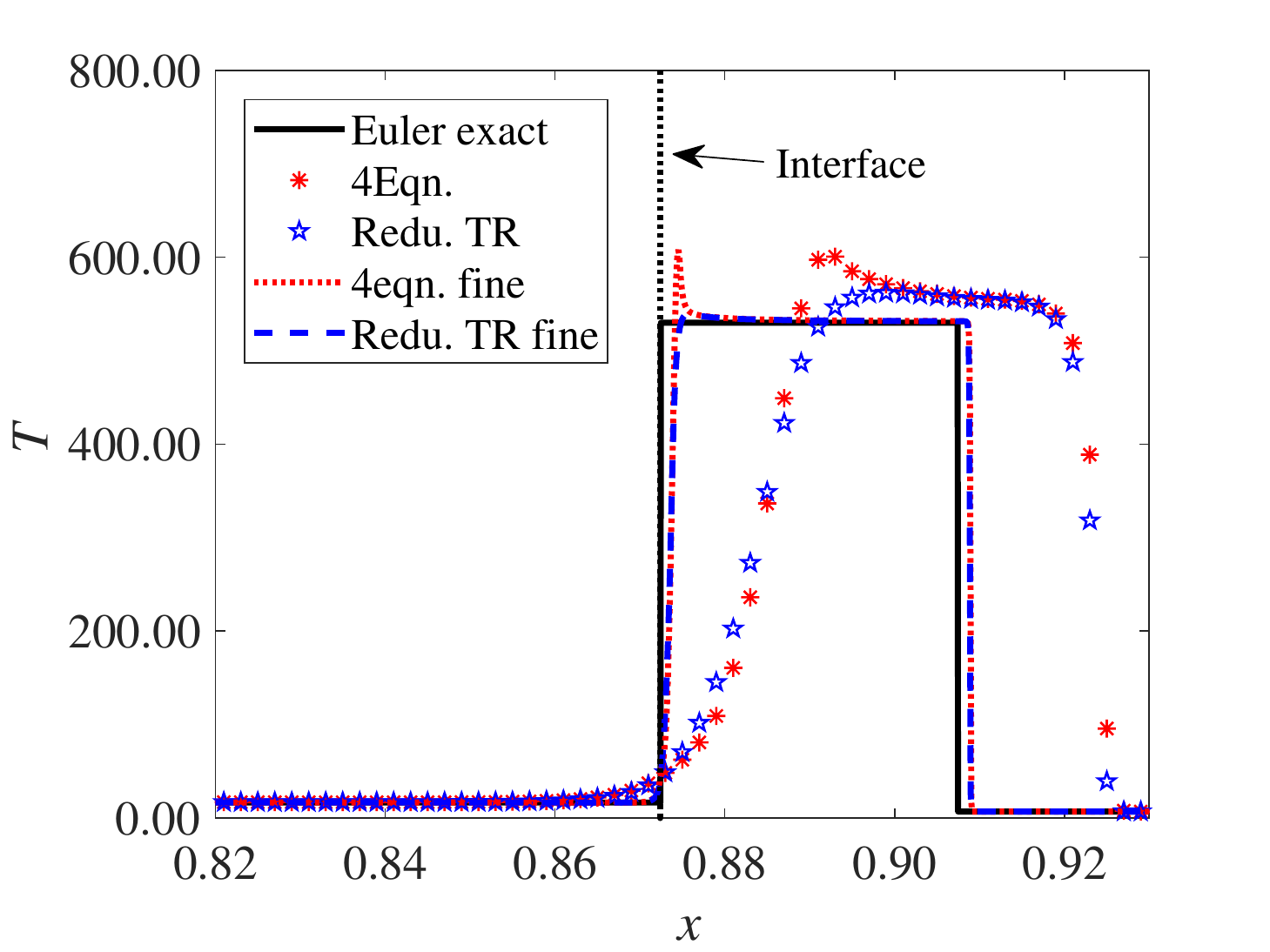}}
\caption{Numerical results obtained with different models on different grids for the shock tube problem.}
\label{fig:figSHOCK} 
\end{figure}

\subsection{Shock passage through a smeared interface}
The mass diffusion creates a physically smeared interface. The interaction between such smeared interface and the shock is experimentally investigated in literature, for example, the gas curtain experiment \cite{balakumar2008simultaneous}. The current test is a 1D analogue of this problem. The initial condition is demonstrated in \Cref{fig:figGC_initial}. A shock in air of Mach 5 impact the smeared SF$_6$ zone, within which the volume fraction is distributed as \cite{mikaelian1996numerical}
\begin{equation}
  \alpha_{\text{SF}_6} = \frac{C}{\text{exp}{\left[ (836 (x-0.02))^2\right]}}.
  \end{equation}  
Here we take the constant $C$ to be 0.95.

The incident shock transmits and reflects in its interaction with the smeared interface.  It compresses the smeared interface to form a thin spike in density. We compare the solutions obtained with the one-temperature four-equation model and the reduced model without/without the temperature relaxation. The numerical results for density, temperature and mass fraction are compared in \Cref{figGC:dens,figGC:dens1}, \Cref{figGC:temp,figGC:temp1} and \Cref{figGC:Y,figGC:Y1}, respectively. We see that the the solutions of the reduced model without the temperature relaxation deviate from those with the temperature relaxation being included implicitly (for the four-equation model) or explicitly (for the reduced model). From \Cref{figGC:temp1} one can observe obvious oscillations in the solutions of the four-equation model.

\begin{figure}[htbp]
\centering
\label{figGC:dens}\includegraphics[width=0.5\textwidth]{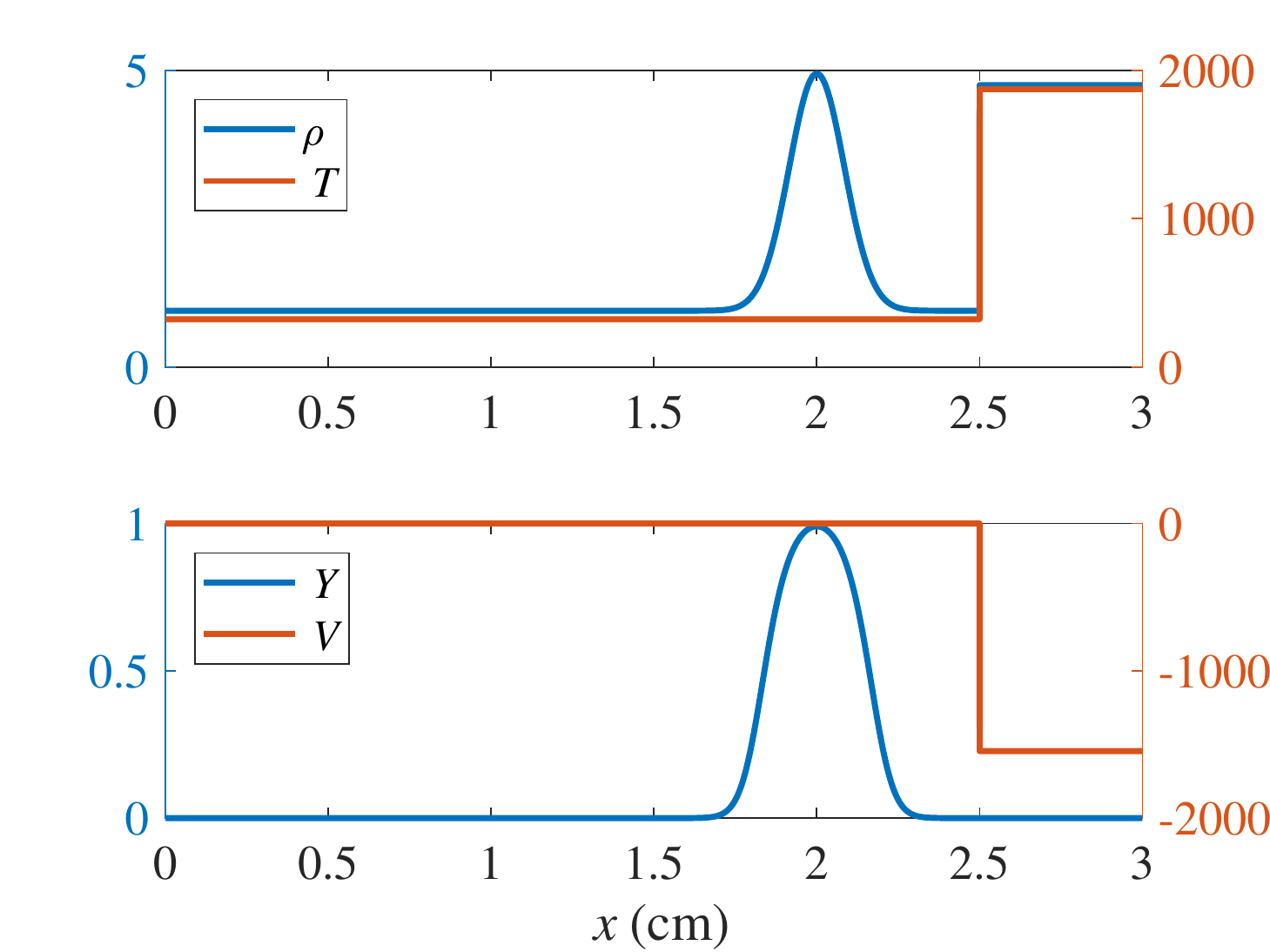}
\caption{Initial condition for the problem of shock passage through a smeared interface.}
\label{fig:figGC_initial} 
\end{figure}

\begin{figure}[htbp]
\centering
\subfloat[density]{\label{figGC:dens}\includegraphics[width=0.5\textwidth]{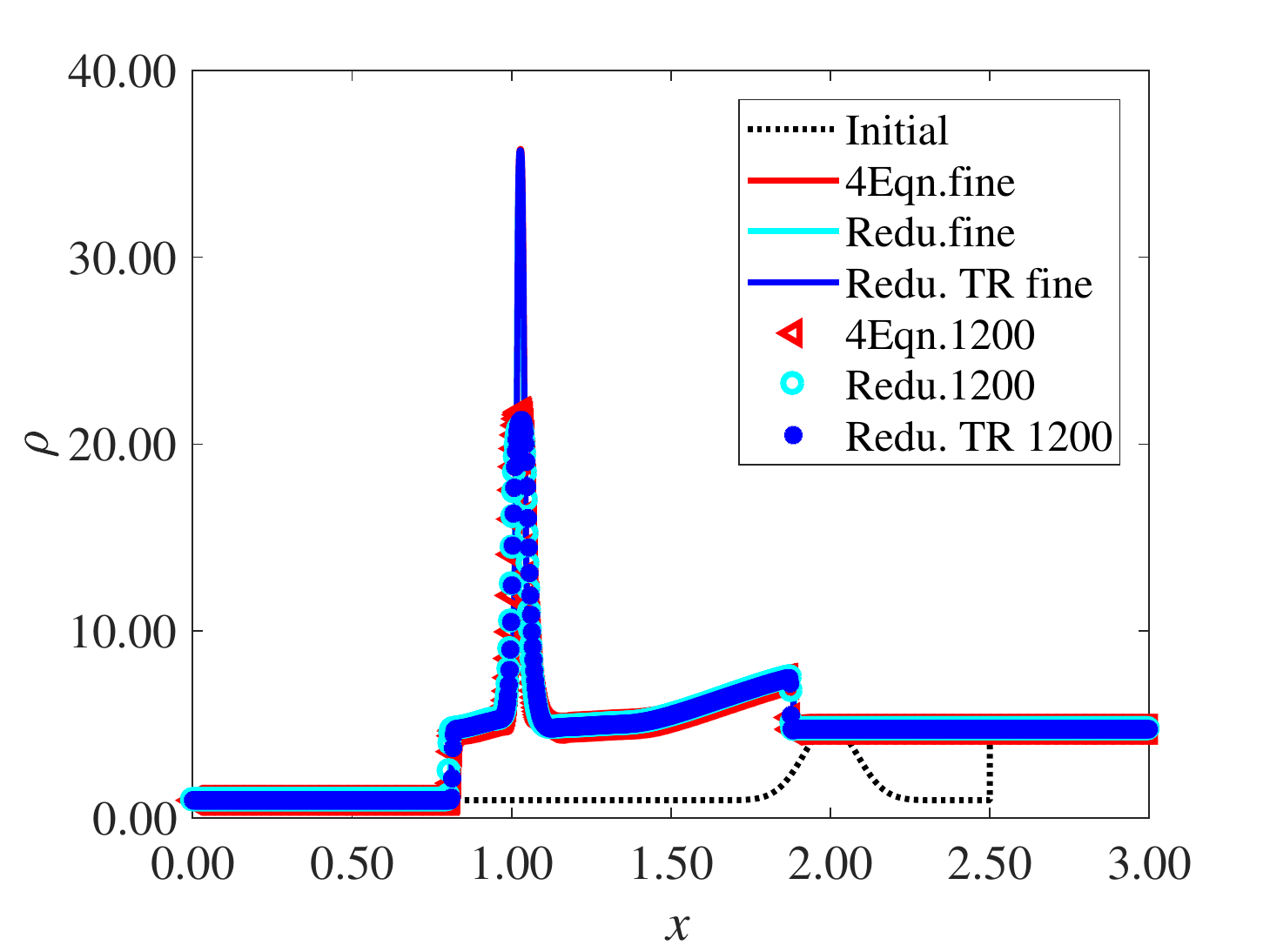}}
\subfloat[density, locally enlarged]{\label{figGC:dens1}\includegraphics[width=0.5\textwidth]{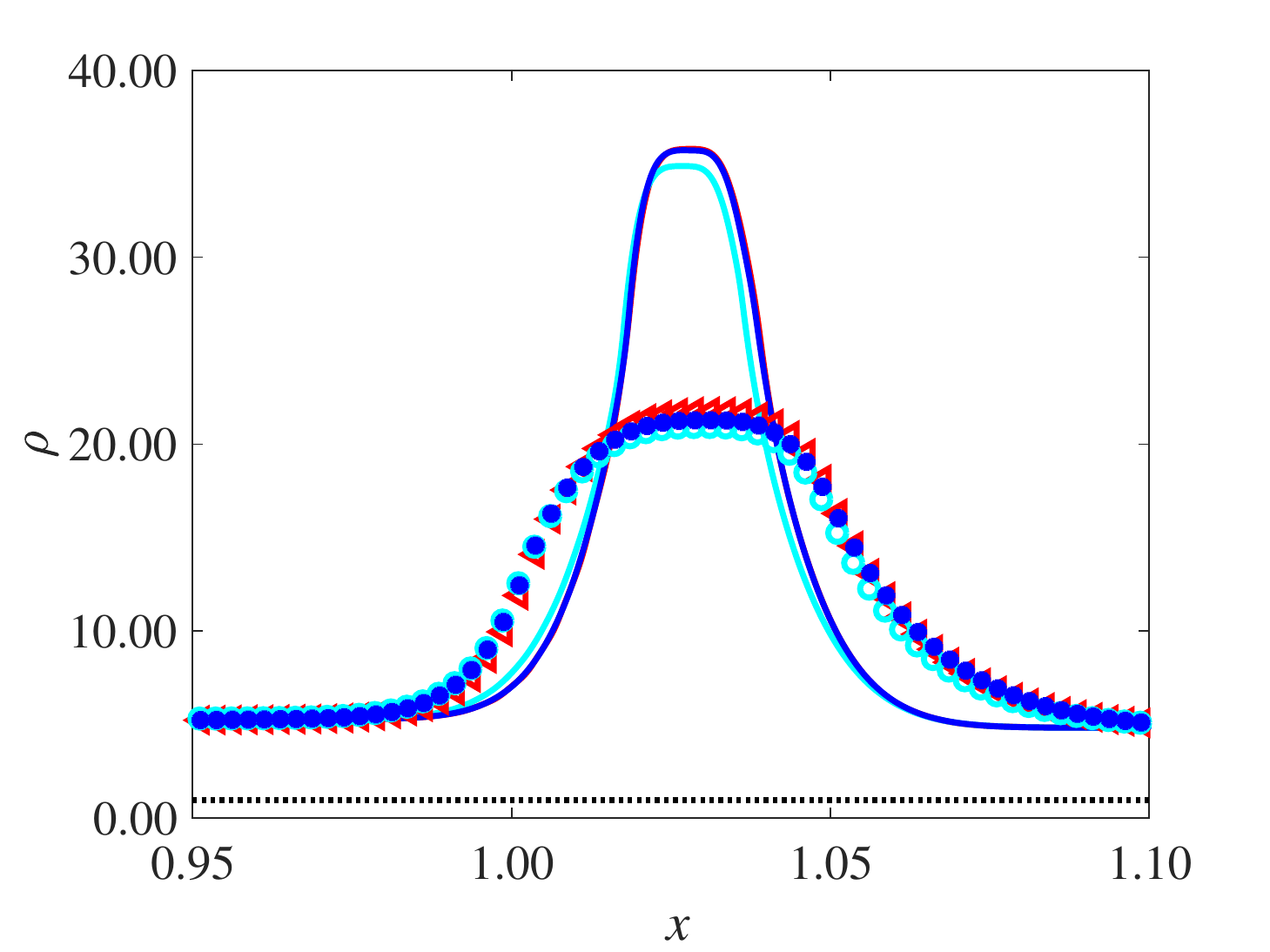}}\\
\subfloat[Temperature]{\label{figGC:temp}\includegraphics[width=0.5\textwidth]{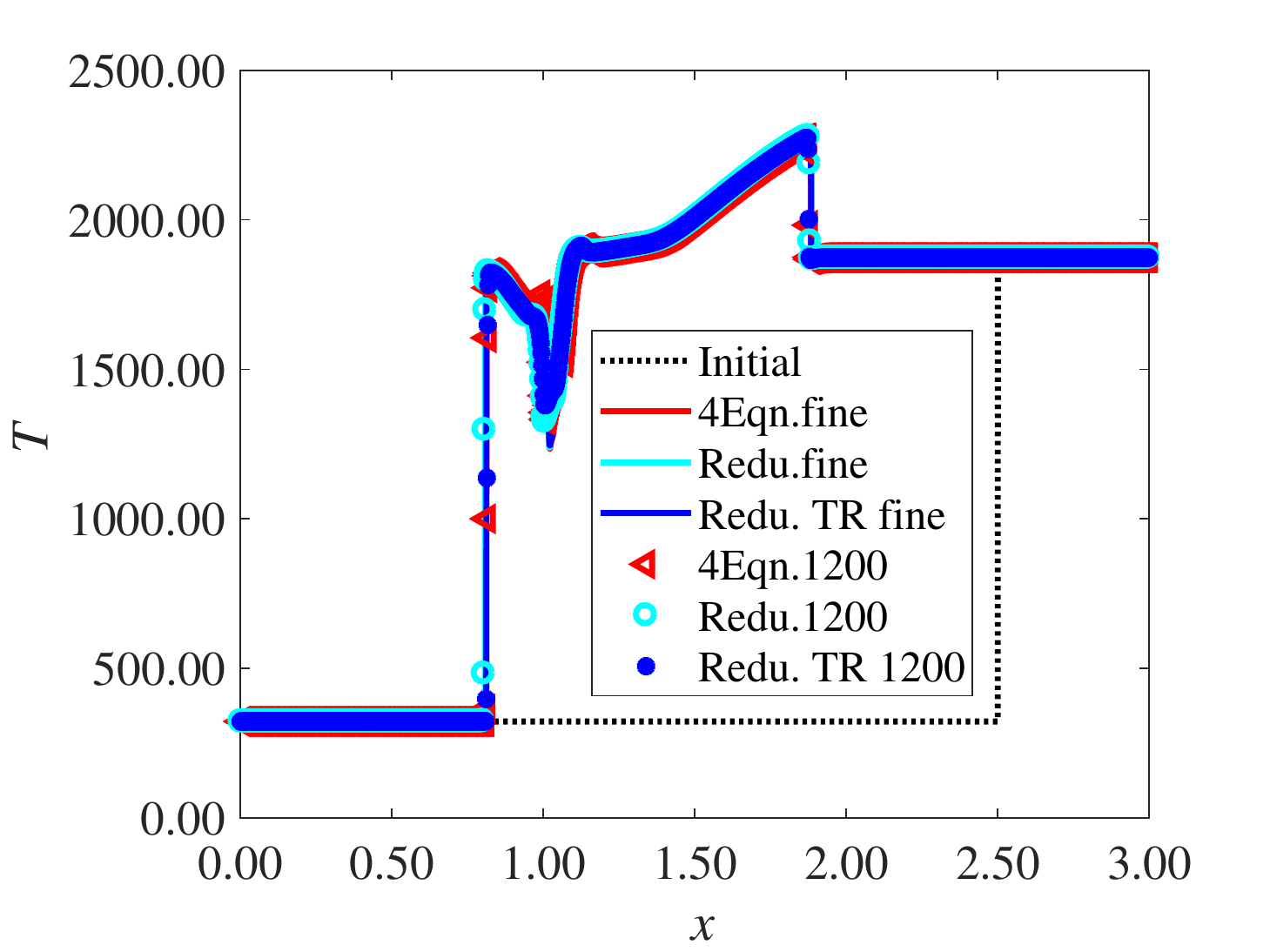}}
\subfloat[Temperature, locally enlarged]{\label{figGC:temp1}\includegraphics[width=0.5\textwidth]{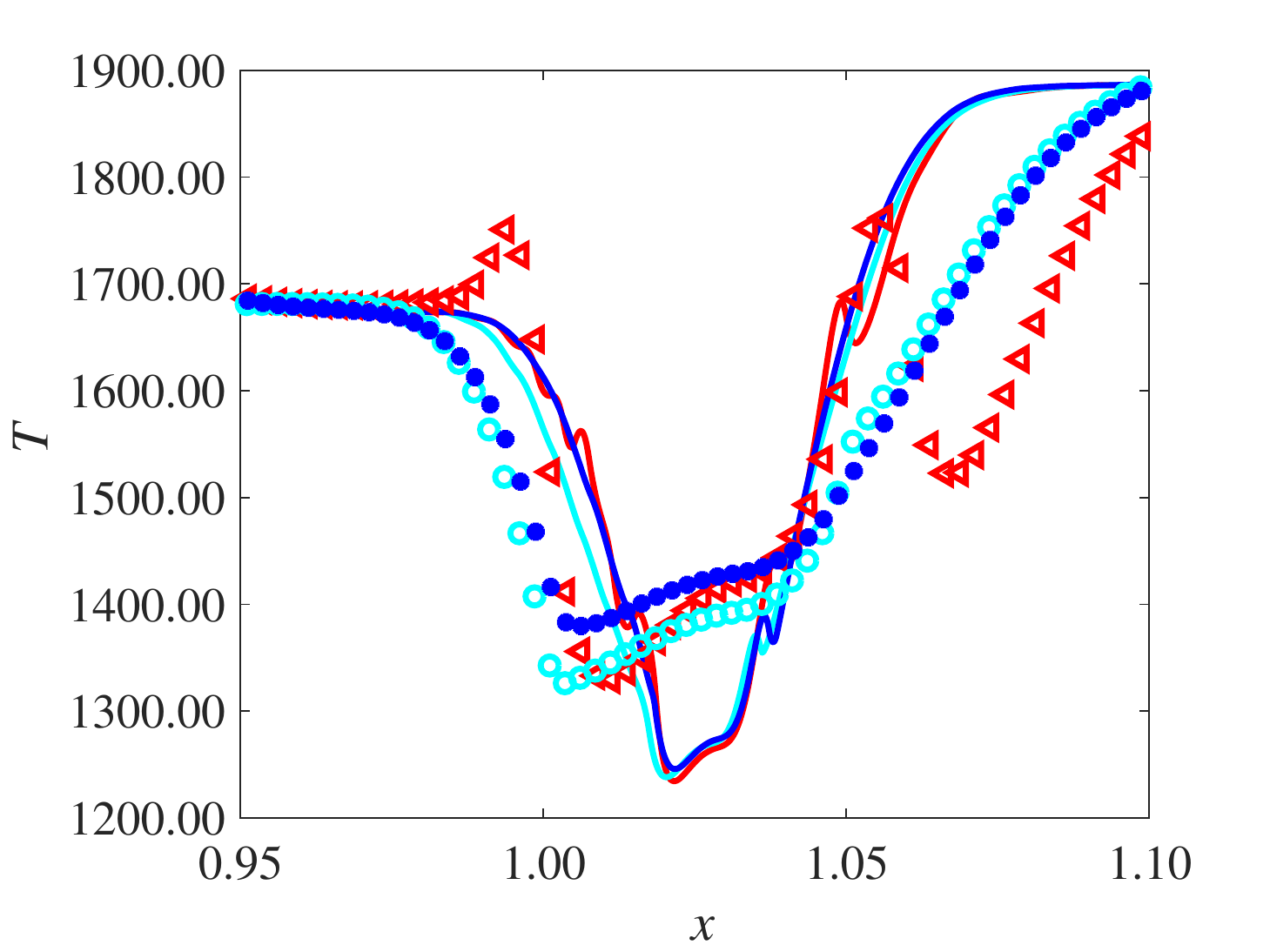}}\\
\subfloat[Mass fraction]{\label{figGC:Y}\includegraphics[width=0.5\textwidth]{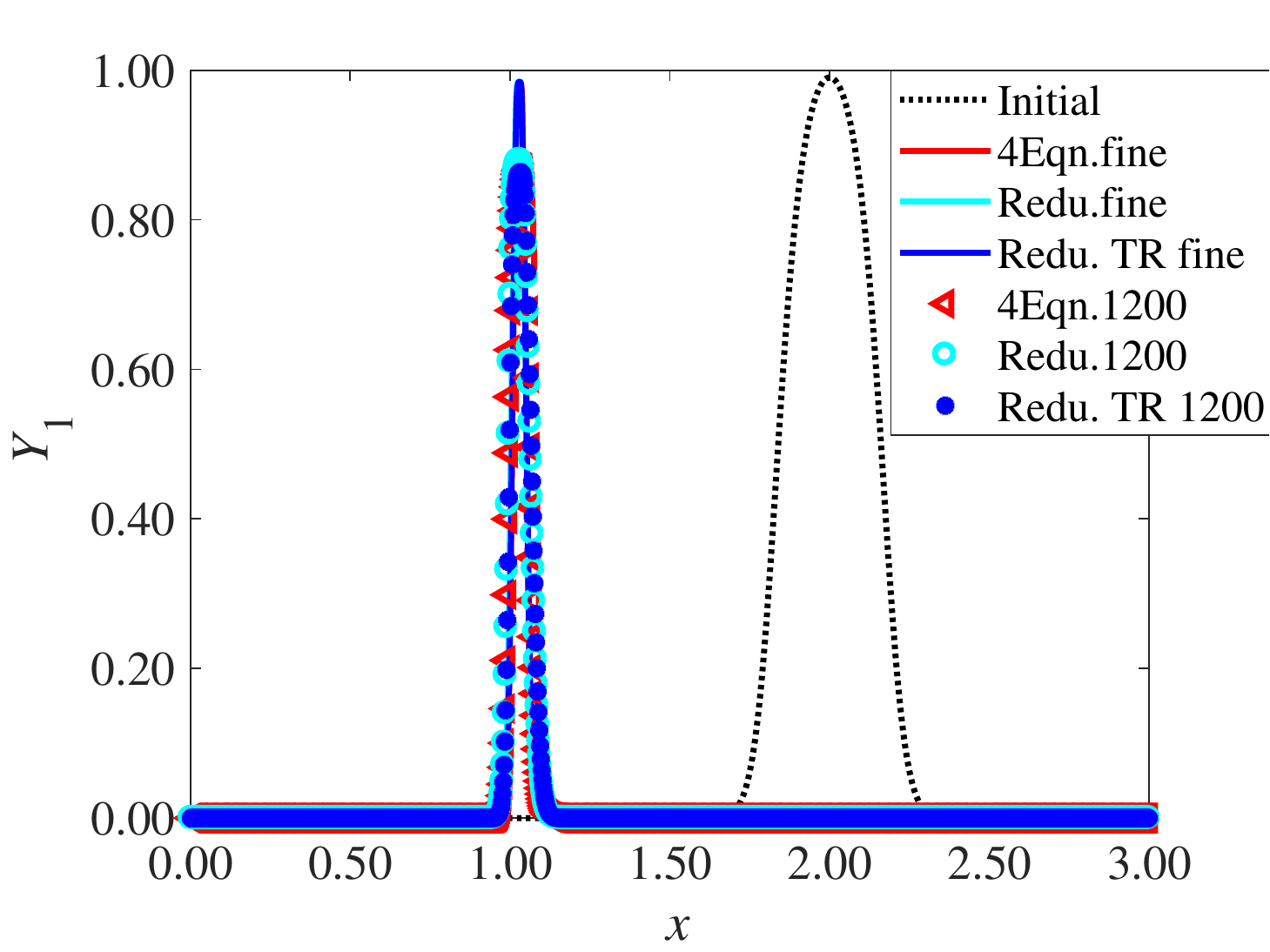}}
\subfloat[Mass fraction, locally enlarged]{\label{figGC:Y1}\includegraphics[width=0.5\textwidth]{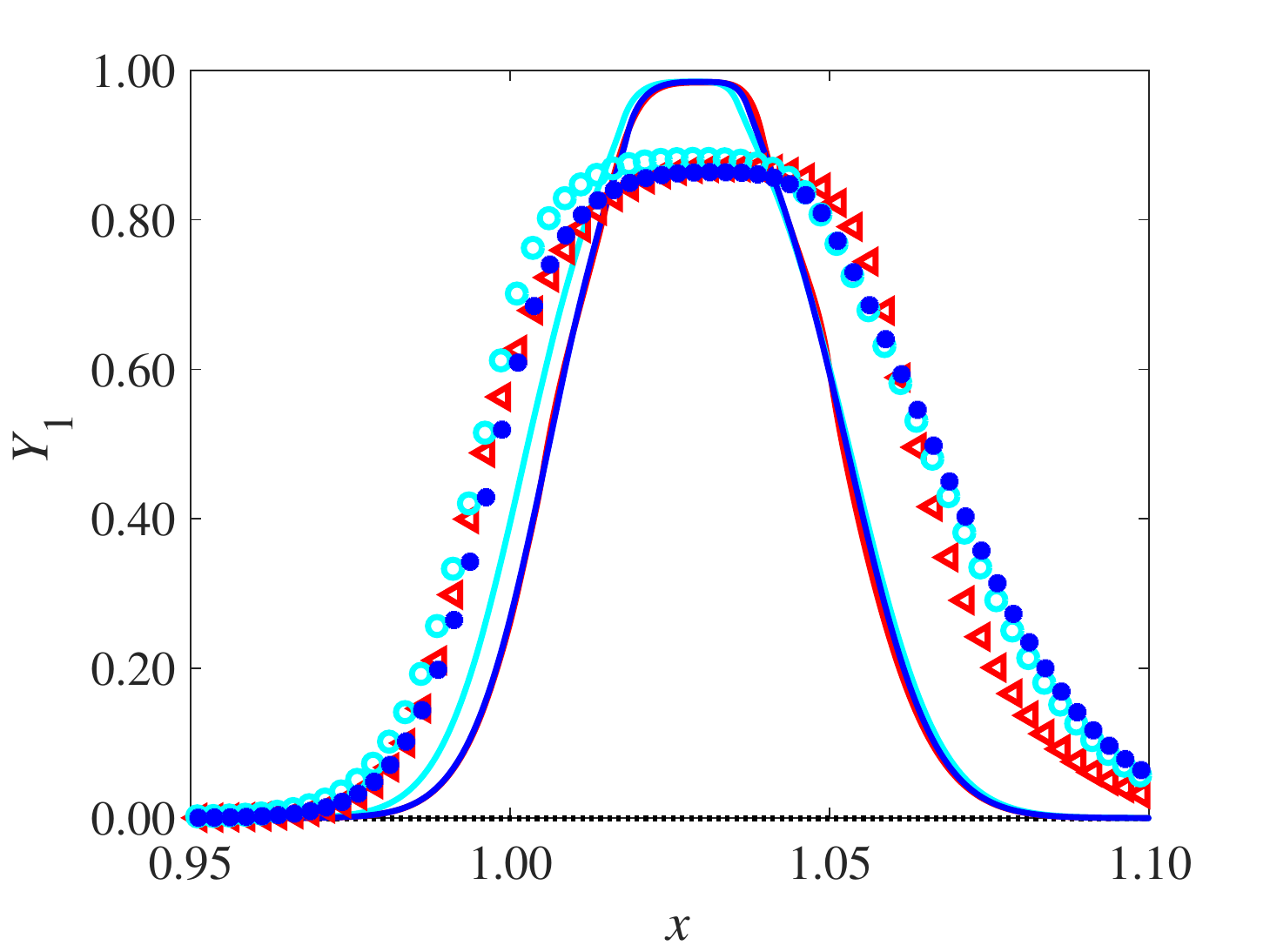}}
\caption{Numerical results for the shock passage through a smeared interface.}
\label{fig:figGC} 
\end{figure}

\subsection{The shock wave passage through a helium bubble}\label{subsec:HeB_temp}
In this section, we consider the interaction of a shock (Mach number 1.22) and a cylindrical helium bubble \cite{Haas1987,2018Capuano}. The computational domain is of size $22.25 \mathrm{~cm} \times 8.90 \mathrm{~cm}$. The bubble with the diameter 5cm is initially located at $(13.80 \mathrm{~cm}, 4.45 \mathrm{~cm})$. The initial data is given as:
\begin{equation}
    \left(\rho, u, p, \gamma, C_v\right)= \begin{cases}\left(1.66 \mathrm{~kg} \cdot \mathrm{m}^{-3},-114 \mathrm{~m} \cdot \mathrm{s}^{-1}, 159080.98 \mathrm{~Pa}, 1.4, 2430.35 \text{J}/(\text{Kg}\cdot\text{K}) \right) & \text { for } x>x_{s} \\ \left(1.2062 \mathrm{~kg} \cdot \mathrm{m}^{-3}, 0,101325 \mathrm{~Pa}, 1.4, 2430.35 \text{J}/(\text{Kg}\cdot\text{K}) \right) & \text { in air, for } x \leq x_{s} \\ \left(0.2204 \mathrm{~kg} \cdot \mathrm{m}^{-3}, 0,101325 \mathrm{~Pa}, 1.6451, 717.50 \text{J}/(\text{Kg}\cdot\text{K})\right) & \text { inside the helium bubble }\end{cases}
\end{equation}
where the $x_s = 16.80$cm is the initial position of the left-going shock wave. 

We compute this problem including viscosity, heat conduction and mass diffusion. The viscosity coefficient is determined with  Sutherland’s equation \cite{Sutherland}. The heat conduction coefficient is determined in the same way as in \cite{2018Capuano}. As for the mass diffusivity, since its dependence on temperature and pressure is $D \sim {T^{3/2}}/{p}$, we calculate it with
\begin{equation}
    D = D_0 \frac{T^{3/2}/p}{T_0^{3/2}/p_0},
\end{equation}
where $D_0$ is a reference value at $(p_0, T_0)$. We use $D_0 = 73.35\times 10^{-6}$m$^2$/s at $p=1$atm and $T=298$K \cite{wasik1969measurements}.

To verify the numerical results we compare the numerically obtained interface motion of the bubble along the horizontal centreline with those of \citep{2018Capuano} (\Cref{fig:interfPos}). Note that computations of \citep{2018Capuano} is performed without mass diffusivity. Since the shock-interface interaction time is very short, the impact of mass diffusivity on the interface motion is marginal. However, it does modify the small-scale flow structures (\Cref{fig:diff_vs_nodiff}).  From the 1D slice along the horizontal centreline, one can see that the diffusivity smears the density profile.

\begin{figure}[htbp]
\centering
\includegraphics[width=0.5\textwidth]{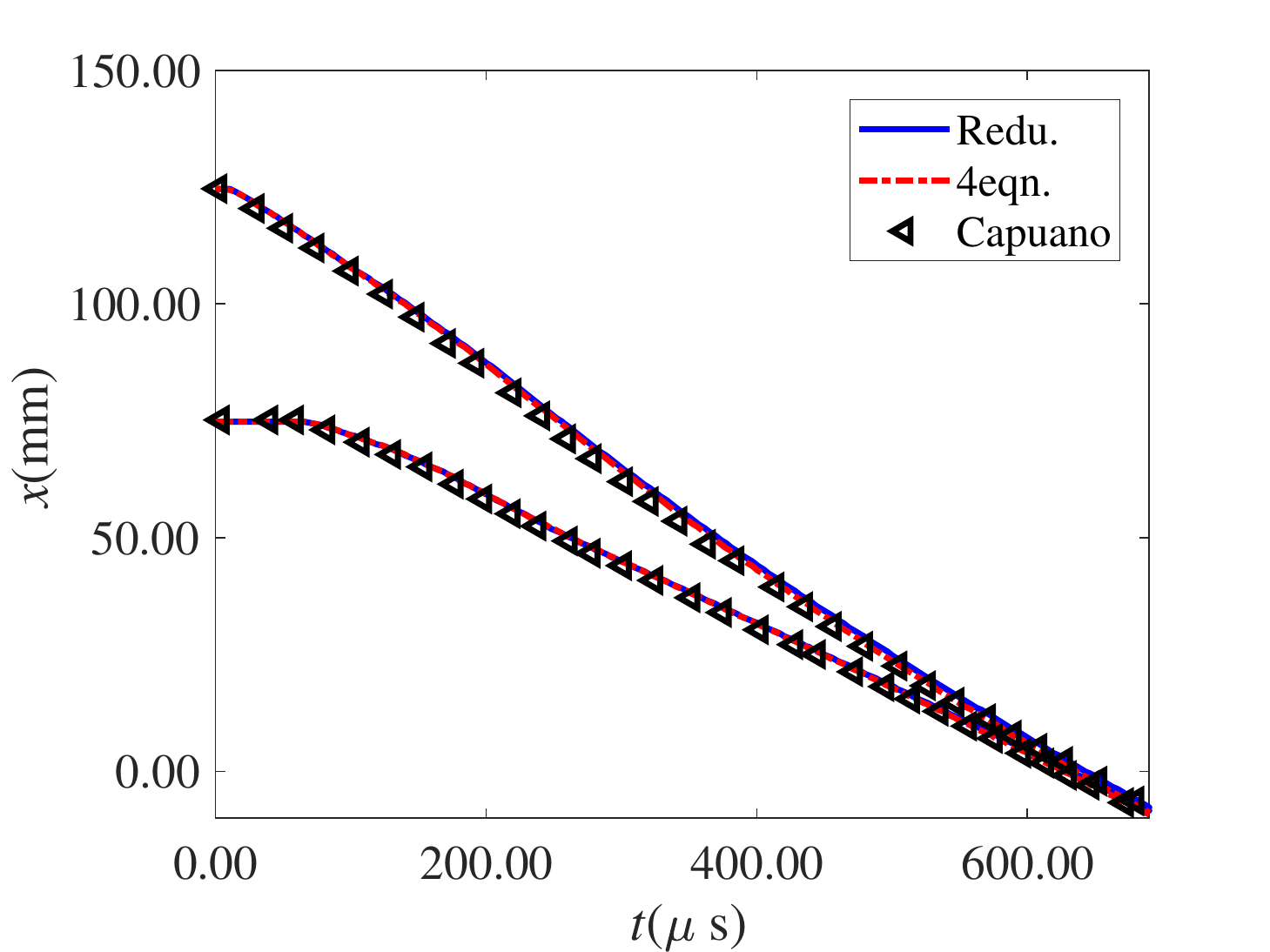}
\caption{The time evolution of the interface position along the horizontal centreline. The lines on the top(bottom) correspond to the positions of the right(left) interface.}
\label{fig:interfPos} 
\end{figure}

\begin{figure}[htbp]
\centering
\includegraphics[width=0.5\textwidth]{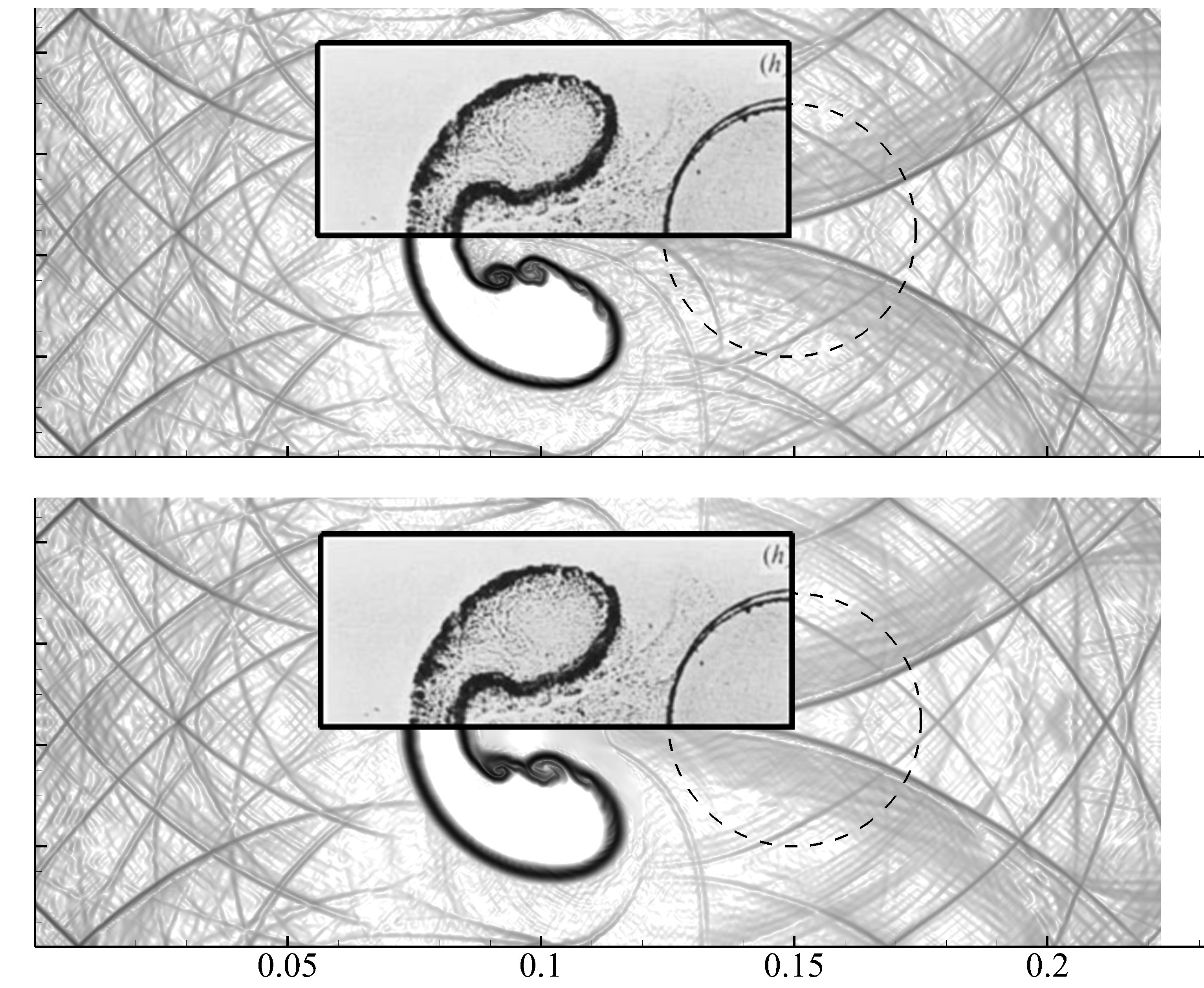}
\caption{The numerical Schlieren obtained with the four-equation model (top) and the five-equation model (bottom) compared with the experimental shadow graph (inside the rectangular area). The dashed circle is the initial position of the bubble.}
\label{fig:HeB_sh} 
\end{figure}

\begin{figure}[htbp]
\centering
\subfloat[The solutions for density in the cases of diffusivity being excluded (left) and included (right)]{\includegraphics[width=0.7\textwidth]{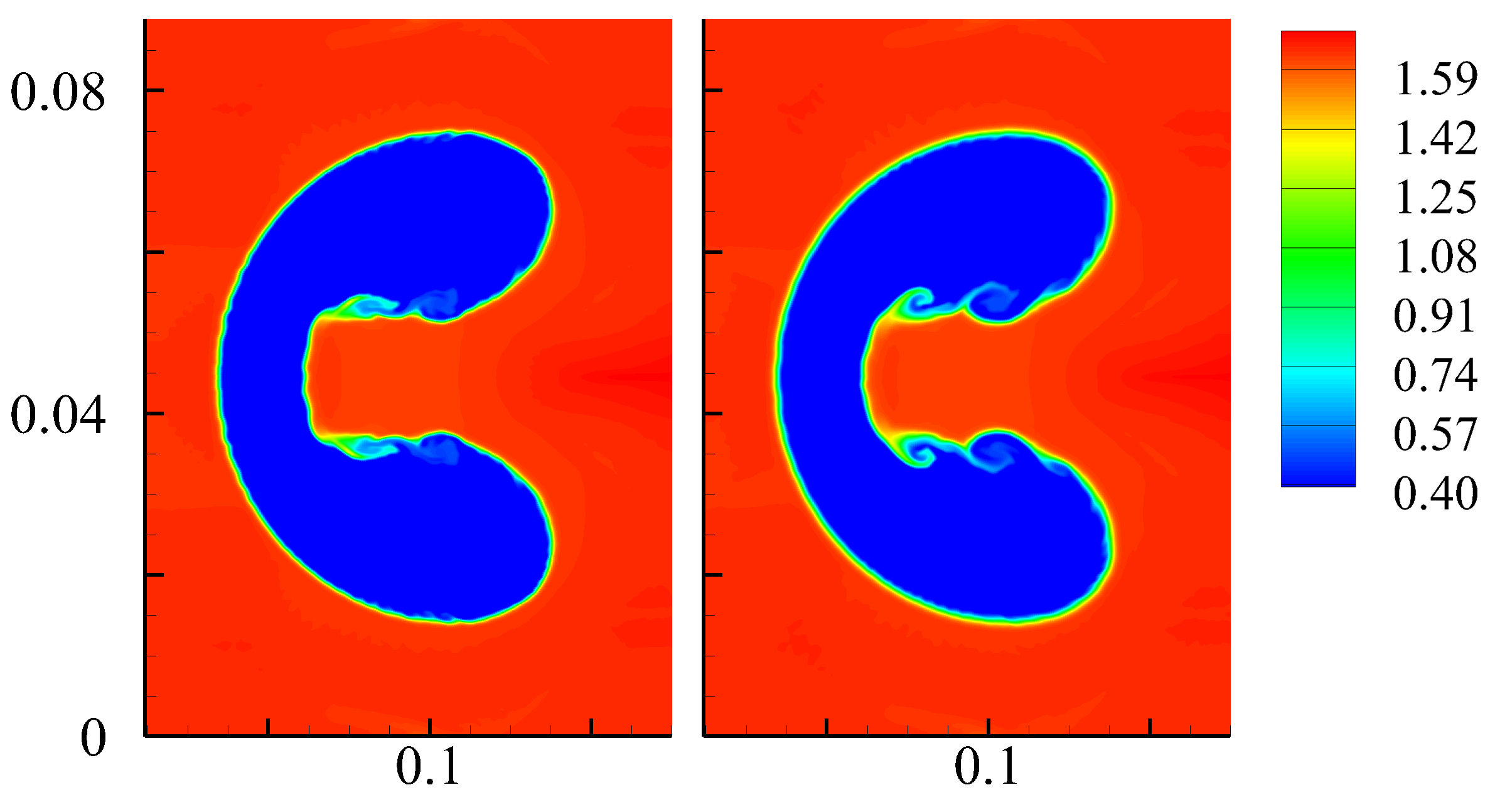}}\\
\subfloat[The density distribution along the horizontal centreline]{\includegraphics[width=0.5\textwidth]{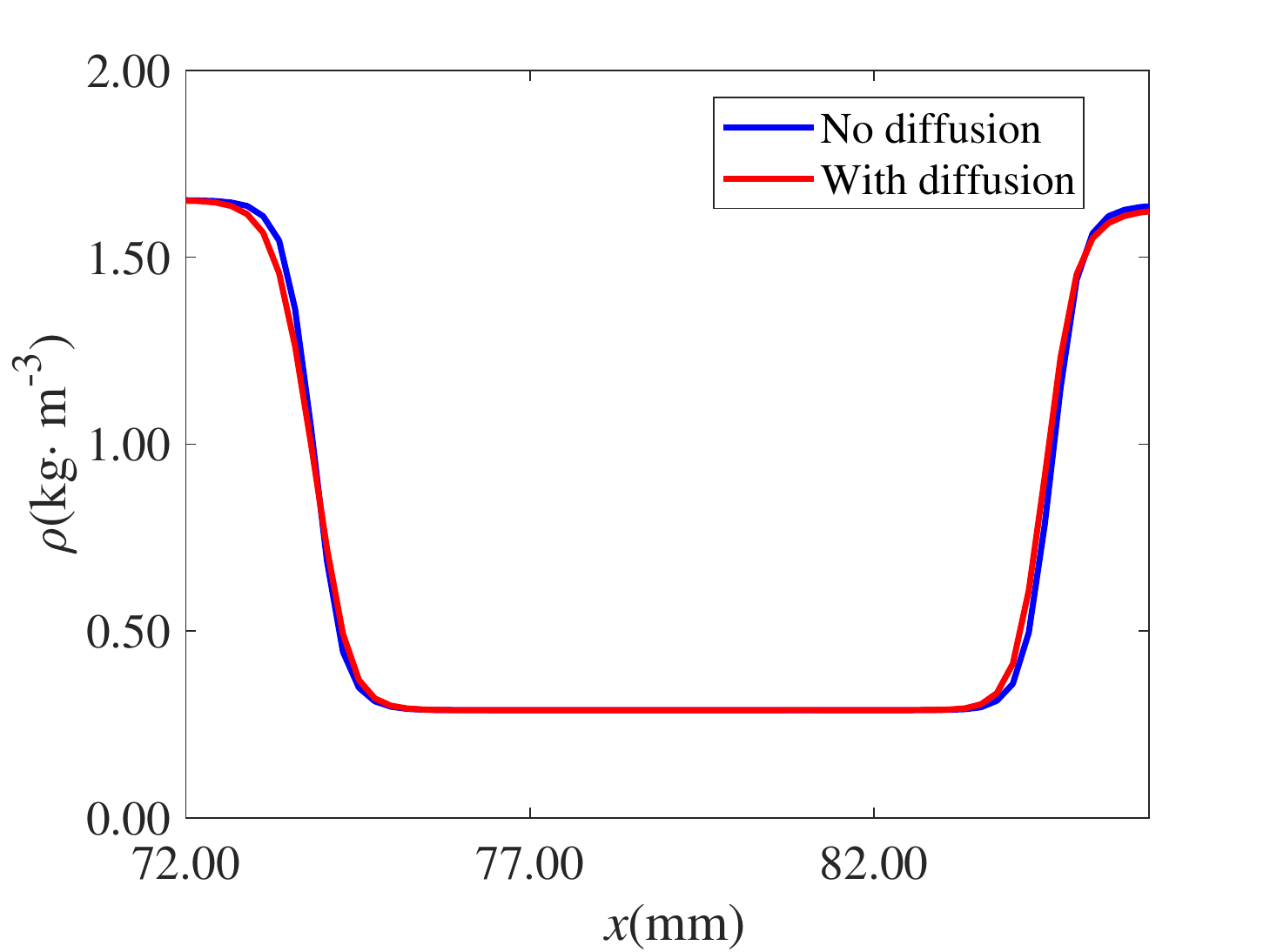}}
\caption{The experiment of RM instability on Omega.}
\label{fig:diff_vs_nodiff} 
\end{figure}

The solutions for the temperature obtained with different models  in the neighbourhood of  the bubble are compared in  \Cref{fig:HeB_temp}. Serious non-physical oscillations arise in across the interface in the solutions of the four-equation model, which is more clear in the 1D distribution along the horizontal centreline in \Cref{HeB_temp:1D}. When temperature diffusivity is significant enough, this non-physical error may have a major impact on the convergence of the model.

\begin{figure}[htbp]
\centering
\subfloat[Temperature distribution (at the time moment $t=427\mu$s timing from the beginning of the shock-interface interaction) obtained with the four-equation model (left) and the reduced five-equation model (right).]{\label{HeB_temp:2D}\includegraphics[width=0.7\textwidth]{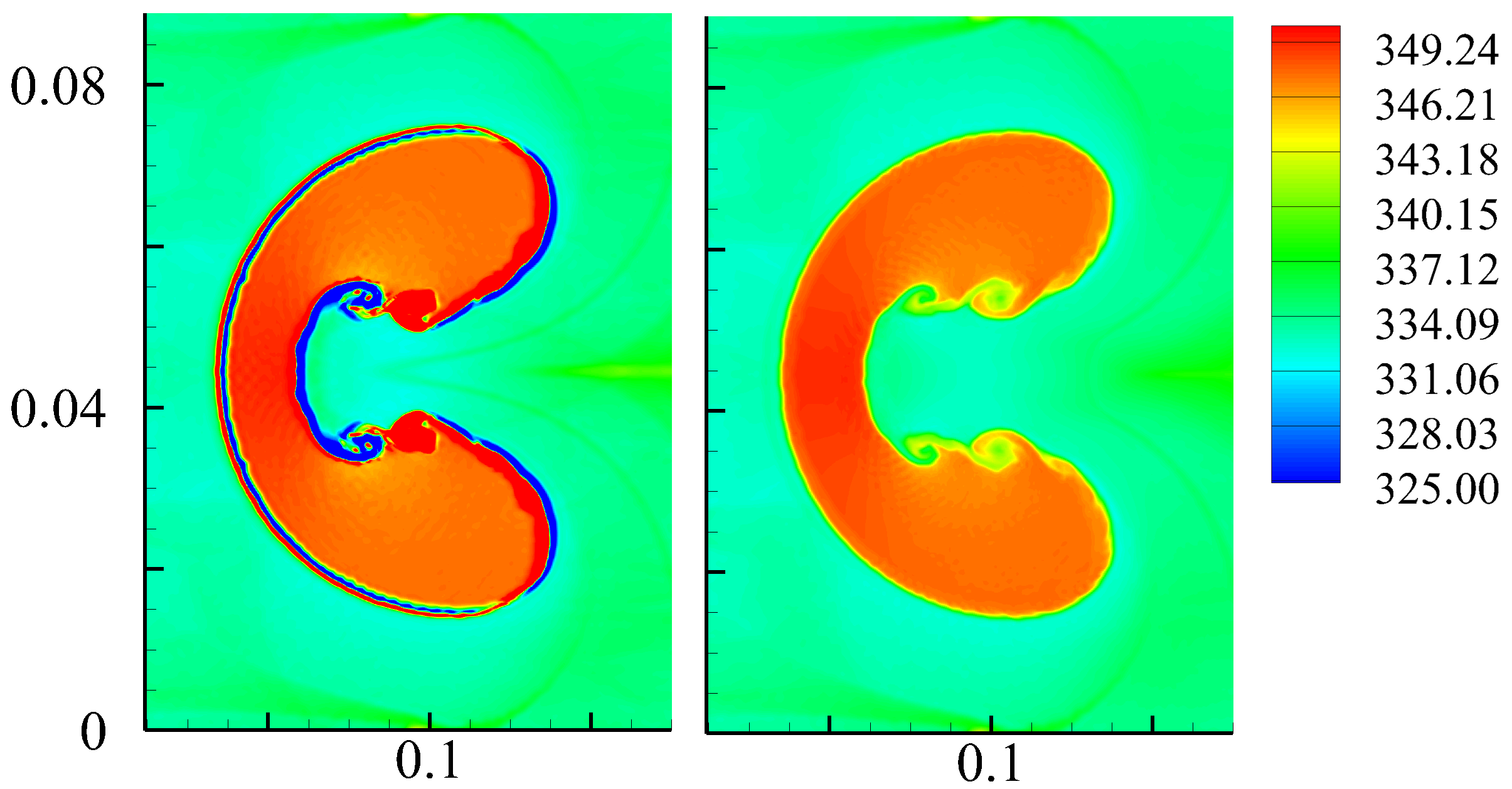}}\\
\subfloat[The temperature distribution along the centreline of the bubble]{\label{HeB_temp:1D}\includegraphics[width=0.5\textwidth]{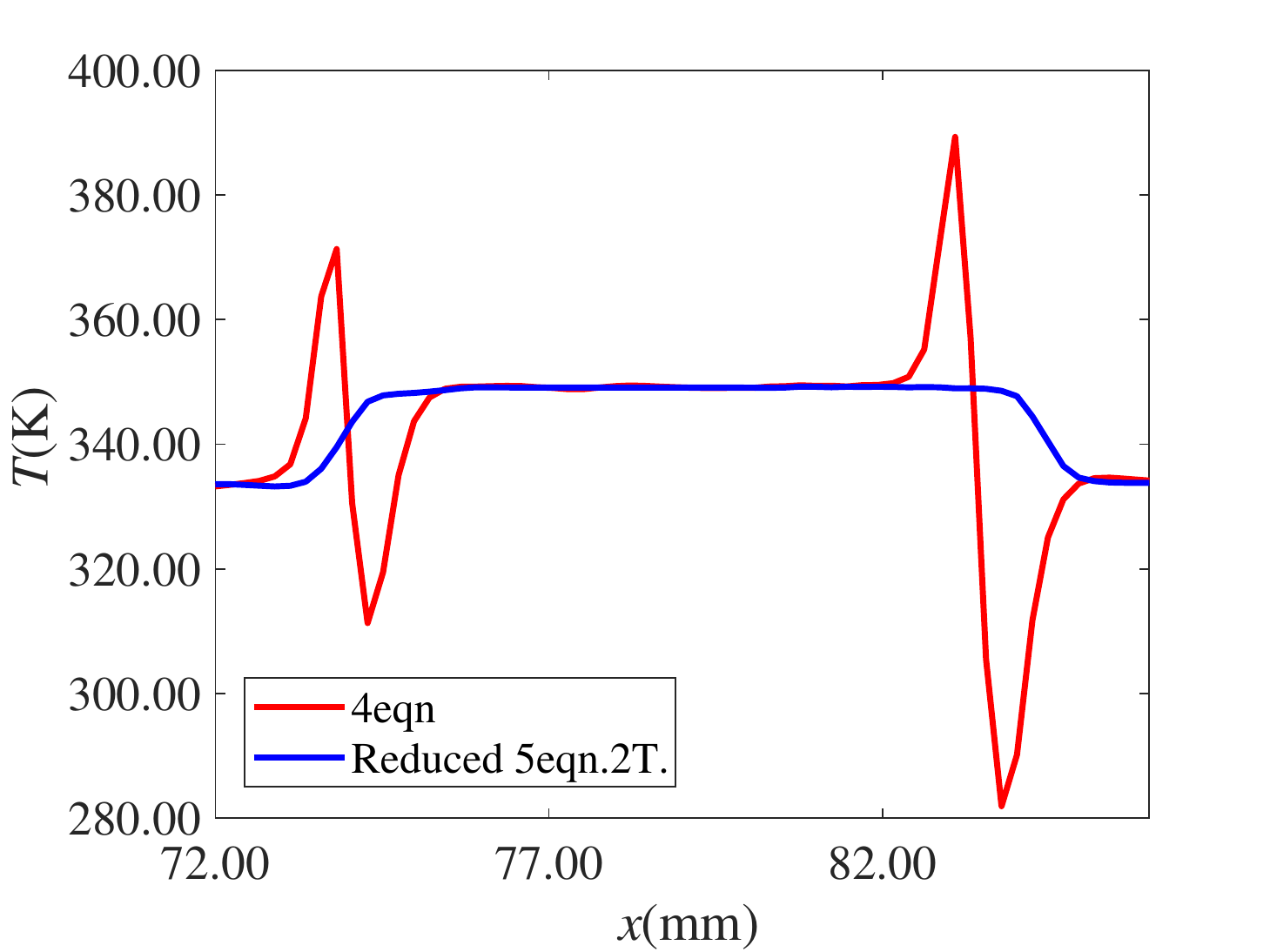}}
\caption{Numerical results for temperature.}
\label{fig:HeB_temp} 
\end{figure}

\subsection{The laser-driven RM instability problem}\label{subsec:ABLRM}
In this section we consider the laser-driven RM instability problem conducted on OMEGA \cite{Robey2003The,miles2004numerical}. The schematic for the experiment is demonstrated in \Cref{fig:ABLsetup}. The multi-material target is assembled into a Beryllium shock tube of diameter 800$\mu$m. The target is made up of two sections: the pusher section of length 150$\mu$m on the right and the payload section of length 19mm on the left.  A strong shock is generated by laser ablation of the pusher section that consists of the polyimide (C$_{22}$H$_{10}$N$_{2}$O$_{4}$) and the polystyrene (C$_{500}$H$_{457}$Br$_{43}$). This section is modeled as a homogeneous material with density 1.41$\text{g}/\text{cm}^3$ and $\gamma = 5/3$. The remainder of the target is carbon foam payload (C-foam, $\rho = 0.1\text{g}/\text{cm}^3, \; \gamma = 1.4$) . The interface between two sections is initially perturbed as a cosine function with wavelength 50$\mu$m and amplitude 2$\mu$m. The laser has a wavelength of 0.351$\mu$m and average intensity $6\times10^{14}$W/cm$^2$. The shock in the central area has good planarity, and thus is approximated as a planar one.

Note that there are many complicated experimental uncertainties that are difficult to account for in numerical simulations, for example, the pre-heating state of the target and the laser energy loss. Moreover, with the polytropic equation of state, the true state of the materials are described with limited accuracy.  Moreover, the experiment diagnosis also introducesss some error. Due to these uncertainties, numerical simulations can hardly reproduce the experimental conditions. We set  the initial temperature to be 290K based on a trial-and-error approach.

Our simulation focuses on one period of the perturbation  with periodical boundary conditions being imposed on sides perpendicular to the incident shock. The present simulation includes a complete physical processes: laser energy deposition, heat conduction, viscosity and mass diffusion. The laser energy deposits in a 20$\mu$m area to the right of the critical density $\rho_{crt}$ of the ablator. According to  the inverse bremsstrahlung absorption theory the critical density is $\rho_{crt} = 2.78\times 10^{-2}$g/cm$^{3}$. The heat conduction coefficient of the plasma is calculated with the Spitzer-Harm model \cite{Spitzer1953}. The plasma viscosity is modeled with Clerouin's model \cite{clerouin1998viscosity}. As for the mass diffusivity, we use the estimates in \cite{robey2004effects}, prior to the interaction of shock and the interface, the materials are in solid states and the mass diffusivity is negligible. After the shock arrival (at $t \approx 2$ns), the Schmidt number ($Sc = \nu / D$, $\nu$ is the average kinetic viscosity) is almost constant 1.  In this estimation, the mass diffusivity is determined with the model of Paquette \cite{paquette1986diffusion}.

Computations are performed on a grid of $2880 \times 60$ cells with the proposed reduced model and the conservative four-equation model. The numerical results for density, temperature and mass fraction are displayed in \Cref{fig:figABL_comp}. It can be seen that the solutions of both models have similar flow structures. The shock wave and the interface move slightly faster in the solutions of the reduced model than that in the solutions of the four-equation model.

\begin{figure}[htbp]
\centering
\includegraphics[width=0.4\textwidth]{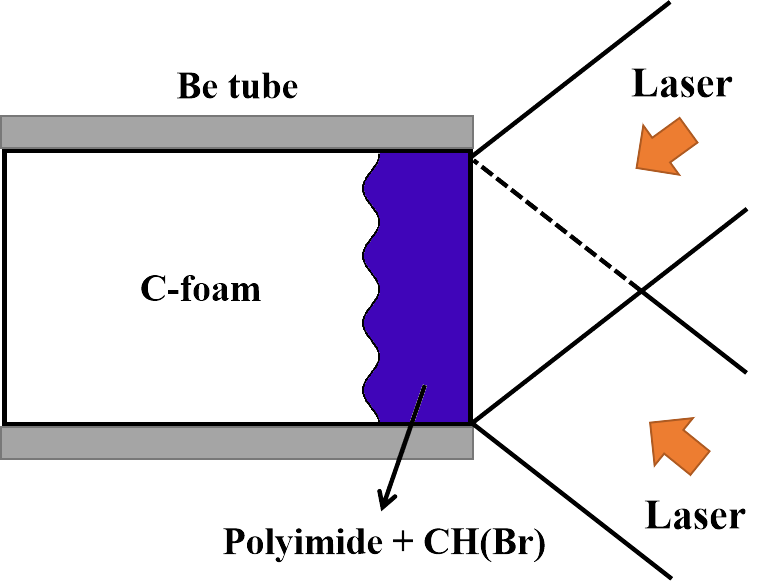}
\caption{The schematic of the laser-driven  RM instability experiment on Omega.}
\label{fig:ABLsetup} 
\end{figure}

\begin{figure}[htbp]
\centering
\subfloat[Density]{\label{figABL_comp:dens}\includegraphics[width=0.98\textwidth]{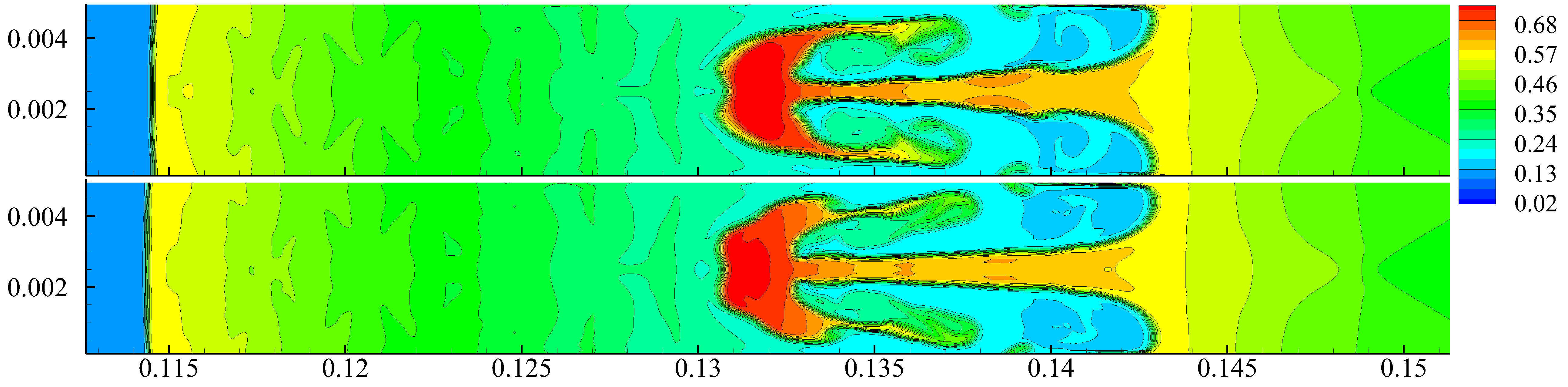}}\\
\subfloat[Temperature]{\label{figABL_comp:temp}\includegraphics[width=0.98\textwidth]{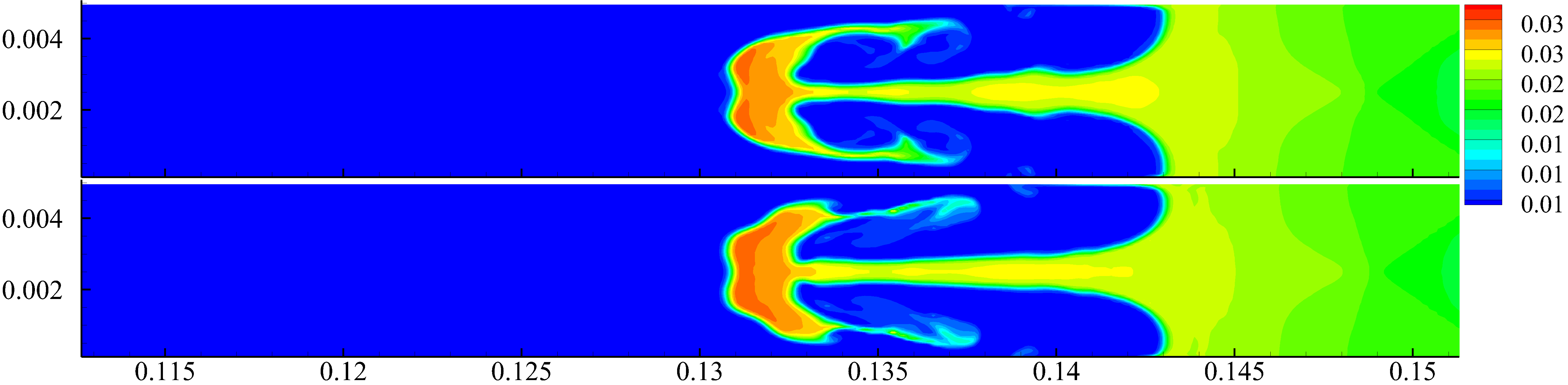}}\\
\subfloat[Mass fraction]{\label{figABL_comp:y}\includegraphics[width=0.98\textwidth]{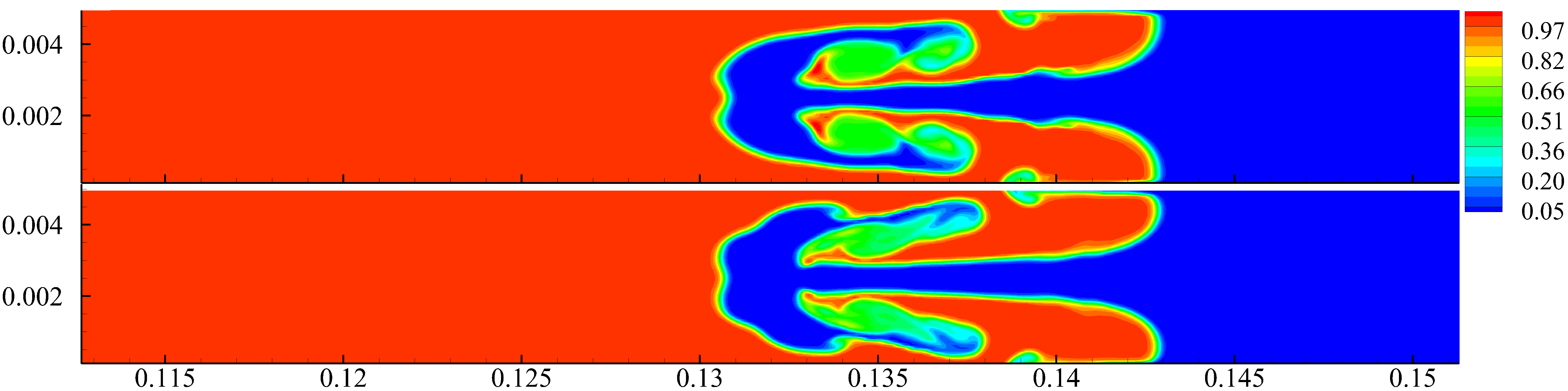}}
\caption{Numerical results for the density, the temperature and the mass fraction  when $t = 12$ns for the laser-driven RM instability problem. Top figures are results of the four-equation model, bottom ones are that of the reduced model. The dimension for length, density and temperature are cm, g/cm$^3$ and MK, respectively. }
\label{fig:figABL_comp} 
\end{figure}

In \Cref{fig:ABL_STA} we compare the numerically obtained interface evolution parameters with the experimental ones. \Cref{ABL_STA:interfPos} demonstrates the evolution of the leftmost interface position (i.e., the distance from its initial position) with time. Good agreement with the experimental results are observed. \Cref{ABL_STA:amp} shows the time evolution of the half peak-to-valley amplitude. Both models give results that lie within the measurement range.

\begin{figure}[htbp]
\centering
\subfloat[Interface position]{\label{ABL_STA:interfPos}\includegraphics[width=0.5\textwidth]{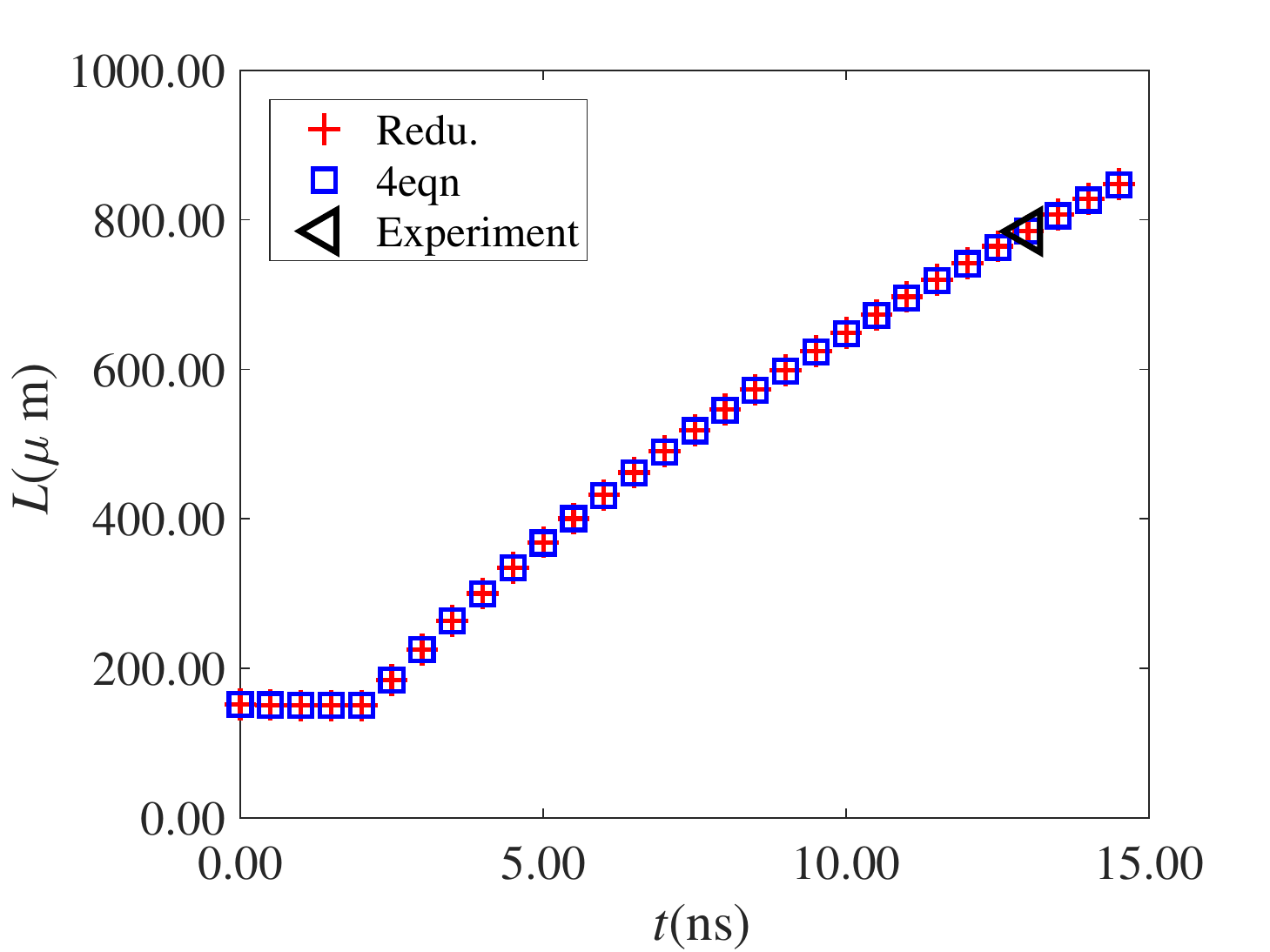}}
\subfloat[Amplitude]{\label{ABL_STA:amp}\includegraphics[width=0.5\textwidth]{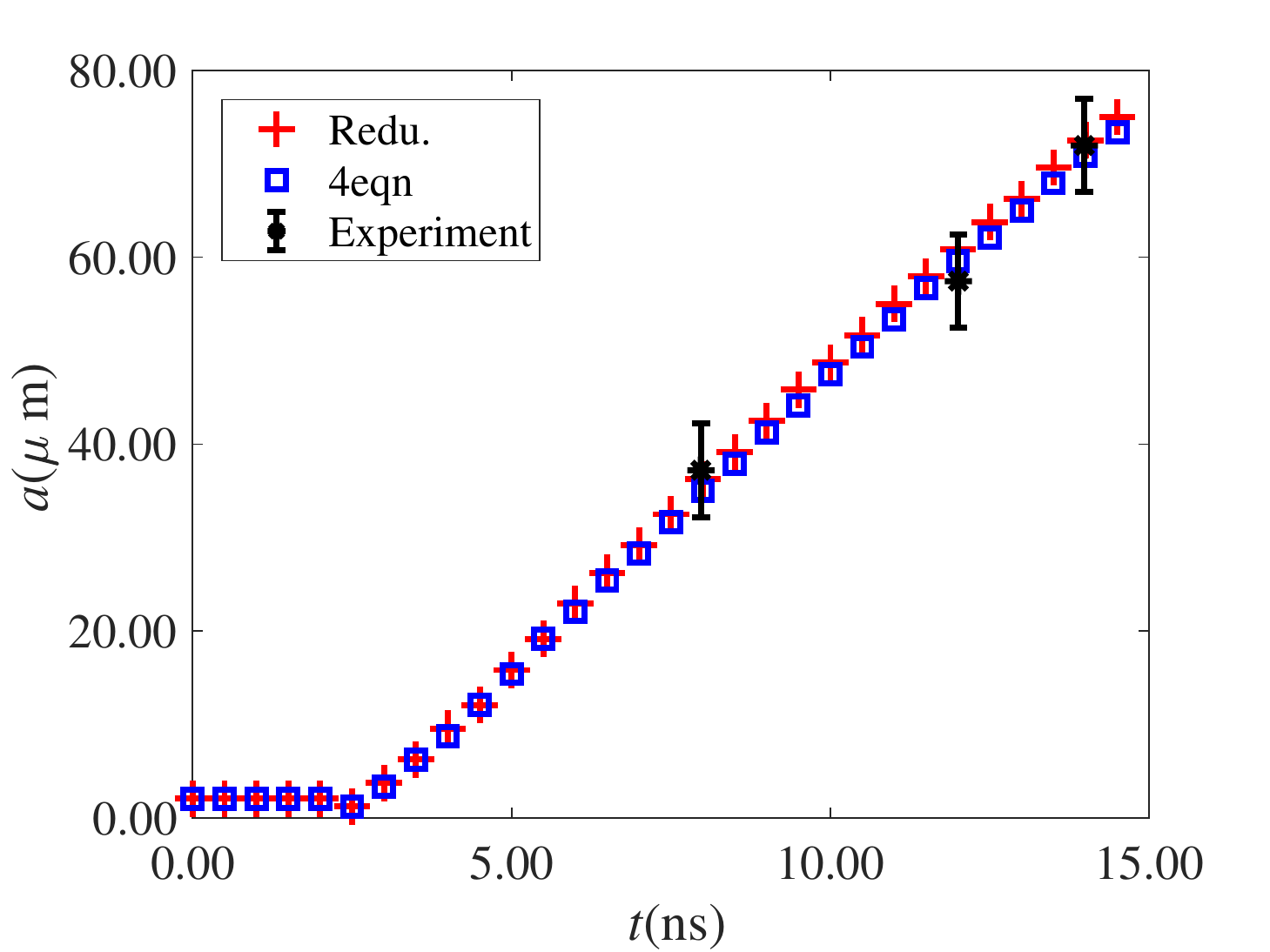}}
\caption{The time evolution of the interface position and amplitude of the laser-driven RM instability problem.}
\label{fig:ABL_STA} 
\end{figure}

 We define the Reynolds number $Re = u L / \nu$, where $u$ is the characteristic velocity $\frac{1}{2}\rho u^2 = E$, $E$ is the the total deposited laser energy, $\rho$ is the average mixture density, the characteristic length $L$ is taken to be the wavelength of the initial perturbation and the $\nu$ the initial maximum mixture kinetic viscosity. To investigate the impact of the diffusivity, we increase $Re$ through the viscosity. 
The numerical solutions for the mass fraction in the case of different diffusivities are compared in \Cref{fig:ABLRM_DIFF_NO_DIFF}. We can see that the transport process tend to wipe out the small flow structures. The last two figures compare the numerical results with and without the mass diffusion, whose effect in smearing the mass fraction is evident.

\begin{figure}[htbp]
\centering
\includegraphics[width=0.6\textwidth]{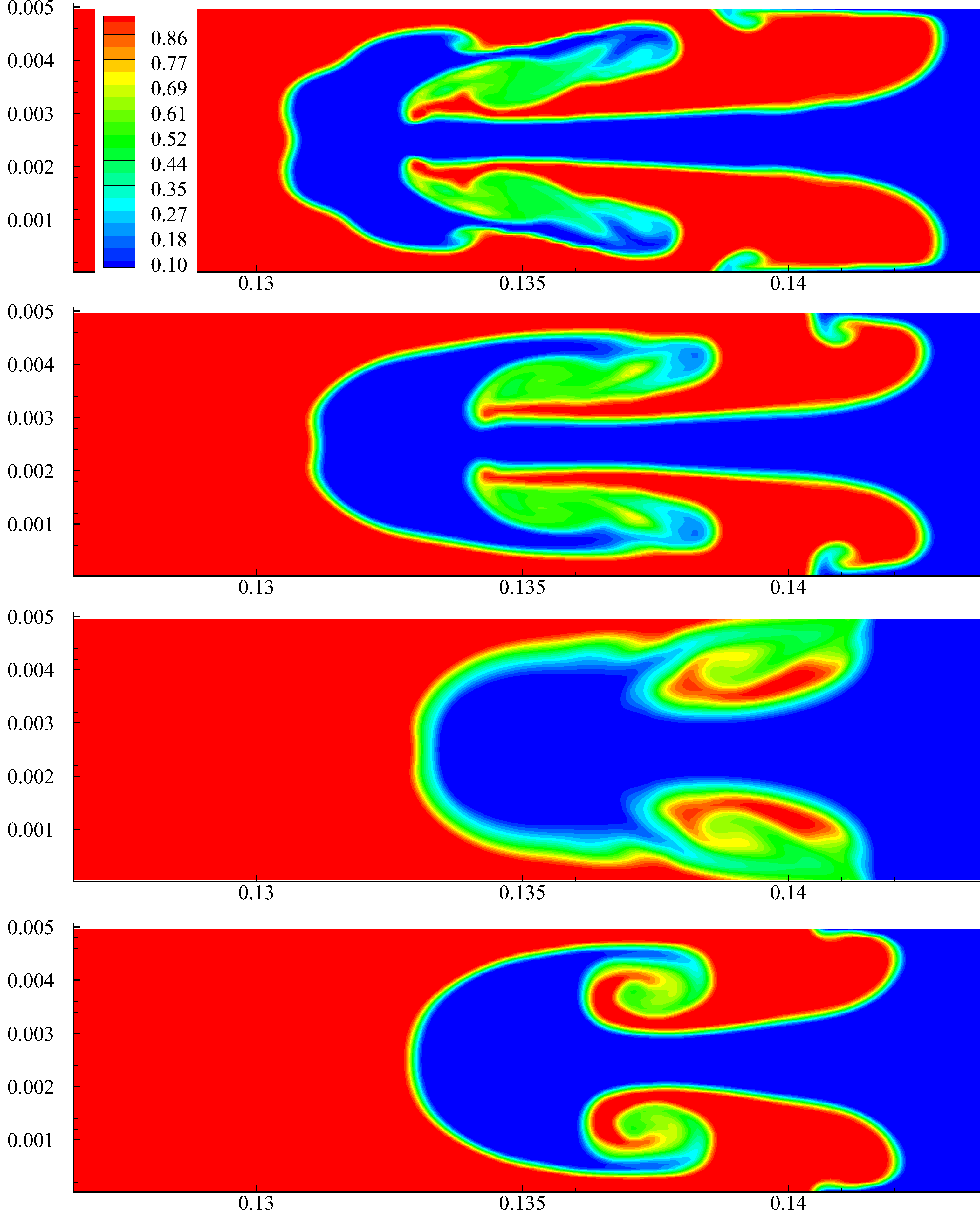}
\caption{numerical solutions for the mass fraction of the C-foam in the case of different diffusivities. The Reynolds number for the experiment condition is used as the reference $Re_{exp}$. From top to bottom: (1) $Re = Re_{exp}, \; Sc = 1$, (2) $Re =  Re_{exp}/10, \; Sc = 1$, (3) $Re = Re_{exp}/50,  \; Sc = 1$, (4) $Re = Re_{exp}, \; 1/Sc \to 0$.  }
\label{fig:ABLRM_DIFF_NO_DIFF} 
\end{figure}

\section{Conclusion}
\label{sec:conclusion}
In the present paper we have presented a diffuse-interface model for compressible multicomponent flows with interphase heat transfer, external energy source, and diffusions (including viscous, heat conduction, mass diffusion, enthalpy diffusion processes). The model is reduced from the Baer-Nunziato model in the limit of instantaneous mechanical relaxations. Difference between time scales of velocity, pressure and temperature relaxations has been accounted for. The reduction procedure results in a temperature-disequilibrium, velocity-disequilibrium, and pressure-equilibrium five-equation model.  The proposed model  if free of the spurious oscillation problem in the vicinity of the interface and respects the laws of thermodynamics. Numerical methods for its solution have been proposed on the basis of the fractional step method. The model is split into five parts including the hydrodynamic part, the viscous part, the temperature relaxation part, the heat conduction part and the mass diffusion part. The hyperbolic equations involved are solved with the Godunov finite volume method, and the parabolic ones with the locally iterative method based on Chebyshev parameters. The developed model and numerical methods have been used for solving several multicomponent problems and verified against analytical and experimental results. Moreover, we have applied our model to simulate the laser ablation process of a multicomponent target, where the Richtmyer-Meshkov instability can be evidently observed. Comparison with experimental results demonstrates that our model captures physical phenomenon of this process.

%\bibliography{ref}
%\bibliographystyle{plain}

%\normalem
\bibliographystyle{unsrt} 
\bibliography{ref}

%\begin{thebibliography}{}
%%% Text of bibliographic item
%\bibliography{ref}
%\bibliographystyle{plain}
%\end{thebibliography}

\end{document}